\documentclass[12pt]{amsart}

\usepackage{amsfonts}
\usepackage{amsmath}
\usepackage{amssymb}
\usepackage{amsthm}
\usepackage{color}
\usepackage[mathscr]{euscript}
\usepackage{verbatim}
\usepackage{graphicx}
\usepackage{multicol}
\usepackage{enumitem}

\newtheorem{theorem}{Theorem}[section]
\newtheorem{lemma}[theorem]{Lemma}
\newtheorem{corollary}[theorem]{Corollary}

\theoremstyle{definition}

\theoremstyle{remark}
\newtheorem{remark}[theorem]{Remark}

\numberwithin{equation}{section}

\newenvironment{subproof}[1][\proofname]{
  
  \begin{proof}[#1]
}{
  \end{proof}
}

\newcommand{\LFE}{\mathscr{L}(\mathscr{F}(E))}
\newcommand{\F}[1]{\mathscr{F}(#1)}
\newcommand{\tens}[2]{#1^{\otimes #2}}

\newcommand{\phiinf}[1]{\varphi_{\infty}(#1)}

\newcommand{\phitens}[1]{\varphi_{\infty}(#1) \otimes I_{H}}
\newcommand{\wtens}[1]{W_{#1} \otimes I_{H}}
\newcommand{\wstartens}[1]{W_{#1}^* \otimes I_{H}}
\newcommand{\mf}[1]{\mathfrak{#1}}
\newcommand{\ms}[1]{\mathscr{#1}}
\newcommand{\mb}[1]{\mathbb{#1}}
\newcommand{\mc}[1]{\mathcal{#1}}

\newcommand{\mcZ}[1]{(Z^{(#1)})^{-1}}

\newcommand{\inc}[2]{(V\!\genfrac{}{}{0pt}{1}{#2}{#1})}
\newcommand{\proj}[2]{(P\!\genfrac{}{}{0pt}{1}{#2}{#1})}

\begin{document}
\title{Interpolation and Commutant Lifting with Weights}
\author{Jennifer R Good}
\address{Department of Mathematics, University of Wisconsin - Platteville, Platteville, WI 53818}
\email{goodje@uwplatt.edu}

\begin{abstract}

Our two principle goals are generalizations of the commutant lifting theorem and the Nevanlinna-Pick interpolation theorem to the context of Hardy algebras built from $W^*$-cor\-res\-pon\-den\-ces endowed with a sequence of weights. These theorems generalize theorems of Muhly and Solel from 1998 and 2004, respectively, which were proved in settings without weights.  Of special interest is the fact that commutant lifting in our setting is a consequence of Parrott's Lemma; it is inspired by work of Arias.

\end{abstract}

\maketitle

\section{Introduction} \label{IntroChap}

This paper concerns weighted interpolation problems in the theory of operator algebras built from  $W^*$-correspondences.
The classic origins of this study are summarized in the phrase \emph{weighted interpolation}.
The original interpolation problem, solved separately by Nevanlinna and Pick in the early twentieth century, sought conditions under which, given $k \in \mb{N}$ and two collections, $\{\omega_i\}_{i = 1}^k \subseteq \mb{D}$ and $\{\lambda_i\}_{i = 1}^k \subseteq \mb{C}$, there exists a function $\phi \in H^\infty(\mb{D})$ of norm at most $1$ that interpolates this data; that is, 
\begin{equation}
\label{e1}
\phi(\omega_i) = \lambda_i, \quad 1 \leq i \leq k.
\end{equation}
Nevanlinna and Pick showed that such a $\phi$ exists if and only if
\begin{equation}
\label{e2}
\left[ \ms{K}(\omega_i, \omega_j) (1-\lambda_i \overline{\lambda_j}) \right]_{i,j=1}^k
\end{equation}
is a positive semidefinite $k \times k$ matrix where $\ms{K}$ is the Szeg\"{o} kernel, $\ms{K}(w, z) = \frac{1}{1 - w \bar{z}}$ for $w,z \in \mb{D}$, reproducing kernel to the Hardy space $H^2(\mb{D})$   \cite{Nevanlinna1919} \cite{Pick1916}.
More generally, a \emph{weighted} Hardy space, as described in  \cite{Shields_1974}, is a reproducing kernel Hilbert space with kernel $\ms{K}: \mb{D}_r \times \mb{D}_r \to \mb{C}$, defined by $\ms{K} (w,z) = \sum_{i = 0}^\infty \frac{\left(w \overline{z} \right)^i}{\beta_i^2}$ 
for a sequence of strictly positive numbers $\beta = \left\{\beta_i \right\}_{i = 0}^\infty$.
Our primary goal is a weighted Nevanlinna-Pick type interpolation theorem that applies to the so-called weighted Hardy algebra, a non-self-adjoint operator algebra recently studied in \cite{Muhly2016}.

While we will expand on the details in a later section, in brief the setting of \cite{Muhly2016}, inspired by Popescu's work in \cite{Popescu2010}, begins with a $W^*$-algebra  $M$, a $W^*$-correspondence over $M$ denoted $E$, and a sequence of operator-valued weights $Z$ that has been constructed from a so-called admissible sequence $X$.
Our algebra of focus is the weighted Hardy algebra $H^\infty(E,Z)$, the algebra of adjointable operators on the Fock space $\ms{F}(E)$ that is generated by the  weighted creation operators associated with $Z$ and the left action maps of $M$ on  $\ms{F}(E)$.
To gain insight into $H^\infty(E,Z)$ we may view it as an algebra of \emph{functions} on its space of representations, as follows.  
Having fixed $\sigma$, a representation of $M$ on a Hilbert space $H$, the associated representations of $H^\infty(E,Z)$ are in a  sense determined by $\mb{D}(X, \sigma)$, a certain collection  of operators in $\ms{B}(E \otimes_\sigma H, H)$.
Every point $\mf{z} \in \mb{D}(X, \sigma)$ gives rise to a  representation $(\sigma \times \mf{z})$ of $H^\infty(E,Z)$ on $H$, and we may consider an operator $Y \in H^\infty(E, Z)$ as a $\ms{B}(H)$-valued function $\widehat{Y}$ on $\mb{D}(X, \sigma)$ defined at $\mf{z} \in \mb{D}(X, \sigma)$ by  $\widehat{Y}(\mf{z}) = (\sigma \times \mf{z})(Y)$.

In this setting, let us preview the main theorem.  
We form a weighted $W^*$-version of the original Szeg\"{o} kernel,  a map $\ms{K}$ from $\mb{D}(X, \sigma) \times \mb{D}(X, \sigma)$ to a collection of completely bounded maps on $\sigma(M)'$.
For any choice of $k \in \mb{N}$, $\{\mf{z}_i \}_{i=1}^k \subseteq \mb{D}(X, \sigma)$,  and two collections of operators $\{ B_i \}_{i=1}^k$ and $\{ F_i \}_{i=1}^k $ in $\ms{B}\left( H \right)$, there is an associated map $\ms{A}$ from the $k \times k$ matrices with entries in $\sigma(M)'$ to the $k \times k$ matrices with entries in $\ms{B}\left( H \right)$ that is defined at $[A_{ij}]_{i,j = 1}^k$ to be
\begin{equation} 
\label{e3}
\left[B_i \cdot \ms{K}(\mf{z}_i,\mf{z}_j)(A_{ij})  \cdot B_j^* - F_i \cdot \ms{K}(\mf{z}_i, \mf{z}_j)(A_{ij})  \cdot F_j^*\right]_{i,j = 1}^k.
\end{equation}  
The conclusion of our weighted Nevanlinna-Pick theorem is that $\ms{A}$ is completely positive if and only if there exists $Y \in H^\infty(E,Z)$ such that $\Vert Y \Vert \leq 1$ and 
\begin{equation}
\label{e4}
B_i \left( \widehat{Y}(\mf{z}_i) \right) = F_i, \quad 1 \leq i \leq k.
\end{equation}
A comparison of equations \eqref{e1} and \eqref{e4}, as well as matrices  \eqref{e2} and \eqref{e3}, reveals the similarities with the original interpolation problem.  
Indeed, when $M = E = \mb{C}$, $\sigma$ is the one-dimensional representation of $M$, $Z$ is the constant sequence at $1$, and $B_i = 1$ for every $i$,  matrix  \eqref{e3} reduces to matrix \eqref{e2} with minor technical adjustments, and we obtain the original result of Nevanlinna and Pick.  
Even in the scalar case, our theorem is more general than that of Nevanlinna and Pick, for it applies to certain weighted Hardy spaces, precisely those that have the so-called \emph{complete Pick property} described in \cite{Agler2002}.  
For example, our results apply to the Hardy and Dirichlet spaces, but not the Bergman space.

Commutant lifting, pioneered by Sarason \cite{Sarason1967} and Sz.-Nagy and Foia\c{s} \cite{Sz.-Nagy1968}, serves as the principle tool for the proof of Muhly and Solel's \emph{unweighted} Nevanlinna-Pick interpolation result, Theorem 5.3 in \cite{Muhly2004a}.  
Wishing to follow suit, we must first establish that commutant lifting can be done in the weighted case.  
That is, an operator on a co-invariant subspace for an induced representation of $H^\infty(E,Z)$ that commutes with the compression of the image of the  representation may be lifted to a commuting operator on the \emph{full} induced space without increasing the norm. 
As it turns out, our proof of the weighted commutant lifting theorem is where we differ most substantially from the unweighted case.
Muhly and Solel's commutant lifting result, Theorem 4.4 of \cite{Muhly1998a}, is proven using so-called isometric dilations, but in our setting, the weights create obstacles that make this method difficult or even impossible.   
Instead, we adapt the proof for a  commutant lifting theorem given by Arias in \cite{Arias2004}, an argument that ultimately makes use of Parrott's lemma \cite{Parrott1978}.  
Along with the results themselves, this new $W^*$-approach to commutant lifting utilizing technology that works in our more general setting is one of the paper's primary attractions.
The $W^*$-setting of our results provides for a generality that encompasses, for instance,  the unweighted commutant lifting and Nevanlinna-Pick theorems of Muhly and Solel in  \cite{Muhly1998a} and \cite{Muhly2004a} and the weighted lifting and interpolation theorems of Popescu in \cite{Popescu2010}.

The paper is organized as follows.
In Section \ref{prelim} we establish definitions and notation; for convenience, we have organized this material into three subsections.
Section \ref{onbChap} is a technical section devoted to the construction of a family of orthonormal bases in preparation for the weighted commutant lifting theorem, Theorem \ref{WCL_theorem}, the principal result of Section \ref{wclChap}.
In Section \ref{dblcomm} we give a weighted double commutant theorem,  Theorem \ref{double_comm_theorem}, that extends its  unweighted analogue in \cite{Muhly2004a}.
Section \ref{wnpChap} contains our main result, the weighted Nevanlinna-Pick interpolation theorem, Theorem \ref{WeightedNevanlinnaPick}.

\section{Preliminaries} \label{prelim}

\subsection{$W^*$-correspondences and the Unweighted Hardy Algebra}

We let $\mb{N}_0 = \mb{N} \cup \{0\}$. 
Hilbert spaces have inner products that are linear in the \emph{second} variable and conjugate linear in the first.  
Throughout the paper, $M$ will denote a $W^*$-algebra, that is a $C^*$-algebra that is also a dual space, thought of abstractly, without reference to a particular representation on Hilbert space.  
Likewise, $E$ will denote a $W^*$-correspondence over $M$ in the sense of Section 2 in \cite{Muhly2004a}.  
That is, $E$ is a self-dual Hilbert $C^*$-module over $M$ in the sense of \cite{Lan95} and \cite{Paschke1973} with a second-linear inner product that is also a left $M$-module with respect to a faithful normal $*$-homomorphism $\varphi:M \to \ms{L}(E)$ where $\ms{L}(E)$ denotes the $W^*$-algebra of adjointable operators on $E$.  
For simplicity, we assume that $\varphi$ is unital. 
At certain points, we will add the assumption that the right action of $E$ is \emph{full} in the sense that the ultraweakly closed linear span of $\{\langle \xi, \eta \rangle \mid \xi, \eta \in E \}$ is all of $M$.  
By Proposition 3.8 of \cite{Paschke1973}, $E$ is a dual space, and we refer to the weak-$*$ topology on $E$ as its \emph{ultraweak} topology.
If $E$ is nonzero, then $E$ has an orthonormal basis, that is, a family $A \subseteq E$ that is maximal with respect to the following two properties: for every $\alpha \in A$,    $ \langle {\alpha}, {\alpha} \rangle$ is a nonzero projection in $N$, and if $ \alpha, \beta \in A$ and $\alpha \neq \beta$, then $  \langle {\alpha}, {\beta} \rangle = 0$ (\cite{Paschke1973}, proof of Theorem 3.12).

If $N$ is a $W^*$-algebra and $F$ is an $(M,N)$ $W^*$-correspondence with left action map $\sigma: M \to \ms{L}(F)$, then there is an $(M,N)$ $W^*$-correspon\-dence,  denoted $E \otimes_\sigma F$, formed by taking the self-dual completion of a quotient of the algebraic tensor product of $E$ and $F$, balanced over $N$; the quotient is determined by the semi inner product satisfying  $\langle \xi_1 \otimes \eta_1, \xi_2 \otimes \eta_2 \rangle = \langle \eta_1, \sigma \langle \xi_1, \xi_2 \rangle (\eta_2) \rangle$ for $\xi_1, \xi_2 \in E$ and $\eta_1, \eta_2 \in F$.
To give the left action of $M$ on $E \otimes_\sigma F$, first form the \emph{induced representation of $\ms{L}(E)$}, $\sigma^E: \ms{L}(E) \to \ms{L}(E \otimes_\sigma F)$,  defined at $S \in \ms{L}(E)$ by $\sigma^E (S) = S \otimes I_F$, as in \cite{R1974b}.  
Then the left action of $M$ on $E \otimes_\sigma F$ is $\sigma^E \circ \varphi$, called the \emph{induced representation of $M$}, mapping $a \in M$ to $\varphi(a) \otimes I_F$ in $\ms{L}(E \otimes_\sigma F)$.  
We observe that an $(M, \mb{C})$ $W^*$-correspondence is simply a Hilbert space $H$ together with a normal unital $*$-homomorphism $\sigma: M \to \ms{B}(H)$; thus, taking $F$ to be $H$  we obtain the Hilbert space $E \otimes_\sigma H$, called the \emph{induced representation space}, and the representation $\sigma^E \circ \varphi: M \to \ms{B}(E \otimes_\sigma H)$.
If, instead, we inductively take $F$ to be $E$ in the tensor product construction, we form the \emph{tensor powers} of $E$,  $\{\tens{E}{k}\}_{k = 0}^{\infty}$, a family of $W^*$-correspondences over $M$.  We write  $\varphi_k$ for the left action of $M$ on $\tens{E}{k}$ for each $k \in \mb{N}_0$.  
To be precise,  $\tens{E}{0}: = M$,  $\tens{E}{1}: = E$, and  $\tens{E}{k}: = E \otimes_{\varphi_{k-1}} \tens{E}{k-1}$ for $k \geq 2$ where  $\varphi_0$ is left multiplication and $\varphi_1 = \varphi$.

The  \emph{Fock space of $E$}, $\ms{F}(E)$, is the ultraweak direct sum of the tensor powers of $E$; that is $\ms{F}(E) := \sum_{k = 0}^\infty \oplus \tens{E}{k}$.   
The Fock space is a $W^*$-correspondence over $M$ with respect to the left action map $\varphi_\infty: M \to \LFE$, defined at $a \in M$ by $ \varphi_\infty(a) = diag [ \varphi_0(a), \varphi_1(a), \varphi_2(a), \ldots ]$.
Along with the class of left action maps, a second important class of operators in $\ms{L}(\ms{F}(E))$ is the class of \emph{creation operators}.
The creation operator $T_\xi$ determined by $\xi \in E$ is defined at $\eta \in \F{E}$ by $T_\xi \left(  \eta \right) = \xi \otimes \eta$;  
matricially, $T_\xi$ has a subdiagonal matrix,
\begin{equation*}
T_\xi = 
\left[ 
\begin{matrix}
0 & 0   & 0 & \\ 
T_\xi^{(0)} & 0 & 0 &  \\ 
0 & T_\xi^{(1)} & 0 &  \\ 
0 & 0 & T_\xi^{(2)} & \ddots \\ 
&  & \ddots & \ddots
\end{matrix}
\right],
\end{equation*}
where for each $j \in \mb{N}_0$,  $T_\xi^{(j)}: E^{\otimes j} \to E^{\otimes (j+1)}$ is defined at $\eta \in E^{\otimes j} $ by $T_\xi^{(j)} (\eta) = \xi \otimes \eta$.  
For arbitrary $k \in \mb{N}_0$ and $\xi \in \tens{E}{k}$, the creation operator $T_\xi$ is defined in an analogous fashion.
The \emph{algebraic tensor algebra}, $\mc{T}^0_+(E)$, is the subalgebra of $\LFE$ generated by the left action and creation operators.  
The principal object of study in the unweighted setting of \cite{Muhly2004a} is the \emph{Hardy algebra}, $H^\infty(E)$, the closure of $\mc{T}^0_+(E)$ in the ultraweak topology of the $W^*$-algebra $\LFE$.

If $F$ is an $(M,N) \ W^*$-correspondence with left action $\sigma$ and $\xi \in E$, the  \emph{(left) insertion operator} associated with $\xi$ is the map $L^F_\xi \in \ms{L}(F ,E \otimes_\sigma F)$, defined at $\eta \in F$ by $L^F_\xi (\eta) = \xi \otimes \eta$.
We omit the superscript $F$ if it is clear from context. 
We note that a creation operator is a specific example of an insertion operator. 
The following facts are easily verified for $\xi, \eta \in E$.
The operator  $L_\xi$ is bounded and $\Vert L_\xi \Vert \leq \Vert \xi \Vert$.  
 For $a \in M$ and $c \in N$,
 $L_{a \cdot \xi \cdot c} = (\varphi(a) \otimes I_F) L_\xi \sigma(c)$. 
If $S \in \ms{L}(E)$,  $T \in \ms{L}(F)$, and $T$ is a left $M$-module homomorphism, 
 $L_{S\xi}  T = (S \otimes T)  L_\xi$.  
 For any $\zeta \in F$,
$L_\xi^*(\eta \otimes \zeta) = \langle \xi, \eta \rangle \cdot \zeta $; hence, 
 $L_\xi^*  L_\eta = \sigma \langle \xi,\eta \rangle$.  
 Finally,
$L_\xi L_\eta^* = \theta_{\xi, \eta} \otimes I_F$ where $\theta_{\xi,\eta} \in \ms{L}(E)$ denotes the rank one operator defined at $\zeta \in E$ by $\theta_{\xi,\eta}(\zeta) = \xi \cdot \langle \eta, \zeta \rangle$.

\subsection{Duality}

Let $\sigma: M \to \ms{B}(H)$ be a  normal, unital $*$-homomor\-phism for a Hilbert space $H$.
If $\psi: M \to \ms{B}(K)$ is another such map for Hilbert space $K$, $\mc{I}(\sigma, \psi)$ denotes the space of \emph{intertwiners}, i.e. the collection of operators $T \in \ms{B}(H, K)$ such that $T \circ \sigma(a) = \psi(a) \circ T$ for every $a \in M$.
With this notation, we define $E^\sigma:= \mc{I}(\sigma, \sigma^E \circ \varphi)$, called the \emph{$\sigma$-dual of $E$}, a subspace of $\ms{B}(H, E \otimes_{\sigma} H)$. 
By Proposition 3.2 of \cite{Muhly2004a},  $E^\sigma$ is a $W^*$-cor\-res\-pon\-dence over the $W^*$-algebra $\sigma(M)'$ with $a \cdot \xi \cdot b : = (I_E \otimes a) \xi b$ and $\langle \xi, \eta \rangle : = \xi^* \eta$ for $a, b \in \sigma(M)'$ and $\xi, \eta \in E^\sigma$. 
We write $\varphi': \sigma(M)' \to \ms{L}(E^\sigma)$ for the left action map.

Let us recall several maps that arise in \cite{Muhly2004a}, simultaneously providing notation for future use.  
In this subsection, we will suppose that $E$ is full.
For $k \in \mb{N}$, there  is a well-defined  $\sigma(M)'$ $W^*$-cor\-res\-pon\-dence isomorphism $\Lambda_k^\sigma : (E^\sigma)^{\otimes k} \to (E^{\otimes k})^\sigma$ with
\begin{equation*}
\Lambda_k^\sigma \left( \substack{k \\ \otimes \\ i = 1} \xi_i \right) = (I_{k-1} \otimes \xi_1) \cdots (I_2 \otimes \xi_{k-2})(I_1 \otimes \xi_{k-1}) \xi_k, \quad\substack{k \\ \otimes \\ i = 1} \xi_i  \in (E^\sigma)^{\otimes k}
\end{equation*}
where $I_j$ denotes the identity operator on $\tens{E}{j}$.
To include $k = 0$, define $\Lambda_0^\sigma: (E^\sigma)^{\otimes 0} \to (E^{\otimes 0})^\sigma$ at $A \in \sigma(M)'$ and $ h \in H $ by $
(\Lambda_0^\sigma \left( A \right)) (h) = 1_M \otimes Ah$. 
The inclusion map $\iota: \sigma(M)' \to \ms{B}(H)$ is a faithful, normal, unital $*$-homomorphism, and for every $k \in \mb{N}_0$, the map $U_k^\sigma : (E^\sigma)^{\otimes k} \otimes_\iota H \to E^{\otimes k} \otimes_\sigma H $ defined for $\xi \in (E^\sigma)^{\otimes k}$ and $h \in H$ by $U_k^\sigma( \xi \otimes h) = \Lambda_k^\sigma(\xi)(h)$ 
is a Hilbert space isomorphism identifying the induced representation spaces.
Equivalently, $U_k^\sigma  L_\xi^H = \Lambda_k^\sigma(\xi)$.
The map $U_\infty^\sigma: = \sum_{k = 0}^\infty \oplus U_k^\sigma$ identifies the spaces $\ms{F}(E^\sigma) \otimes_\iota H$ and $\ms{F}(E) \otimes_\sigma H$ (\cite{Muhly2004a}; Lemma 3.8).
Let $Ad(U_\infty^\sigma): \ms{B} \left( \ms{F}(E^\sigma) \otimes_\iota H\right) \to \ms{B} \left(  \ms{F}(E) \otimes_\sigma H\right)$, denote the isomorphism  that sends an operator $T \in \ms{B} \left( \ms{F}(E^\sigma) \otimes_\iota H\right)$ to $U_\infty^\sigma T U_\infty^{\sigma *}$.  
It follows that $\rho^\sigma: \ms{L}(\ms{F}(E^\sigma)) \to \ms{B}(\ms{F}(E) \otimes_\sigma H)$, defined by $
\rho^\sigma = Ad(U_\infty^\sigma) \circ \iota^{\ms{F}(E^\sigma)}$
is a faithful, normal, unital $*$-homomorphism; $\rho^\sigma$ is called ``$\rho$'' in Theorem 3.9 of \cite{Muhly2004a}.
Specifically, for $Y \in \ms{L}(\ms{F}(E^\sigma))$,
$\rho^\sigma(Y) = U_\infty^\sigma (Y \otimes I_H) U_\infty^{\sigma *}$.  
The map $\pi^\sigma: \LFE \to \ms{B}(\ms{F}(E^\sigma) \otimes_\iota H)$, defined by $
\pi^\sigma = Ad(U_\infty^{\sigma *}) \circ \sigma^{\ms{F}(E)}$
 is also a faithful, normal, unital $*$-homomorphism; $\pi^\sigma$ is called ``$\rho$'' in Section 5 of \cite{Muhly2004a}.
For $Y \in \ms{L}(\ms{F}(E))$, $\pi^\sigma(Y) = U_\infty^{\sigma *} (Y \otimes I_H) U_\infty^\sigma$.

To summarize, using the $W^*$-algebra $M$, the $W^*$-correspondence over the algebra $E$, and the representation of the algebra $\sigma$, we've defined the maps $\Lambda_k^\sigma$, $U_k^\sigma$, $U_\infty^\sigma$, $\rho^\sigma$, and $\pi^\sigma$.
We repeat the process with the $W^*$-algebra $\sigma(M)'$, the correspondence $E^\sigma$, and the representation $\iota$.
The analogue of $\sigma(M)'$ is the commutant of the image of $\iota$ in $\ms{B}(H)$, which is precisely $\sigma(M)$.  
The analogue of $\iota$ is the inclusion map that we denote by $\jmath: \sigma(M) \to \ms{B}(H)$.  
We write $E^{\sigma \iota}$ in place of $(E^\sigma)^\iota$, a $W^*$-correspondence over $\sigma(M)$ with left action $\varphi_\infty'': \sigma(M) \to \ms{L}(\ms{F}(E^{\sigma \iota}))$.
We obtain the analogous collection of maps $\Lambda_k^\iota$, $U_k^\iota$, $U_\infty^\iota$, $\rho^\iota$, and $\pi^\iota$.
Repeating the process once again would utilize the $W^*$-algebra $\sigma(M)$, the correspondence $E^{\sigma \iota}$, and the representation $\jmath$, but if $\sigma$ is faithful this is unnecessary since these constructs may be naturally identified with $M$, $E$, and $\sigma$ as follows.  
The identifications of $M$ with $\sigma(M)$ and $\sigma$ with $\jmath$ are immediate.  
After appropriate identifications, the map $\omega: E \to E^{\sigma \iota}$ defined at  $\xi \in E$ by $\omega(\xi) = U_1^{\sigma *}  L^H_\xi$ is an isomorphism of $W^*$-correspondences that is studied in and around Theorem 3.6 of \cite{Muhly2004a}.
Defining $\omega_k : = \omega \otimes \cdots \otimes \omega$ identifies $\tens{E}{k}$ with $\tens{(E^{\sigma \iota})}{k}$ for each $k$, and $\omega_\infty: = \sum_{k = 0}^\infty \omega_k$ identifies $\ms{F}(E)$ with $\ms{F}(E^{\sigma \iota})$.
For every $k \in \mb{N}_0$, $U_k^\sigma U_k^\iota (\omega_k \otimes I_H)  = I_{E^{\otimes k} \otimes_\sigma H}$, as demonstrated in and around Corollary 3.10 of \cite{Muhly2004a}.  
Equivalently $U_k^\sigma  \circ \Lambda_k^\iota (\omega_k \xi)  = L^H_\xi$ for every $\xi \in \tens{E}{k}$.
More generally, $U_{k+m}^\sigma \left( I'_{m} \otimes  \Lambda_k^\iota (\omega_k \xi) \right)=  L_\xi^{(\tens{E}{m} \otimes_\sigma H)} U_m^\sigma$  for $k,m \in \mb{N}_0$ and $\xi \in \tens{E}{k}$,
where $I_m'$ denotes the identity operator on $\tens{(E^\sigma)}{m}$.   
It follows that
$U_\infty^\sigma U_\infty^\iota (\omega_\infty \otimes I_H)  = I_{\F{E} \otimes_\sigma H}$.
Equations (5.1) and (5.2) of \cite{Muhly2004a} show how $\pi^\sigma$ acts on the generators of  $H^\infty(E)$:
for $a \in M$, $\pi^\sigma (\phiinf{a})  = I'_{\infty} \otimes \sigma(a)$, and for $\xi \in E$, $\pi^\sigma (T_\xi )$ has the subdiagonal matrix,
\begin{equation*}
\pi^\sigma (T_\xi ) = 
\left[ 
\begin{matrix}
0 & 0   & 0 & \\ 
\omega\xi & 0 & 0 &  \\ 
0 & I'_{1} \otimes \omega \xi & 0 &  \\ 
0 & 0 & I'_{2} \otimes \omega \xi & \ddots \\ 
&  & \ddots & \ddots
\end{matrix}
\right].
\end{equation*}

\subsection{The Weighted Hardy Algebra} 
We now turn to the weighted setting of \cite{Muhly2016}.  
Throughout, the sequence $X = \{X_k\}_{k = 0}^{\infty}$ will denote an \emph{admissible} sequence; this means that
$X_k \in \varphi_{k}(M)^c$ for each $k \in \mb{N}_0$, 
$X_k \geq 0$ for each $k \in \mb{N}_0$, 
$X_0 = 0$, 
$X_1$ is invertible, and finally that 
$\limsup_{k \to \infty} \Vert X_k \Vert ^{1/k} < \infty$,
where $\varphi_{k}(M)^c$ denotes the commutant of $\varphi_{k}(M)$ in $\ms{L}(\tens{E}{k})$ 
(\cite{Muhly2016}, Definition 4.1). 
Let  $R = \{R_k\}_{k = 0}^{\infty}$ denote the sequence 
\begin{equation*} 
R_k  = 
\begin{cases}
I_M, &\text{if  } k = 0\\
\left( \sum_{j = 1}^k \left( \sum_{\alpha \in \mc{F}(k,j)} \substack{j \\ \otimes \\ i = 1}  X_{\alpha(i)} \right)\right)^{1/2},  &\text{if  } k > 0
\end{cases},
\end{equation*}
where 
$\mc{F}(k,j): = \left\{ \alpha: \{1, \ldots, j\} \to \mb{N} \bigm| \sum_{i = 1}^j \alpha(i) = k \right\}$ if $1 \leq j \leq k$ (\cite{Muhly2016}, Equation 4.4).
Each $R_k$ is a positive, invertible element in $\varphi_k(M)^c$.
A sequence of operators $Z = \{Z_k\}_{k = 0}^{\infty}$  is called a \emph{sequence of weights associated with $X$} if $Z_k$ is invertible and belongs to $\varphi_{k}(M)^c$ for each $k \in \mb{N}_0$, $Z_0 = I_M$, and $Z^{(k)*} Z^{(k)} = R_k^{-2}$ for all $k \in \mb{N}_0$, where $Z^{(k)}$, thought of as the product of the weights $\{Z_k, Z_{k-1}, \ldots, Z_1, Z_0 \}$, is defined
$Z^{(k)} :  =
Z_k (I_1 \otimes Z_{k-1} ) \cdots (I_{k-1} \otimes Z_{1}) (I_k \otimes Z_0) \in \varphi_{k}(M)^c$
(\cite{Muhly2016}, Definition 4.6).  
We define $Z^{(k,j)} := Z^{(k)}(I_{k - j} \otimes Z^{(j)})^{-1}$ when $0 \leq j \leq k$, equivalently  $Z^{(k,j)} = 
Z_k (I_1 \otimes Z_{k-1} ) \cdots (I_{ k-j-1} \otimes Z_{j+1})$, which may be thought of as the product of the weights $\{Z_k, Z_{k-1}, \ldots, Z_{j + 1} \}$.
One example of a weight sequence is the \emph{canonical sequence} $Z = \{Z_k\}_{k = 0}^\infty$, with $Z_0 : = I_M$ and $Z_k: = R_k^{-1}(I_1 \otimes R_{k-1}), k \geq 1$.  
We note that the unweighted setting studied, for instance, in \cite{Muhly1998a} and \cite{Muhly2004a} occurs when $Z$ is the sequence of identity operators.  
To work in the setting of \cite{Popescu2010}, we take the $W^*$-correspondence $E$ to be $\mb{C}^d$ and add the condition that $X$ be a sequence of diagonal operators. 

If $Z$ is a sequence of weights associated with $X$, then $\sup_{j \in \mb{N}_0} \Vert Z_j \Vert < \infty$ (\cite{Muhly2016}, proof of Theorem 4.5).  
It follows that for every $k \in \mb{N}_0$ the operator $D^Z_k = D_k  := diag [ 0, \ldots, 0,  Z^{(k)}, Z^{(k+1,1)} , Z^{(k+2,2)}, \ldots ]$ belongs to $\LFE$.  
We omit the superscript `$Z$' when the weight sequence is understood.
For $\xi \in E^{\otimes k}$, we define the \emph{weighted creation operator} $W^Z_\xi = W_\xi$ in $\ms{L}(\ms{F}(E))$ by $W_\xi  : = D_k \cdot T_\xi$.  
The matrix for $W_\xi$ is $k$-subdiagonal;
\begin{equation*}
W_\xi = 
\left[ 
\begin{matrix}
0 & 0 & 0 & \\
\vdots & \vdots & \vdots &  \\
0 &  &  &  \\ 
Z^{(k)} T_\xi^{(0)} & 0 & &  \\ 
0 & Z^{(k+1,1)} T_\xi^{(1)}& 0 &  \\ 
\vdots & 0 & Z^{(k + 2,2)} T_\xi^{(2)}& \ddots \\ 
&  & \ddots & \ddots
\end{matrix}
\right].
\end{equation*}
where $T_\xi^{(j)}$ maps  $\eta \in \tens{E}{j}$ to $\xi \otimes \eta  \in \tens{E}{j + k}$ for $j \in \mb{N}_0$.  
We observe that while our arrival at the definition of $W_\xi$ differs from that in \cite{Muhly2016}, a comparison of the preceding matrix with equation (3.2) in \cite{Muhly2016} confirms that the two definitions agree.
When $k,l \in \mb{N}_0$; $a, b \in M$; $\xi \in \tens{E}{k}$; and $\eta \in \tens{E}{l}$; 
$W_{a \cdot \xi \cdot b}  = \varphi_\infty(a) \circ W_\xi \circ \varphi_\infty(b)$ and $W_\xi \circ W_\eta  = W_{\xi \otimes \eta}$.  
For $a \in M$, $W_a = \varphi(a)$.
The \emph{$Z$-algebraic tensor algebra}, $\mc{T}^0_+(E,Z)$, is the subalgebra of $\LFE$ generated by the left action and weighted creation operators. 
The \emph{$Z$-Hardy algebra}, $H^\infty(E,Z)$, is the closure of $\mc{T}^0_+(E,Z)$ in the ultraweak topology of $\LFE$. 

Let $\sigma: M \to \ms{B}(H)$ be a normal, unital $*$-homomorphism.  
We observe that $\mc{I}(\sigma^E \circ \varphi, \sigma)^* = \mc{I}( \sigma, \sigma^E \circ \varphi) = E^\sigma$, which was defined above. 
For a point $\mf{z} \in \mc{I}(\sigma^E \circ \varphi, \sigma)$ and $k \in \mb{N}_0$, we define the \emph{$k^{th}$ tensorial power}, $\mf{z}^{(k)}: =\mf{z} (I_1 \otimes \mf{z}) \cdots (I_{k-1} \otimes \mf{z})$.
The intertwining property of $\mf{z}$ guarantees that $\mf{z}^{(k)}$ is a well-defined operator in $\ms{B}(\tens{E}{k} \otimes_\sigma H,H)$; moreover,  $\mf{z}^{(k)} \in \mc{I}(\sigma^{\tens{E}{k}} \circ \varphi_k, \sigma)$.  
Computation shows that $\mf{z}^{(k+l)} = \mf{z}^{(k)}(I_k \otimes \mf{z}^{(l)})$ for $k,l \in \mb{N}_0$.
We define
\begin{equation*}
\mb{D}(X, \sigma) := \left\{\mf{z} \in \mc{I}(\sigma^E \circ \varphi, \sigma) \biggm| \left\Vert \sum_{k = 1}^\infty \mf{z}^{(k)} (X_k \otimes I_H) \mf{z}^{(k)*} \right\Vert < 1 \right\};
\end{equation*}
(\cite{Muhly2016}, Definition 4.3).  
Here and throughout  unless stated otherwise, infinite sums indicate convergence with respect to  the ultraweak topology for a $W^*$-correspondence or $W^*$-algebra and the norm topology for a Hilbert space.  
If $\mf{z} \in \mb{D}(X, \sigma)$, there is an ultraweakly continuous, completely contractive representation $(\sigma \times \mf{z}) : H^\infty(E, Z) \to \ms{B}(H)$ such that $(\sigma \times \mf{z}) (\phiinf{a})  = \sigma(a)$ for every $a \in M$ and $(\sigma \times \mf{z}) (W_\xi)  = \mf{z}^{(k)} L_\xi$ for every $\xi \in \tens{E}{k}$ with $k \geq 1$ (\cite{Muhly2016}, Corollary 5.9).
For  $Y \in H^\infty(E, Z)$, we define the noncommutative function $\widehat{Y}: \mb{D}(X, \sigma) \to \ms{B}(H)$ at the point $\mf{z} \in \mb{D}(X, \sigma)$ by $\widehat{Y}(\mf{z}) = (\sigma \times \mf{z})(Y)$.

\section{An Orthonormal Basis for Tensor Products} \label{onbChap}

The proof of our weighted commutant lifting theorem in Section \ref{wclChap} is modeled after the proof of a commutant lifting theorem by Alvaro Arias in \cite{Arias2004}.  
Working in the setting of complex $n$-space, Arias makes use of the fact that if $\{e_i\}_{i = 1}^n$ is an orthonormal basis for $\mb{C}^n$ and $k \in \mb{N}$, then one orthonormal basis for $(\mb{C}^n)^{\otimes k}$ consists of the simple tensors $\substack{k \\ \otimes \\ i = 1} e_{\alpha(i)}$ such that $\alpha$ is any function from $\{1, \ldots, k\}$ to $\{1, \ldots, n\}$.
In search of a similarly constructed orthonormal basis for $\tens{E}{k}$ or, more generally, for any correspondence formed by tensor product, let us consider an obstacle that must be overcome. 
Let $F$ be a nonzero $(M,N)$ $W^*$-correspondence with orthonormal basis $A$, and let $G$ be a nonzero $(N,P)$ $W^*$-correspondence with left action $\sigma$ and orthonormal basis $B$ for $W^*$-algebras $M$, $N$, and $P$. 
If $M = N = P = \mb{C}$, then $F$ and $G$ are simply Hilbert spaces,  the $W^*$-tensor product of $F$ and $G$ is the same as their Hilbert space tensor product, and 
$\{ \alpha \otimes \beta: \alpha \in A, \beta \in B \}$,
which is a Hilbert space orthonormal basis for the tensor product, is also an orthonormal basis in the sense of Hilbert $W^*$-modules.
For general $M$, $N$, and $P$, the issue is more complicated.  
If $\alpha \in A$ and $\beta \in B$, there is no reason to think that  $\langle \alpha \otimes \beta, \alpha \otimes \beta \rangle$ is a nonzero projection in $P$, in which case $\alpha \otimes \beta$ cannot belong to an orthonormal basis for $F \otimes_\sigma G$.  
In fact, it is possible that $F \otimes_\sigma G$ is zero, in which case $F \otimes_\sigma G$ does not \emph{have} an orthonormal basis.
We begin this section by overcoming these obstacles to construct an orthonormal basis for a nonzero tensor product of two $W^*$-correspondences that consists of simple tensors.  
Then we obtain, via an inductive process, a family of orthonormal bases, one basis for each nonzero $\tens{E}{k}$, which we use in Section \ref{wclChap} to prove the weighted commutant lifting theorem using Arias' technique.

Along with $M$, let $N$ and $P$ be $W^*$-algebras.  
Let $F$ be an $(M,N)$ $W^*$-correspondence, and let $G$ be an $(N,P)$ $W^*$-correspondence with left action $\sigma: N \to \ms{L}(G)$.  
We will write $a \cdot \eta$ in place of $\sigma(a)(\eta)$ for $a \in N$ and $\eta \in G$.  
If $q$ is a projection in $N$, then $q \cdot G = \{ q \cdot \eta \mid \eta \in G \}$ is the kernel of $\sigma(1_N - q)$.  
Thus $q \cdot G$ is an ultraweakly closed (right) $P$-submodule of $G$ and is therefore a Hilbert $W^*$-module over $P$
(\cite{Baillet1988}; Consequence 1.8).  
As such,  for every projection $q \in N$ such that $q \cdot G \neq \{0\}$ we may fix an orthonormal basis for $q \cdot G$.
Define $Q$ to be the, possibly empty, subset of $F$,
\begin{equation*}
Q : =  \left\{\xi \in F \mid \langle \xi, \xi \rangle \text{ is a projection in $N$  and } \langle \xi, \xi \rangle \cdot G \neq \{0\} \right\}.
\end{equation*}
If $Q$ is nonempty and $\xi \in Q$, define $B(\xi)$ to be the orthonormal basis for $\langle \xi, \xi \rangle \cdot G$ chosen above. 
Fix $A$, an orthonormal basis for $F$.  
If $\alpha \in A$, then $\langle \alpha, \alpha \rangle$ is a projection in $N$. 
The subset of $F \otimes_\sigma G$,
\begin{equation*}
C :=\{ \alpha \otimes \beta \mid \alpha \in A \cap Q \text{ and } \beta \in B(\alpha) \},
\end{equation*}
is empty if and only if $A \cap Q$ is empty.
For simplicity, when we say that ``$\alpha \otimes \beta \in C$'', we mean that $\alpha \in A \cap Q \text{ and } \beta \in B(\alpha)$. 
The emphasis is needed since the expression of an element in $F \otimes_\sigma G$ in terms of simple tensors is not unique.
If $\alpha \otimes \beta \in C$, then $\langle \alpha, \alpha \rangle \cdot \beta = \beta$.  
In our first theorem, we show that when $F \otimes_\sigma G$ is nonzero, $C$ is an orthonormal basis.

\begin{theorem}\label{onbEtensF}
The correspondence $F \otimes_\sigma G$ is nonzero if and only if $C$ is nonempty.  
In this case, $C$ is an orthonormal basis for $F \otimes_\sigma G$, and if $\alpha_i \otimes \beta_i \in C$ for $i = 1,2$, then
\begin{equation} 
\label{onbEtensF_e4}
\langle \alpha_1 \otimes \beta_1, \alpha_2 \otimes \beta_2 \rangle
= \begin{cases}
\langle \beta_2, \beta_2 \rangle & \text{ if } \alpha_1 = \alpha_2 \text{ and } \beta_1 = \beta_2 \\
0 & \text{ otherwise} 
\end{cases}.
\end{equation}
\end{theorem}

\begin{proof} 
First, we show that for any $\xi \in F$ and $\eta \in G$,
\begin{equation}
\label{onbEtensF_e3} 
\xi \otimes \eta  
= \begin{cases}
\sum_{\alpha \in A \cap Q} \ \alpha \otimes  \langle \alpha, \xi \rangle \cdot \eta &  \text{ if } A \cap Q \neq \emptyset \\
0 & \text{ if } A \cap Q = \emptyset
\end{cases}.
\end{equation}
By Fourier expansion with respect to $A$, we have $\xi = \sum_{\alpha \in A} \alpha \cdot \langle \alpha, \xi \rangle$
(\cite{Paschke1973}, proof of Theorem 3.12).
It follows that $\xi \otimes \eta = \sum_{\alpha \in A } \alpha \otimes  \langle \alpha, \xi \rangle \cdot \eta$.
For every $\alpha \in A$, we have $\left\langle  \alpha \cdot \langle \alpha, \alpha \rangle - \alpha, \alpha \cdot \langle \alpha, \alpha \rangle - \alpha \right\rangle = 0$, so $\alpha \cdot \langle \alpha, \alpha \rangle = \alpha$.  
Thus $\langle \alpha, \xi \rangle \cdot \eta$ belongs to $\langle \alpha, \alpha \rangle \cdot G$, so if $\alpha \notin Q$,  then $\langle \alpha, \xi \rangle \cdot \eta = 0$.  
Equation \eqref{onbEtensF_e3} follows.

When $C$ is empty, $A \cap Q$ is also empty, so  by equation \eqref{onbEtensF_e3} every simple tensor in $F \otimes_\sigma G$ is zero.   
Thus $F \otimes_\sigma G = \{0\}$ since $F \otimes_\sigma G$ is the ultraweak closure of the linear span of the simple tensors.

Suppose that $C$ is not empty.  
Using the properties of $A$ and $C$,
equation \eqref{onbEtensF_e4} follows from straightforward computations. 
It follows that for any $\alpha \otimes \beta \in C$, $
\langle \alpha \otimes \beta, \alpha \otimes \beta \rangle$
is a nonzero projection in $P$, so $\alpha \otimes \beta \neq 0$ and $F \otimes_\sigma G \neq \{0\}$.  
It only remains to show that $C$ is an orthonormal basis for $F \otimes_\sigma G$.  
By equation \eqref{onbEtensF_e4}, $C$ is an orthonormal \emph{set}.  
Towards maximality, suppose that $C'$ is an orthonormal subset of $F \otimes_\sigma G$ such that $C \subseteq C'$.  
For any $\xi \in F$ and $\eta \in G$, if $\alpha$ belongs to $A \cap Q$, then as above,  $\langle \alpha, \xi \rangle \cdot \eta$ is an element in $\langle \alpha, \alpha \rangle \cdot G$.  
Using the Fourier expansion with respect to the orthonormal basis $B(\alpha)$,
$ \langle \alpha, \xi \rangle \cdot \eta
 = \sum_{\beta \in B(\alpha)} \beta \cdot \langle \beta, \langle \alpha, \xi \rangle \cdot \eta \rangle 
 = \sum_{\beta \in B(\alpha)} \beta \cdot \langle \alpha \otimes \beta, \xi \otimes \eta \rangle$.
Thus, by equation \eqref{onbEtensF_e3}, 
\begin{equation} \label{onbEtensF_e6} 
\xi \otimes \eta 
= \sum_{\alpha \in A \cap Q} \left( \sum_{\beta \in B(\alpha)} (\alpha \otimes \beta) \cdot \langle \alpha \otimes \beta, \xi \otimes \eta \rangle \right).
\end{equation}
Now if $C$ is \emph{properly} contained in $C'$, then there exists $\zeta \in C'$ such that for every $\alpha \otimes \beta \in C$, $\langle \zeta, \alpha \otimes \beta \rangle = 0$. 
Since $F \otimes_\sigma G$ is the ultraweak closure of the linear span of the simple tensors,  there is a net $\left \{ \sum_{j = 1}^{N_\lambda} \xi_{\lambda, j} \otimes \eta_{\lambda,j} \right\}_\lambda$ converging ultraweakly to $\zeta$ in $F \otimes_\sigma G$.  By equation \eqref{onbEtensF_e6},
\begin{multline*} 
\langle \zeta, \zeta \rangle  
= \lim_{\lambda}  \sum_{j = 1}^{N_\lambda} \langle \zeta, \xi_{\lambda, j} \otimes \eta_{\lambda,j}  \rangle  \\
 = \lim_{\lambda}  \sum_{j = 1}^{N_\lambda}  \left( \sum_{\alpha \in A \cap Q }  \left( \sum_{\beta \in B({\alpha}) } \langle \zeta, \alpha \otimes \beta \rangle \langle \alpha \otimes \beta, \xi_{\lambda, j} \otimes \eta_{\lambda,j} \rangle \right) \right) = 0,
\end{multline*} 
 contradicting the fact that $\langle \zeta, \zeta \rangle$ is a \emph{nonzero} projection.  
Therefore, $C$ is an orthonormal basis for $F \otimes_\sigma G$.
\end{proof}

Let us inductively construct a family $\{C_k\}$ of orthonormal bases, one basis for each nonzero $\tens{E}{k}$.  
First we establish some notation, mirroring the discussion prior to Theorem \ref{onbEtensF}.  
Fix an orthonormal basis for the Hilbert $W^*$-module $q \cdot E$ for every projection $q \in M$ such that $q \cdot E \neq \{0\}$.
Define 
\begin{equation*}
Q = \left\{\xi \in E \mid \langle \xi, \xi \rangle \text{ is a projection in $M$  and } \langle \xi, \xi \rangle \cdot E \neq \{0\} \right\}.
\end{equation*}  
If $\xi \in Q$, let $B(\xi)$ denote the orthonormal basis for  $\langle \xi, \xi \rangle \cdot E$ chosen above. 
Since $E$ is nonzero, let $A$ be an orthonormal basis for $E$.  
Let $C_0 = \{1_M\}$ where $1_M$ is the identity element in $M$.  
Let $C_1=A$.  
When $k \geq 2$,  define $C_k$ to be the possibly empty subset of $\tens{E}{k}$ consisting of the simple tensors $\substack{k \\ \otimes \\ i = 1} \xi_i  = \xi_1 \otimes \cdots \otimes \xi_k$ such that the following properties hold:
\begin{enumerate}[label=(\arabic*), ref=(\arabic*)]
\item \label{onbEk_ep1} $\xi_1 \in A \cap Q$,
\item \label{onbEk_ep2} $\xi_i \in B(\xi_{i-1}) \cap Q$ for every $i$ such that $1 < i < k$, and
\item \label{onbEk_ep3}$\xi_k \in B(\xi_{k-1})$.
\end{enumerate}
For simplicity, when we say that $\substack{k \\ \otimes \\ i = 1} \xi_i \in C_k$, we always mean that $\{\xi_i\}_{i = 1}^k$ satisfies properties   \ref{onbEk_ep1}, \ref{onbEk_ep2}, and \ref{onbEk_ep3}.
Notice that if $\substack{k \\ \otimes \\ i = 1} \xi_i \in C_k$, then for every $i$ such that $1 < i \leq k$, we have $\langle \xi_{i-1}, \xi_{i-1} \rangle \cdot \xi_i = \xi_i$.  
We arrive at the following theorem.

\begin{theorem}\label{onbEk}
For $k \in \mb{N}_0$, the correspondence 
$\tens{E}{k}$ is nonzero if and only if $C_k $ is nonempty.  
In this case, $C_k$ is an orthonormal basis for $E^{\otimes k}$, and for any $\xi = \substack{k \\ \otimes \\ i = 1} \xi_i$ and $ \eta = \substack{k \\ \otimes \\ i = 1} \eta_i $ in $C_k$, 
\begin{equation}
\label{onbEk_e1} 
\langle \xi , \eta \rangle = 
\begin{cases}
 \langle \eta_k, \eta_k \rangle & \text{ if } \xi_i = \eta_i \text{    for every $i$ such that } 1 \leq i \leq k\\
0 & \text{ otherwise }
\end{cases}.
\end{equation}
\end{theorem}

\begin{proof}
The cases when $k = 0$ and $k = 1$ follow readily.
Inductively, suppose the conclusions of the theorem are satisfied for some $k \in \mb{N}$.  
Towards applying Theorem \ref{onbEtensF} to $\tens{E}{k} \otimes_\varphi E$, define
\begin{equation*}
Q_k : =  \left\{\xi \in \tens{E}{k} \mid \langle \xi, \xi \rangle \text{ is a projection in $M$  and } \langle \xi, \xi \rangle \cdot E \neq \{0\} \right\}.
\end{equation*}
If $\xi \in Q_k$, let $B(\xi)$ denote the orthonormal basis for $\langle \xi, \xi \rangle \cdot E$.  
Define
\begin{equation*}
C_{k+1}' : = \left\{ \left(\substack{k \\ \otimes \\ i = 1} \xi_i \right) \otimes \xi_{k+1} \biggm| \substack{k \\ \otimes \\ i = 1} \xi_i \in C_k \cap Q_k \text{ and } \xi_{k+1} \in B\left(\substack{k \\ \otimes \\ i = 1} \xi_i \right) \right\}.
\end{equation*}
Using the inductive assumption we see that $C_{k+1}$ is the image of $C_{k+1}'$ under the natural isomorphism between $\tens{E}{k} \otimes_\varphi E$ and $\tens{E}{k+1}$.  
It now follows from  Theorem \ref{onbEtensF} that $\tens{E}{k+1}$ is nonzero if and only if $C_{k+1}$ is nonempty and that $C_{k+1}$ is an orthonormal basis in that case.
If $\xi = \substack{k+1 \\ \otimes \\ i = 1} \xi_i$ and $\eta = \substack{k+1 \\ \otimes \\ i = 1} \eta_i $ belong to $C_{k+1}$, then $\left(\substack{k \\ \otimes \\ i = 1} \xi_i \right) \otimes \xi_{k+1}$ and $\left(\substack{k \\ \otimes \\ i = 1} \eta_i \right) \otimes \eta_{k+1}$ belong to $C_{k+1}'$ and $\left(\substack{k \\ \otimes \\ i = 1} \xi_i \right)$ and $\left(\substack{k \\ \otimes \\ i = 1} \eta_i \right)$ belong to $C_k$.  
Since $\langle \eta_k, \eta_k \rangle \cdot \eta_{k+1} = \eta_{k+1}$ and 
$\langle \xi, \eta \rangle 
 = \left\langle  \xi_{k+1}, \left\langle \substack{k \\ \otimes \\ i = 1} \xi_i ,\substack{k \\ \otimes \\ i = 1} \eta_i \right\rangle\cdot  \eta_{k+1} \right\rangle$, 
 equation \eqref{onbEk_e1} follows from Theorem \ref{onbEtensF} and the inductive hypothesis.
\end{proof}

We devote the remainder of the section to establishing notation and giving technical lemmas that will be useful in proving our main theorems.  
Define $\mb{S}:=\left\{k \in \mb{N}_0 \mid E^{\otimes k} \neq \{0\} \right\}$, which is either $\mb{N}_0$ or $\{0, 1, \ldots, n\}$ for some $n \in \mb{N}$. 
For $k \in \mb{N}_0$, define $A_k$ to be $C_k$ when $k \in \mb{S}$ and the zero set when $k \notin \mb{S}$. 
If $k \in \mb{N}_0$ then $v_k \in \ms{L}(\tens{E}{k},\ms{F}(E))$ denotes the  isometry mapping $\xi \in \tens{E}{k}$ to $\widehat{\xi} : = ( \delta_{i = k} \ \xi )_{i = 0}^\infty$ where $\delta_{i = k}$ is $1$ when $i = k$ and is otherwise zero. 
A handy fact is that a diagonal operator $S = diag[S_0, S_1, S_2, \cdots] \in \LFE$ may be expressed as $S = \sum_{j = 0}^\infty v_j S_j v_j^*$. 
The insertion operator $L_{\widehat{\xi}}$ maps $x \in H$ to $\widehat{\xi} \otimes x \in \ms{F}(E) \otimes_\sigma H$.
We define $Q_k$ in $\LFE$ to be the projection onto $v_k(\tens{E}{k})$, an ultraweakly closed submodule of $\F{E}$, noting that $Q_k = v_kv_k^*$. 
For $\xi \in \tens{E}{k}$, $\theta_\xi \in \ms{L}(\tens{E}{k})$ denotes the positive rank one operator that sends $\eta \in \tens{E}{k}$ to $\xi \cdot \langle \xi, \eta \rangle$.

\begin{lemma} \label{handysums} 
For $k \in \mb{N}_0$, 
\begin{enumerate}[label=(\arabic*),ref=(\arabic*)]
\item \label{handysums_1} if  $S \in \ms{L}(\tens{E}{k}, F)$ for a  Hilbert $W^*$ $M$-module $F$, then $SS^* = \sum_{\xi  \in A_k } \theta_{S\xi } $ in $\ms{L}(F)$;
\item \label{handysums_2}$Q_k \otimes I_H = \sum_{\xi \in A_k} L_{\widehat{\xi}} L_{\widehat{\xi}}^*$ in $\ms{B}(\ms{F}(E) \otimes_\sigma H)$; and
\item \label{handysums_4} $ \sum_{\xi \in A_k} T_{\xi} T_{\xi}^* $ is the projection in $\LFE$ onto the ultraweakly closed submodule $\{(\zeta_i)_{i=0}^\infty \mid \zeta_i = 0 \text{ when } 0 \leq i < k\}$ of $\ms{F}(E)$.
\end{enumerate} 
\end{lemma}

\begin{proof}
If $k \notin \mb{S}$, then $\tens{E}{k} = \{0\}$, and all of the operators in parts \ref{handysums_1}, \ref{handysums_2}, and \ref{handysums_4} are zero.  
Suppose $k \in \mb{S}$.  
For $\eta \in E^{\otimes k}$ and $\mb{F}$ a finite subset of $A_k$,
\begin{equation*}
0 \leq 
\left\langle \left( \sum_{\xi \in \mb{F}} \theta_{\xi} \right) \eta, \eta \right\rangle 
 = \sum_{\xi \in \mb{F}} \langle \eta, \xi \rangle \langle \xi, \eta \rangle 
 \leq \sum_{\xi \in A_k} \langle \eta, \xi \rangle \langle \xi, \eta \rangle 
 = \langle \eta, \eta \rangle
\end{equation*} 
by a Parseval type identity
(\cite{Paschke1973}, proof of Theorem 3.12).
Therefore, $\sum_{\xi \in A_k } \theta_{\xi}$ converges ultraweakly to a positive operator in the $W^*$-algebra $\ms{L}(\tens{E}{k})$ and for any $\eta \in E^{\otimes k}$, 
$
\left\langle \left( \sum_{\xi \in A_k} \theta_{ \xi} \right) \eta, \eta \right\rangle = \left\langle  \eta, \eta \right\rangle$.
Thus by polarization, 
\begin{equation}
\label{handysums_e2}
I_k = \sum_{\xi  \in A_k } \theta_{\xi}.
\end{equation}
The map $Ad(S):\ms{L}(\tens{E}{k}) \to \ms{L}(F)$ that sends $Y \in \ms{L}(\tens{E}{k})$ to $SYS^*$  is linear and ultraweakly continuous, and $Ad(S)(\theta_{\xi}) = \theta_{S\xi}$ for any $\xi \in A_k$.  
Therefore, part \ref{handysums_1} follows from applying $Ad(S)$ to equation \eqref{handysums_e2}.

Taking $S$ to be $v_k$ in part \ref{handysums_1}, we obtain $\sum_{\xi \in A_k} \theta_{\widehat{\xi}} = Q_k$ in $\LFE$.
Since $\sigma^{\ms{F}(E)}$ is linear and ultraweakly continuous, 
$Q_k \otimes I_H
 = \sum_{\xi \in A_k} \theta_{\widehat{\xi}} \otimes I_H =  \sum_{\xi \in A_k} L_{\widehat{\xi}} L_{\widehat{\xi}}^*$, giving part \ref{handysums_2}.

Let $i \in \mb{N}_0$ such that $i \geq k$.
The induced representation $\varphi_{i-k}^{\tens{E}{k}}$ is linear and ultraweakly continuous, so it follows from equation \eqref{handysums_e2} that
\begin{equation} 
\label{handysums_e1} I_{i}  = 
\sum_{\xi \in A_k}  \theta_{\xi} \otimes I_{i-k} \ .
\end{equation}
For any $\xi \in A_k$, we have $T_\xi T_\xi^* =  diag[0, \ldots, 0, \theta_\xi, \theta_\xi \otimes I_1, \theta_\xi \otimes I_2, \ldots] = \sum_{i = k}^\infty  v_i (\theta_{\xi} \otimes I_{i-k}) v_i^*$.
Summing over all $\xi \in A_k$ and using equation \eqref{handysums_e1},
\begin{multline*} 
\sum_{\xi \in A_k } T_\xi T_\xi^*
 = \sum_{\xi \in A_k } \left(\sum_{i = k}^\infty  v_i ( \theta_{\xi}  \otimes I_{ i-k} ) v_i^*  \right) \\
 = \sum\limits_{i = k}^\infty  v_i \left( \sum_{\xi \in A_{k} } \theta_{\xi} \otimes I_{i-k} \right) v_i^* 
 = \sum_{i = k}^\infty  Q_i.
\end{multline*}
noting that sums may be interchanged as all terms are non-negative.
It is readily shown that $\sum_{i = k}^\infty  Q_i$ is the projection in $\LFE$ onto the ultraweakly closed submodule $\{(\zeta_i)_{i=0}^\infty \mid \zeta_i = 0 \text{ when } 0 \leq i < k\}$.
\end{proof}

The following technical lemma is similar in flavor to Lemma 5.2 of \cite{Muhly2016}.

\begin{lemma}\label{sumstoprojection}
If $Z$ is a  sequence of weights associated with $X$,
\begin{equation*}
 \sum_{j=1}^\infty  \left( \sum_{\xi \in A_j} W_{X_j^{1/2} \xi} W^*_{X_j^{1/2} \xi} \right) = I_{\ms{F}(E)} - Q_0 \quad \text{ in } \LFE.
 \end{equation*}
\end{lemma}

\begin{proof}
First we note that if $i \in \mb{N}$ then $\sum_{j = 1}^i X_j \otimes R_{i-j}^2 = R_i^2$ by equation (4.7) of \cite{Muhly2016}, so
\begin{equation}
\label{sumstoprojection_e2b}
\sum_{j = 1}^i Z^{(i)}(X_j \otimes R_{i-j}^2)Z^{(i)*} 
= Z^{(i)} R_i^2 Z^{(i)*}
= I_i.  
\end{equation}
Also when $j \in \mb{N}$,  $\sum_{\xi \in A_j} \theta_{X_j^{1/2} \xi} = X_j$ by Lemma \ref{handysums}. 
Thus if $i \geq j \geq 1$,
\begin{equation}
\label{sumstoprojection_e1}
 \sum_{\xi \in A_j} Z^{(i)}(\theta_{X_j^{1/2} \xi} \otimes R_{i-j}^2) Z^{(i)*}
= Z^{(i)}(X_j \otimes R_{i-j}^{2}) Z^{(i)*}.
\end{equation}
For $j \in \mb{N}$ and $\xi \in A_j$, computation shows that $W_{X_j^{1/2} \xi} W_{X_j^{1/2} \xi}^*$ has a diagonal matrix whose $i^{th}$ diagonal entry is   $Z^{(i)} ( \theta_{X_j^{1/2} \xi} \otimes R_{i-j}^2 ) Z^{(i)*}$ when $i \geq j$ and is otherwise zero.
Thus by equation \eqref{sumstoprojection_e1},
\begin{align*}
\sum_{\xi \in A_j} W_{X_j^{1/2} \xi} W_{X_j^{1/2} \xi}^*
& =  \sum_{\xi \in A_j} \left( \sum_{i = j}^\infty v_i Z^{(i)}(\theta_{X_j^{1/2} \xi} \otimes R_{i-j}^2) Z^{(i)*} v_i^* \right) \\
& = \sum_{i = j }^\infty v_i \left(   \sum_{\xi \in A_j} Z^{(i)}(\theta_{X_j^{1/2} \xi} \otimes R_{i-j}^2) Z^{(i)*} \right) v_i^*\\
& =\sum_{i = j }^\infty v_i Z^{(i)}(X_j \otimes R_{i-j}^{2}) Z^{(i)*} v_i^*.
\end{align*} 
Summing over $j$ and using  equation \eqref{sumstoprojection_e2b},
\begin{multline*}
\sum_{j = 1}^\infty  \left( \sum_{\xi \in A_j} W_{X_j^{1/2} \xi} W_{X_j^{1/2} \xi}^* \right)
= \sum_{j = 1}^\infty  \left( \sum_{i = j}^\infty  v_i   Z^{(i)}(X_j \otimes R_{i-j}^{2}) Z^{(i)*} v_i^* \right) \\
= \sum_{i = 1}^\infty v_i \left( \sum_{j=1}^i  Z^{(i)}(X_j \otimes R_{i-j}^{2}) Z^{(i)^*} \right)  v_i^*
= \sum_{i = 1}^\infty v_i  v_i^*
= I_{\ms{F}(E)} - Q_0,
\end{multline*} 
which completes the proof.
\end{proof}

To ``factor'' an element $\xi \in A_k$ when $k \in \mb{N}$ we establish the following notation.   
If $k \in \mb{S}$,  $\xi = \substack{ k \\ \otimes \\ i = 1} \xi_i \in A_k$, and $1 \leq j \leq k$ define
\begin{equation*}
\xi^j  := \substack{ j \\ \otimes \\ i = 1} \xi_i \in \tens{E}{j} \quad  \text{and} \quad
\xi^j_\circ := 
\begin{cases}
\substack{ k \\ \otimes \\ i = j+1} \xi_i  & \text{ if } j < k\\
\langle \xi, \xi \rangle  & \text{ if } j = k
\end{cases} \ \in \tens{E}{k-j}.
\end{equation*}
If $k \notin \mb{S}$, let $\xi^j $ and $\xi^j_\circ$ to be the zero elements in their respective spaces.

\begin{lemma} \label{splitonb} \hfill

\begin{enumerate}[label=(\arabic*),ref=(\arabic*)]
\item \label{splitonb_1} For $k \in \mb{N}_0$ and $ \xi \in A_k$, $\langle \xi, \xi \rangle$ is a projection in $M$ that is zero if and only if $k \notin \mb{S}$.  In any case, $\xi \cdot \langle \xi, \xi \rangle = \xi$. 
\item \label{splitonb_2} For $j \in \mb{N}$, $l \in \mb{N}_0$, and $\xi \in A_{l + j}$,   we have
 $ \xi = \xi^j \otimes \xi^j_\circ$ and
  $\langle \xi^j, \xi^j \rangle \cdot \xi^j_\circ = \xi^j_\circ$.  
If $l + j \in \mb{S}$, then $\xi^j \in A_j$.
\item \label{splitonb_3} If $j \in \mb{N}$, $l \in \mb{N}_0$, and $l + j \in \mb{S}$, then as a disjoint union,
\begin{equation*}
A_{l + j} = \coprod_{\xi \in A_j} \{\eta \in A_{l + j} \mid \eta^j = \xi \}.
\end{equation*}
\item \label{splitonb_4} For $ j \in \mb{N}$ and $\xi \in A_j$,
\begin{equation*}
\varphi_\infty\langle \xi, \xi \rangle  = \sum_{l = 0 }^\infty \left( \sum_{\substack{\eta \in A_{l+j}\\ \eta^j = \xi }} \theta_{ \widehat{\eta^j_\circ} } \right),
\end{equation*}
where any sum taken over an empty set is assumed to be zero.
\end{enumerate}
\end{lemma}

\begin{proof}
The proofs of the first three assertions are straightforward.
To show part \ref{splitonb_4}, let $l \in \mb{N}_0$.  
If $\eta \in A_{l+j}$, then $T_{\xi}^{(l) *}(\eta) = \langle \xi , \eta^j \rangle \cdot \eta^j_\circ$ which, by considering separately the cases when $l+j \notin \mb{S}$ and $l+j \in \mb{S}$, we find to be equal to $\eta^j_\circ$ when $\xi = \eta^j$ and zero when $\xi \neq \eta^j$.
Since $\varphi_{l}\langle \xi, \xi \rangle = T_\xi^{(l)*}T_\xi^{(l)}$, it now follows from Lemma \ref{handysums}\ref{handysums_1} that
\begin{equation*}
v_l (\varphi_l\langle \xi, \xi \rangle) v_l^*
 = v_l \left( \sum_{\eta \in A_{l+j}} \theta_{T_\xi^{(l)*}\eta  }  \right) v_l^*
 =  v_l \left(\sum_{\substack{\eta \in A_{l+j}\\ \eta^j = \xi }} \theta_{\eta^j_\circ} \right) v_l^* = \sum_{\substack{\eta \in A_{l+j}\\ \eta^j = \xi }} \theta_{\widehat{\eta^j_\circ}}.
\end{equation*}
Since $\varphi_\infty \langle \xi, \xi \rangle = diag[\varphi_0 \langle \xi, \xi \rangle, \varphi_1 \langle \xi, \xi \rangle, \varphi_2 \langle \xi, \xi \rangle, \ldots]$,  we obtain part \ref{splitonb_4} by summing over $l \in \mb{N}_0$.
\end{proof}

For the final technical lemma of this section, we  decompose certain induced representation spaces in terms of the family $\{A_k\}_{k=0}^\infty$.  
Suppose that $H$ is a Hilbert space and $\sigma: M \to \ms{B}(H)$ is a faithful, normal, unital $*$-homomorphism.
If $p$ is a projection in $M$, then $pM$ is a Hilbert $W^*$-module over $M$. 
Also, $\sigma(p)(H)$ is a closed subspace of $H$, and it is readily shown that $pM \otimes_\sigma H$ and $\sigma(p)(H)$ are isomorphic Hilbert spaces under the identification of $a \otimes x$ with $\sigma(a)(x)$ for $a \in pM$ and $x \in H$.
If $k \in \mb{N}_0$ and $\xi \in A_k$, then  $\langle \xi, \xi \rangle$ is a projection in $M$; we let $H_\xi $ denote the Hilbert space  $\sigma \langle \xi, \xi \rangle(H)$, and we define
$\mb{H}_k:= 
\sum_{\xi \in A_k} \oplus H_\xi$.
Note that if $k \notin \mb{S}$, then $\mb{H}_k = \{0\}$.

\begin{lemma} \label{hilbertspaceisos} \hfill

\begin{enumerate}[label=(\arabic*),ref=(\arabic*)]
\item \label{hilbertspaceisos_1} For $k \in \mb{N}_0$, there is a Hilbert space isomorphism $\gamma_k: E^{\otimes k} \otimes_{\sigma} H \to \mb{H}_k$ such that for $\eta \in E^{\otimes k}$, $h \in H$, and $\overline{h} = (h_\xi)_{\xi \in A_k} \in \mb{H}_k$,
\begin{equation*}
\gamma_k ( \eta \otimes h )   = \Big( \ \sigma\langle \xi, \eta \rangle (h) \ \Big)_{\xi \in A_k} \quad \text{and} \quad
\gamma_k^* \left( \overline{h} \right)  = \sum_{\xi \in A_k} \xi \otimes h_\xi .
\end{equation*} 
\item \label{hilbertspaceisos_2} Matricially
$\gamma_k (Y \otimes I_{H}) \gamma_k^* = [\sigma\langle \xi, Y \eta \rangle ]_{\xi,\eta \in A_k}$ for $Y \in \ms{L}(E^{\otimes k})$, $k \in \mb{N}_0$.
\item \label{hilbertspaceisos_3} There is an isomorphism of Hilbert spaces $\Gamma: \ms{F}(E) \otimes_\sigma H \to \sum_{k=0}^\infty \oplus \mb{H}_k$ such that for $\eta = (\eta_k)_{k = 0}^\infty$ in $\ms{F}(E)$, $h \in H$, and $\overline{h} = \left( \ (h_{k,\xi})_{\xi \in A_k} \ \right)_{k=0}^\infty \in \sum_{k=0}^\infty \oplus \mb{H}_k$, 
\begin{equation*}
\Gamma ( \eta \otimes h )  = \Big(  \ \gamma_k ( \eta_k \otimes h )  \ \Big)_{k = 0}^\infty \quad \text{and} \quad
\Gamma^* \left(\overline{h} \right)  = \sum_{k = 0 }^\infty \left( \sum_{\xi \in A_k } \widehat{\xi} \otimes h_{k, \xi} \right).
\end{equation*}
\end{enumerate}
\end{lemma}

\begin{proof}
When $k \notin \mb{S}$, we define $\gamma_k$ to be the zero map.
When $k \in \mb{S}$, $A_k$ is an orthonormal basis for $E^{\otimes k}$, so $ \tens{E}{k}$ and $\left( \sum_{\xi \in A_k} \oplus \ \langle \xi, \xi \rangle M \right)$ are isomorphic Hilbert $W^*$-modules (\cite{Paschke1973}, proof of Theorem 3.12).  
Using the discussion preceding the lemma, parts \ref{hilbertspaceisos_1} and \ref{hilbertspaceisos_3} readily follow. 
Part \ref{hilbertspaceisos_2} follows from routine computation.
\end{proof}

\section{Weighted Commutant Lifting} \label{wclChap}

Commutant lifting is the principle tool used to obtain the Nevan\-linna-Pick interpolation result, Theorem 5.3, in \cite{Muhly2004a}.  
Towards obtaining our \emph{weighted} interpolation theorem in Section \ref{wnpChap}, the purpose of the present section is to show that commutant lifting can be done in the weighted case.
Since the weights present obstacles for the method producing the (unweighted) commutant lifting theorem, Theorem 4.4 in \cite{Muhly1998a}, our approach is instead inspired by Arias' proof of Theorem 3.1 in \cite{Arias2004}, a lifting result in complex $n$-space that ultimately makes use of Parrott's lemma, Theorem 1 of \cite{Parrott1978}.

\begin{theorem}[Weighted Commutant Lifting]\label{WCL_theorem}
Let $\sigma: M \to \ms{B}(H)$ be a faithful, normal, unital, $*$-homomorphism for a Hilbert space $H$, and let $Z$ be a  sequence of weights associated with $X$.  
Suppose that $J$ is a  closed linear subspace of $\ms{F}(E) \otimes_{\sigma} H$ such that for every $Y \in H^\infty(E,Z)$,
$(Y^* \otimes I_{H}) (J)  \subseteq J$.
With $V$ the inclusion map of $J$ into $\ms{F}(E) \otimes_{\sigma} H$, suppose there exists $G \in \ms{B}(J)$ such that for every $Y \in H^\infty(E,Z)$,
$G\left(  V^* (Y \otimes I_{H}) V \right)   =  \left( V^* (Y \otimes I_{H}) V \right) G$.
Then there exists $\widetilde{G} \in \ms{B}(\ms{F}(E) \otimes_{\sigma} H)$ such that 
\begin{enumerate}[label=(\arabic*),ref=(\arabic*)]
\item \label{WCL_theorem_i1} $\widetilde{G}^* (J) \subseteq J$,
\item \label{WCL_theorem_i2} $ V^* \widetilde{G} V = G$,
\item \label{WCL_theorem_i3} $\widetilde{G} (Y \otimes I_{H})  =  (Y \otimes I_{H}) \widetilde{G}$ for all $Y \in H^\infty(E,Z)$, and
\item \label{WCL_theorem_i4} $\Vert \widetilde{G} \Vert = \Vert G \Vert$.
\end{enumerate}
\end{theorem}

\begin{proof}
First, we consider the case when $\Vert G \Vert = 1$.
Let $K = \ms{F}(E) \otimes_{\sigma} H$, and let $P \in \ms{B}(K)$ be the projection map onto $J$.   
Note that $V^*$ is the range restriction of $P$ to $J$, $V^*V$ is the identity map on $J$, and $VV^* = P$.
For $n \in \mb{N}_0$, let $K_n$ be the isomorphic image of $\left( \sum_{j = 0}^n  \oplus   \tens{E}{j} \right) \otimes_\sigma H$ in $K$. 
For technical purposes, let $K_{-1}$ be the zero subspace of $K$. 

Our construction of $\widetilde{G} \in \ms{B}(K)$ will utilize an ascending sequence of subspaces $\{J_i\}$ of $K$ and operators $\{G_i\}$ such that each $G_i \in \ms{B}(K, J_i)$.
We construct these sequences inductively.
To begin the process, let $n_1 = -1$, let $J_1 = J$, and define $G_1 : = G   V^*$.   
To organize the numerous properties we wish to maintain at each step of the inductive argument,  we make the following definition.  
If $m \in \mb{N}$, we say $\left(\{n_i\}_{i=1}^m , \{J_i\}_{i = 1}^m,  \{G_i\}_{i = 1}^m \right)$ is an ``$m$-triple'' if for each $i$ with $1 \leq i \leq m$ we have that  $n_i \in \mb{Z}$, $J_i$ is a closed subspace of $K$, $G_i \in \ms{B}(K, J_i)$, and the following conditions hold in all cases: with $V_i$  the inclusion map of $J_i$ into $K$,
\begin{itemize}
\item[]Condition $(1,m)$: \quad$n_1 < n_2 <  \cdots < n_m$, \quad (if $m \geq 2$);
\item[]Condition $(2,m)$:  \quad$J_1 \subsetneq J_2 \subsetneq \cdots \subsetneq J_m$, \quad (if $m \geq 2$);
\item[]Condition $(3,m)$: \quad$K_{n_i} \subseteq J_i$, \quad$1 \leq i \leq m$;
\item[]Condition $(4,m)$:  \quad$(Y^* \otimes I_H)(J_i) \subseteq J_i$, \quad$ 1 \leq i \leq m$, $Y \in H^\infty(E,Z)$;
\item[]Condition $(5,m)$: \quad$( V_{i}^* (Y \otimes I_H) V_{i} ) G_i = G_i  (Y \otimes I_H)$, \quad $ 1 \leq i \leq m$, $Y \in H^\infty(E,Z)$;
\item[]Condition $(6,m)$: \quad$V_i^*V_j G_j = G_i$, \quad $1 \leq i \leq j \leq m$;
\item[]Condition $(7,m)$: \quad $\Vert G_i \Vert = 1$, \quad $ 1 \leq i \leq m$.
\end{itemize}
Our hypotheses guarantee that $\left(\{n_i\}_{i=1}^1, \{J_i\}_{i = 1}^1,  \{G_i\}_{i = 1}^1\right)$ is a $1$-triple.  

Let us show that if there exists an $m$-triple $\left(\{n_i\}_{i=1}^m , \{J_i\}_{i = 1}^m,  \{G_i\}_{i = 1}^m \right)$ with $m \in \mb{N}$ such that $J_m = K$, then to satisfy the conclusion of our theorem, we may take $\widetilde{G}$ to be $G_m$.  
Here, $G_m \in \ms{B}(K)$, $V_m$ is the identity map on $K$, and $V_1 = V$.  
By Condition $(6,m)$, $ V^* G_m = V_1^* V_m  G_m = G_1 = G V^*$.  
Thus $P G_m = V V^* G_m = V G V^*$ and
$P G_m P =  ( V G V^* )P =  V G V^* = P G_m$.
It follows that $G_m$ satisfies property \ref{WCL_theorem_i1} in the conclusion of the theorem.
Using again the fact that $ V^* G_m = GV^*$, we have $V^* G_m  V=  G V^* V = G$, so property \ref{WCL_theorem_i2} holds.  
Condition $(5,m)$ gives us property \ref{WCL_theorem_i3}, and property \ref{WCL_theorem_i4} follows from Condition $(7,m)$.

We are left to consider the case when there is no $m \in \mb{N}$ such that $J_m = K$ for some $m$-triple $\left(\{n_i\}_{i=1}^m, \{J_i\}_{i = 1}^m,  \{G_i\}_{i = 1}^m \right)$.  
Proving the following fact will occupy us for quite some time:
if $m \in \mb{N}$ and $\left(\{n_i\}_{i=1}^m, \{J_i\}_{i = 1}^m,  \{G_i\}_{i = 1}^m \right)$ is an $m$-triple such that $J_m \subsetneq K$, then there exist $n_{m+1}$, $J_{m+1}$, and $G_{m+1}$ such that  $\left(\{n_i\}_{i=1}^{m+1}, \{J_i\}_{i = 1}^{m+1}, \{G_i\}_{i = 1}^{m+1}\right)$ is an $(m+1)$-triple.  

We begin by defining $n_{m+1}$ and $J_{m+1}$.
Since $K$ is the norm-closure of $\bigcup_{j=0}^\infty K_j$ and $J_m \neq K$, the set $\{ j \in \mb{N}_0 \mid  K_j \nsubseteq J_m\}$ is nonempty, so we define $n_{m+1}$ to be its least element.
Define $J_{m+1}$ to be the closed subspace of $K$ that is generated by $J_m \cup K_{n_{m+1}}$.
It is readily shown that the first three conditions for an $(m+1)$-triple are satisfied.    
To show Condition $(4,m+1)$, we first prove that for any $j \in \mb{N}$ and $\xi \in \tens{E}{j}$,
\begin{align}
\label{cond5_projections_a}
P_{m+1}  (\wtens{\xi})  P_{m} 
& = P_{m+1}  (\wtens{\xi}) \\
\label{cond5_projections_b}
& = P_{m+1} (\wtens{\xi}) P_{m+1},
\end{align}
where $ P_{m}$ and $P_{m+1}$ are the projection maps  in $\ms{B}(K)$ onto $J_m$ and $J_{m+1}$, respectively.
First, let $j = 1$ and $\xi \in E$.  
By Condition $(4,m)$, $(\wstartens{\xi}) (J_m) \subseteq J_m$.
At the same time, from the definition of $W_\xi$ it follows that for any $n \in \mb{N}_0$, 
$(\wstartens{\xi}) (K_{n}) \subseteq K_{n-1}$.  
Therefore, by definition of $n_{m+1}$, $(\wstartens{\xi}) (K_{n_{m+1}}) \subseteq K_{(n_{m+1})-1} \subseteq J_m$,
so $(\wstartens{\xi}) (J_{m+1}) \subseteq J_m$.  
Thus, equation \eqref{cond5_projections_a} holds when $j = 1$.
Since $P_{m}P_{m+1} = P_{m}$ and $E^{\otimes k+1}$ is generated by elements of the form  $\xi = \eta \otimes \zeta $ where $\eta \in E$ and $\zeta \in E^{\otimes k}$, an inductive argument gives equation \eqref{cond5_projections_a} in the general case.   
Equation \eqref{cond5_projections_b} follows since $P_{m}P_{m+1} = P_{m}$. 
For $a \in M$, it is readily seen that 
$\phitens{a^*}$ leaves $K_n$ invariant for any $n \in \mb{N}$ and leaves $J_m$ invariant by Condition $(4,m)$.
It follows that for $Y = \varphi_\infty(a)$,
\begin{equation}
\label{4_m+1.1}
(Y^* \otimes I_H)(J_{m+1}) \subseteq J_{m+1}.
\end{equation}
By equation \eqref{cond5_projections_b}, containment \eqref{4_m+1.1} is also satisfied for $Y = W_\xi$ when $\xi \in \tens{E}{j}$ for some $j \in \mb{N}$. 
It follows that containment \eqref{4_m+1.1} holds for every $Y \in H^\infty(E,Z)$.  
This, together with Condition $(4,m)$, implies Condition $(4,m+1)$.

Before considering the remaining conditions we must define $G_{m+1}$, which will require some preliminary work.  Let us begin by recalling and establishing notation.
Let $\{A_k\}_{k = 0}^\infty$ be the family defined in the paragraph before Lemma \ref{handysums}.
For $k \in \mb{N}_0$ and $\xi \in \tens{E}{k}$, $\widehat{\xi}$ denotes $v_k(\xi) = (\delta_{i = k} \ \xi)_{i = 0}^\infty \in \ms{F}(E)$.
For $\xi \in \ms{F}(E)$,  $L_\xi $ is the insertion operator mapping $h \in H$ to $\xi \otimes h$.
Define $g : = G_m  L_{\widehat{1_M}} \in \ms{B}( H , J_m)$.   
Note that $V_m^*$ is the range restriction of $P_m$, $V_m^*V_m$ is the identity on $J_m$, and $V_mV_m^* = P_m$, with analogous properties for $V_{m+1}$.
For notational convenience, define  $\alpha: H^{\infty} (E, Z) \to \ms{B}(J_m)$ and $\beta: H^{\infty} (E, Z) \to \ms{B}(J_m, J_{m+1})$ at $Y \in H^\infty(E,Z) $ by,
\begin{equation*}
\alpha(Y)   = V_{m}^* (Y \otimes I_{H})  V_{m} \quad \text{and} \quad
\beta(Y)  = V_{m+1}^*  (Y \otimes I_{H})  V_{m}.
\end{equation*}
The following technical lemma collects facts about $\alpha$ and $\beta$ that will be useful in future computations.

\begin{lemma}\label{g,A,B_step} \hfill
\begin{enumerate}[label=(\arabic*),ref=(\arabic*)]
\item \label{g,A,B_step_2} 
$\alpha$ is contractive, linear, multiplicative, unital, ultraweakly continuous, and for any $Y \in H^\infty(E, Z)$, 
\begin{equation*} \alpha(Y)  G_m = G_m(Y \otimes I_{H}).
\end{equation*}
\item \label{g,A,B_step_3}
$\beta$ is contractive, linear, continuous with respect to the ultraweak topology on $H^\infty(E,Z)$ and the  weak operator topology on $\ms{B}(J_m, J_{m+1})$, and for any $k \in \mb{N}, \xi \in E^{\otimes k}$, and $Y \in H^{\infty} (E, Z)$,
\begin{equation*} \beta(W_\xi)  \alpha(Y) = \beta(W_\xi Y).
\end{equation*}
\item \label{g,A,B_step_4}
For any $a \in M, k \in \mb{N}_0$, and $\xi \in E^{\otimes k}$, 
$\beta(W_{\xi})^* V_{m+1}^* (\phitens{a}) V_{m+1}  = \beta(W_{a^* \cdot \xi})^*$.
\item \label{g,A,B_step_5b} 
For all  $k \in \mb{N}_0$ and $\xi \in E^{\otimes k}$, \quad
$\alpha(W_{\mcZ{k} \xi})  g = G_m  L_{\widehat{\xi}}$,
\item \label{g,A,B_step_5d}
For all  $a \in M$, $k \in \mb{N}_0$, and $\xi \in E^{\otimes k}$, \
$ \alpha(W_{\xi \cdot a})g  =  \alpha(W_\xi) g \sigma (a)$.
\item \label{g,A,B_step_6}
For all $a \in M$, $k \in \mb{N}_0$, and $\xi \in E^{\otimes k}$, \qquad $g^*  \beta(W_{\xi \cdot a})^*  = \sigma(a^*)  g^*  \beta(W_\xi)^*$.
\item \label{g,A,B_step_7}
 For any $j \in \mb{N}_0$, $\eta \in E^{\otimes j}$, and $T \in \ms{L}(E^{\otimes j})$,
\begin{equation*}
\alpha(W_{T\eta})g  = \sum_{\xi \in A_j} \alpha(W_{T \xi \langle \xi, \eta \rangle}) g,
\end{equation*}
with convergence in the weak operator topology on  $\ms{B}(H, J_m)$.
\item \label{g,A,B_step_8}
For any $j \in \mb{N}_0$, $\eta \in E^{\otimes j}$, and $T \in \ms{L}(E^{\otimes j})$,
\begin{equation*}
 g^* \beta(W_{T\eta})^*  = \sum_{\xi \in A_j} g^* \beta(W_{T \xi \langle \xi, \eta \rangle})^*,
\end{equation*} 
with convergence in the weak operator topology on $\ms{B}(J_{m+1}, H)$.
\end{enumerate}
\end{lemma}

\begin{subproof}[Proof of Lemma \ref{g,A,B_step}] \hfill

\ref{g,A,B_step_2}:
The induced representation $\sigma^{\ms{F}(E)}$ is a normal, unital $*$-homomor\-phism.  
We form $\alpha$ by composing the restriction of  $\sigma^{\ms{F}(E)}$ to $H^\infty(E,Z)$ with $Ad(V_m^*):\ms{B}(K) \to \ms{B}(J_m)$, which sends $T$ to $ V_m^* T V_m$.  
It follows that $\alpha$ is  linear, unital, contractive, and ultraweakly continuous.
By Condition $(4,m)$, for $Y \in H^\infty(E,Z)$, $V^*_{m} (Y \otimes I_H) =  V_{m}^* (Y \otimes I_H)  P_{m}$.  
It follows that  $\alpha$ is multiplicative.
By Condition $(5,m)$, for any $Y \in H^\infty(E, Z)$, $\alpha(Y) G_m = G_m (Y \otimes I_{H})$.

\ref{g,A,B_step_3}: The linearity, contractivity, and continuity of $\beta$ follow from arguments similar to those used above for $\alpha$. 
By equation \eqref{cond5_projections_a}, if $k \in \mb{N}$, $\xi \in \tens{E}{k}$, and $Y \in H^{\infty} (E, Z)$, then 
$\beta(W_\xi)  \alpha(Y) 
 = V_{m+1}^* (W_\xi \otimes I_H) P_{m} (Y \otimes I_H) V_{m} 
  = V_{m+1}^* (W_\xi \otimes I_H) (Y \otimes I_H) V_{m} 
   = \beta(W_\xi Y)$.

\ref{g,A,B_step_4}: 
Since Condition $(4,m+1)$ holds and $\phiinf{a^*} W_\xi = W_{a^* \cdot \xi}$, 
\begin{multline*}
 V_{m+1}^* (\phitens{a^*}) V_{m+1} \beta(W_{\xi}) 
  = V_{m+1}^* (\phitens{a^*}) P_{m+1} (\wtens{\xi}) V_{m}  \\
  = V_{m+1}^* (\phitens{a^*}) (\wtens{\xi}) V_{m} 
  = \beta(W_{a^* \cdot \xi}).
\end{multline*}
The result follows from taking the adjoint.

\ref{g,A,B_step_5b}:
By part \ref{g,A,B_step_2}, 
$\alpha(W_{\mcZ{k} \xi}) g 
 = G_m (\wtens{\mcZ{k} \xi}) L_{\widehat{1_M}}   
  = G_m L_{\widehat{\xi}}$.

\ref{g,A,B_step_5d}: Since $\varphi(a) = W_a  $ and $Z^{(0)} = I_M$, by part  \ref{g,A,B_step_5b}
\begin{equation}
\label{g,A,B_step_5a2}
\alpha(\phiinf{a}) g 
 = \alpha(W_a) g  = G_m L_{\widehat{a}}  = G_m L_{\widehat{1_M}}\sigma(a)  = g \sigma(a).
\end{equation}
Since $\alpha$ is multiplicative,
$\alpha(W_{\xi \cdot a}) g 
= \alpha(W_\xi) \alpha(\phiinf{a}) g
= \alpha(W_\xi) g \sigma(a)$.

\ref{g,A,B_step_6}: By Condition $(4,m)$ and equation \eqref{g,A,B_step_5a2},   
\begin{multline*}
\beta(W_{\xi \cdot a})  g  
=V_{m+1}^*(W_\xi \otimes I_H)(\varphi_\infty(a) \otimes I_H) V_m g \\
=V_{m+1}^*(W_\xi \otimes I_H)P_m(\varphi_\infty(a) \otimes I_H) V_m g \\
= \beta(W_\xi) \alpha(\phiinf{a}) g
= \beta(W_\xi) g \sigma(a).
\end{multline*}  
The result follows by taking the adjoint.

\ref{g,A,B_step_7}: We have $\eta = \sum_{\xi \in A_j } \xi \langle \xi, \eta \rangle$ since $A_j$ is either an orthonormal basis or the zero set. 
As an element of $\ms{L}(\tens{E}{j})$, $T$ is an ultraweakly continuous, right $M$-module homomorphism;
also the map that sends $\zeta \in \tens{E}{j}$ to $W_\zeta \in \LFE$ is linear and ultraweakly continuous, as is  $\alpha$.  
The result follows.

\ref{g,A,B_step_8}:  The result follows along similar lines as part \ref{g,A,B_step_7}.
\end{subproof}

One fact we will require in defining $G_{m+1}$  is that
\begin{equation} 
\label{Gm*_step}
G_m^*(x)  = \sum_{k = 0}^\infty \left( \sum_{\xi \in A_k}  \widehat{\xi} \otimes g^* \alpha \left(W_{\mcZ{k} \xi} \right)^*x \right), \quad  x \in J_{m}.
\end{equation}
To prove this fact, we will show that
\begin{equation}
\label{Gm*_step.2}
\left( \ \left( \ g^* \alpha \left(W_{\mcZ{k} \xi} \right)^*x \ \right)_{\xi \in A_k} \ \right)_{k = 0}^\infty
\end{equation}
 is an element in $\sum_{k=0}^\infty \oplus \mb{H}_k$, using the notation of Lemma \ref{hilbertspaceisos}.  
Let $k \in \mb{N}_0$ and $\xi \in A_k$.  
By Lemma \ref{splitonb}\ref{splitonb_1}, $\xi \cdot \langle \xi, \xi \rangle = \xi $, so by Lemma \ref{g,A,B_step}\ref{g,A,B_step_5d}, we have
$g^* \alpha\left(W_{\mcZ{k} \xi}\right)^*x = \sigma\langle \xi, \xi \rangle g^* \alpha\left(W_{\mcZ{k} \xi}\right)^*x \in H_\xi$. 
By  Lemma \ref{g,A,B_step}\ref{g,A,B_step_5b} and Lemma \ref{handysums}\ref{handysums_2},
\begin{multline} \label{Gm*_step.1}
\sum_{\xi \in A_k} \left\Vert g^* \alpha\left(W_{\mcZ{k} \xi }\right)^* x \right\Vert^2 
 = \sum_{\xi \in A_k} \left\Vert L_{\widehat{\xi}}^* G_m^* x \right\Vert^2 \\
 = \sum_{\xi \in A_k} \left\langle  L_{\widehat{\xi}} L_{\widehat{\xi}}^* G_m^* x, G_m^* x \right\rangle \\
 = \langle ( Q_k \otimes I_H) G_m^* x, G_m^* x \rangle
 \leq \Vert G_m^* x \Vert^2 .
\end{multline}
We conclude that for every $k \in \mb{N}_0$, $\left( \ g^* \alpha\left(W_{\mcZ{k} \xi }\right)^*x \ \right)_{\xi \in A_k}$ is a well-defined element of $\mb{H}_k$. 
Moreover,  $\sum_{k = 0 }^\infty (Q_k \otimes I_H )= I_{K}$ with strong operator convergence in $\ms{B}(K)$, so by  computation \eqref{Gm*_step.1} we find,
\begin{multline*} 
\sum_{k = 0}^\infty \left\Vert \big( \ g^* \alpha\left(W_{\mcZ{k} \xi}\right)^*x \ \big)_{\xi \in A_k} \right\Vert^2
= \sum_{k = 0}^\infty \langle (Q_k \otimes I_H) G_m^* x, G_m^* x \rangle \\
= \Vert G_m^* x \Vert ^2.
\end{multline*} 
Thus, the element in expression \eqref{Gm*_step.2} belongs to $\sum_{k=0}^\infty \oplus \mb{H}_k$.  
By Lemma \ref{hilbertspaceisos}\ref{hilbertspaceisos_3}, its image under $\Gamma^*$ in $K$ is the element in the right-hand side of equation \eqref{Gm*_step}.  
To show equality with $G_m^*(x)$, let $h \in H$, $j \in \mb{N}_0$, and $ \eta \in E^{\otimes j}$.  
For $k \in \mb{N}_0$ and $\xi \in \tens{E}{k}$, $\left\langle \widehat{\xi}, \widehat{\eta} \right\rangle$ equals $\left\langle \xi, \eta \right\rangle$ when $k = j$ and is otherwise zero.  
By Lemma \ref{g,A,B_step} parts \ref{g,A,B_step_5d}, \ref{g,A,B_step_7}, and \ref{g,A,B_step_5b},
\begin{multline*}
 \left\langle \sum_{k = 0 }^\infty \left(\sum_{\xi \in A_k }  \widehat{\xi} \otimes g^* \alpha\left(W_{\mcZ{k} \xi}\right)^*x \right), \widehat{\eta} \otimes h \right\rangle \\
 = \sum_{\xi \in A_j }  \left\langle x,  \alpha\left(W_{\mcZ{j} \xi}\right) g \sigma \langle \xi, \eta \rangle   h \right\rangle 
 = \sum_{\xi \in A_j }
\Big\langle  x  ,  \alpha\left(W_{\mcZ{j} \xi  \langle \xi , \eta \rangle}\right) g ( h)  \Big\rangle \\
=  \Big\langle  x  ,  \alpha\left(W_{ \mcZ{j}  \eta }\right) g ( h)  \Big\rangle 
=  \langle  x  ,  G_m L_{\widehat{\eta}} ( h)  \rangle
=  \langle G_m^* x  ,  \widehat{\eta} \otimes  h  \rangle.
\end{multline*} 
Equation \eqref{Gm*_step} follows from routine approximation arguments.

Another fact we need before constructing $G_{m+1}$ is that
\begin{equation}
\label{givesFcontractive}
 \sum_{k=1}^\infty \left(\sum_{\eta \in A_k} \beta\left(W_{\mcZ{k} \eta}\right) gg^* \beta\left(W_{\mcZ{k} \eta}\right)^* \right) \leq  I_{J_{m+1}}.
\end{equation}
For convenience, we break the proof into three main steps, as follows: 
\begin{align}
\label{Step1}
I_{J_{m+1}} 
& \geq  \qquad \sum_{j=1}^\infty  \left( \sum_{\xi \in A_j}   \beta\left(W_{X_j^{1/2} \xi}\right) G_m G_m^* \beta\left(W_{X_j^{1/2} \xi}\right)^* \right) \\*
\label{Step2}
\begin{split}
& = \qquad \sum_{k=1}^\infty \left( \sum_{j=1}^k \left( \sum_{\eta \in A_k} \beta\left(W_{\left(X_j^{1/2} \otimes \mcZ{k-j}\right) \eta}\right) g \right.\right. \\*
& \qquad \qquad \qquad \qquad \qquad \left. \left. \cdot g^* \beta\left(W_{\left(X_j^{1/2} \otimes \mcZ{k-j}\right) \eta}\right)^* \right) \right) 
\end{split} \\*
\label{Step3}
& = \qquad \sum_{k=1}^\infty \left( \sum_{\eta \in A_k} \beta\left(W_{\mcZ{k} \eta}\right) gg^* \beta\left(W_{\mcZ{k} \eta}\right)^* \right).
\end{align}

We begin with inequality \eqref{Step1}.  
It follows from Lemma \ref{sumstoprojection} that for any $j \in \mb{N}$, 
$\sum_{\xi \in A_j} W_{X_j^{1/2} \xi}W_{X_j^{1/2} \xi}^*$ converges in $\LFE$.
If $\mb{F}$ is a finite subset of $A_j$, then because $\sigma^{\ms{F}(E)}$ is a $*$-homomorphism and  $\Vert V_{m} G_m \Vert \leq 1$, 
\begin{align*} 
& \sum_{\xi \in \mb{F}}  \beta\left(W_{X_j^{1/2} \xi} \right) G_m G_m^* \beta\left(W_{X_j^{1/2} \xi} \right)^* \\
&  \qquad = \sum_{\xi \in \mb{F}} V_{m+1}^*  \left(\wtens{X_j^{1/2} \xi} \right) V_{m} G_m G_m^* V_{m}^* \left(\wstartens{X_j^{1/2} \xi} \right) V_{m+1} \\
&  \qquad \leq  V_{m+1}^* \left( \left( \sum_{\xi \in \mb{F}} W_{X_j^{1/2} \xi} W_{X_j^{1/2} \xi}^* \right) \otimes I_H \right) V_{m+1} \\
& \qquad \leq V_{m+1}^* \left( \left( \sum_{\xi \in A_j} W_{X_j^{1/2} \xi} W_{X_j^{1/2} \xi}^* \right) \otimes I_H \right) V_{m+1}.
\end{align*} 
It follows that $\sum_{\xi \in A_j}  \beta\left(W_{X_j^{1/2} \xi}\right) G_m G_m^* \beta\left(W_{X_j^{1/2} \xi} \right)^*$ converges ultraweakly in $\ms{B}(J_{m+1})$  and
$V_{m+1}^* \left( \left( \sum_{\xi \in A_j} W_{X_j^{1/2} \xi} W_{X_j^{1/2} \xi}^* \right) \otimes I_H \right) V_{m+1}$ is an upper bound.
This, together with Lemma \ref{sumstoprojection},  implies that for $N \in \mb{N}$,
\begin{multline*}
\sum_{j = 1}^N \left( \sum_{\xi \in A_j}  \beta\left(W_{X_j^{1/2} \xi} \right) G_m G_m^* \beta\left(W_{X_j^{1/2} \xi} \right)^* \right) \\
\leq  V_{m+1}^* \left( \sum_{j = 1}^N  \left(\sum_{\xi \in A_j} W_{X_j^{1/2} \xi} W_{X_j^{1/2} \xi}^* \right) \otimes I_H \right) V_{m+1}  
\leq  I_{J_{m+1}}.
\end{multline*} 
Inequality \eqref{Step1} follows.
Towards proving equation \eqref{Step2}, fix $j \in \mb{N}$ and $\xi \in A_j$.   
Since $\xi \cdot \langle \xi, \xi \rangle = \xi$, it follows from Lemma \ref{g,A,B_step} parts \ref{g,A,B_step_2} and \ref{g,A,B_step_3} that
\begin{equation}
\label{halfway.1}
\beta\left( W_{X_j^{1/2} \xi} \right) G_m= \beta\left( W_{X_j^{1/2} \xi}  \right) G_m  \left( \varphi_\infty \langle \xi, \xi \rangle \otimes I_H \right).
\end{equation}
Also, by Lemma \ref{g,A,B_step}\ref{g,A,B_step_5b}, whenever $l \in \mb{N}_0$ and $\zeta \in \tens{E}{l}$,
\begin{equation*}
G_m\left( \theta_{ \widehat{\zeta}} \otimes I_H\right) G_m^* 
=  G_m L_{ \widehat{\zeta}} L_{\widehat{\zeta}}^* G_m^*
=  \alpha\left(W_{\mcZ{l} \zeta}\right) gg^* \alpha\left(W_{\mcZ{l} \zeta}\right)^*.
\end{equation*}
Therefore, by Lemma \ref{splitonb}\ref{splitonb_4},
\begin{multline}
\label{Gm_phiinf_Gm*_step}
G_m  \left(\varphi_\infty \langle \xi, \xi \rangle \otimes I_{H} \right) G_m^*
  =  \sum_{l = 0 }^\infty \left( \sum_{\substack{\eta \in A_{l+j}\\ \eta^j = \xi }}  G_m\left( \theta_{ \widehat{\eta^j_0}} \otimes I_H\right) G_m^* \right)  \\
 =  \sum_{l = 0 }^\infty \left( \sum_{\substack{\eta \in A_{l+j}\\ \eta^j = \xi }}  \alpha\left(W_{\mcZ{l} \eta^j_0}\right) gg^* \alpha\left(W_{\mcZ{l} \eta^j_0}\right)^* \right).
\end{multline}
If $l \in \mb{N}_0$ and $\eta \in A_{l + j}$ with $\eta^j = \xi$, by Lemma \ref{g,A,B_step}\ref{g,A,B_step_3} and Lemma \ref{splitonb}\ref{splitonb_2}, 
\begin{equation}
\label{halfway.2}
\beta\left( W_{X_j^{1/2} \xi}  \right)  \alpha \left( W_{\mcZ{l} \eta^j_0} \right) g = \beta\left( W_{ \left( X_j^{1/2} \otimes \mcZ{l} \right) \eta} \right) g.
\end{equation}
Putting together equations \eqref{halfway.1}, \eqref{Gm_phiinf_Gm*_step}, and \eqref{halfway.2},
\small
\begin{multline*}
\beta\left( W_{X_j^{1/2} \xi} \right) G_m G_m^* \beta\left( W_{X_j^{1/2} \xi}\right)^* \\
= \beta\left( W_{X_j^{1/2} \xi} \right) G_m\left( \varphi_\infty \langle \xi, \xi \rangle \otimes I_H \right)  G_m^* \beta\left( W_{X_j^{1/2} \xi}\right)^* \\
= \sum_{l = 0 }^\infty \left(
\sum_{\substack{\eta \in A_{l+j}\\ \eta^j = \xi}}  \beta\left( W_{X_j^{1/2} \xi}  \right)  \alpha \left( W_{\mcZ{l} \eta^j_0} \right) gg^* \alpha \left( W_{\mcZ{l} \eta^j_0} \right)^*  \beta\left( W_{X_j^{1/2} \xi} \right)^* \right) \\
= \sum_{l = 0 }^\infty \left(
\sum_{\substack{\eta \in A_{l+j}\\ \eta^j = \xi}}  \beta\left( W_{ \left( X_j^{1/2} \otimes \mcZ{l} \right) \eta} \right) gg^*  \beta\left( W_{ \left( X_j^{1/2} \otimes \mcZ{l} \right) \eta} \right)^* \right).
\end{multline*}
\normalsize
Summing over $\xi \in A_j$,  interchanging sums, and  using Lemma \ref{splitonb}\ref{splitonb_3},
\small
\begin{align*}
&  \sum_{\xi \in A_j}   \beta\left( W_{X_j^{1/2} \xi} \right)  G_m G_m^* \beta\left( W_{X_j^{1/2} \xi} \right) ^*  \\
&  = \sum_{\xi \in A_j} \left( \sum_{l = 0 }^\infty \left( \sum_{\substack{\eta \in A_{l+j}\\ \eta^j =\xi} } \beta\left( W_{ \left( X_j^{1/2} \otimes \mcZ{l} \right)  \eta} \right)  gg^*  \beta\left( W_{ \left( X_j^{1/2} \otimes \mcZ{l} \right)  \eta} \right) ^* \right) \right)\\
& =  \sum_{l=0}^\infty \left( \sum_{\eta \in A_{l+j} }  \beta\left( W_{ \left( X_j^{1/2} \otimes \mcZ{l} \right)  \eta} \right)  gg^*  \beta\left( W_{ \left( X_j^{1/2} \otimes \mcZ{l} \right)  \eta} \right) ^* \right). 
\end{align*}
\normalsize
Finally, summing over $j \in \mb{N}$, re-indexing, and interchanging sums,
\small
\begin{align*}
& \sum_{j=1}^\infty  \left( \sum_{\xi \in A_j}   \beta\left( W_{X_j^{1/2} \xi} \right)  G_m G_m^* \beta\left( W_{X_j^{1/2} \xi} \right) ^* \right) \\
& = \sum_{j=1}^\infty \left(   \sum_{l=0}^\infty \left( \sum_{\eta \in A_{l+j} }  \beta\left( W_{ \left( X_j^{1/2} \otimes \mcZ{l} \right)  \eta} \right)  gg^*  \beta\left( W_{ \left( X_j^{1/2} \otimes \mcZ{l} \right)  \eta} \right) ^* \right) \right) \\
&   =  \sum_{k = 1 }^\infty \left( \sum_{j=1}^k \left( \sum_{\eta \in A_{k}}  \beta\left( W_{ \left( X_j^{1/2} \otimes \mcZ{k-j} \right)  \eta} \right)  gg^*  \beta\left( W_{ \left( X_j^{1/2} \otimes \mcZ{k-j} \right)  \eta} \right) ^*  \right) \right)
\end{align*}
\normalsize
as desired.  
Finally, to prove equation \eqref{Step3} it suffices to show that 
\begin{align}
\notag
& \sum_{j=1}^k \left( \sum_{\eta \in A_k}  \beta\left(W_{\left(X_j^{1/2} \otimes \mcZ{k-j}\right) \eta} \right) gg^* \beta\left(W_{\left(X_j^{1/2} \otimes \mcZ{k-j}\right) \eta} \right)^* \right)\\
\label{twothirds.1} 
&  \qquad = \sum_{\eta \in A_k}   \beta\left(W_{\mcZ{k} \eta}\right) gg^* \beta\left(W_{\mcZ{k} \eta}\right)^*
\end{align}
for any $k \in \mb{N}$.
Let $x \in J_{m+1}$.
By definition of $D_k$ and Lemma \ref{handysums}\ref{handysums_4}, 
$D_k D_k^* 
= D_k \left( \sum_{\xi \in A_k} T_{\xi} T_{\xi}^* \right) D_k^* 
= \sum_{\xi \in A_k} W_{\xi} W_{\xi}^*$.
Since $\Vert g \Vert \leq 1$, 
\begin{multline*}
\sum_{\xi \in A_k} \Vert g^* \beta(W_{ \xi})^*x \Vert^2  =
\sum_{\xi \in A_k} \Vert g^* V_{m}^* (\wstartens{\xi}) V_{m+1} x \Vert^2 \\
\leq \sum_{\xi \in A_k} \Vert  (\wstartens{\xi})  x \Vert^2 
= \sum_{\xi \in A_k} \left\langle ( W_\xi W_\xi^* \otimes I_H) x, x \right\rangle  \\
= \left\langle ( D_k D_k^* \otimes I_H) x, x \right\rangle 
= \left\Vert ( D_k^* \otimes I_H) x \right\Vert^2.
\end{multline*} 
Using the notation and results from Lemma \ref{hilbertspaceisos}, it follows from Lemma \ref{g,A,B_step}\ref{g,A,B_step_6} that for every $\xi \in A_k$,  $g^* \beta(W_{ \xi})^*x$ belongs to $H_\xi $.  
The preceding computation shows that $( g^* \beta(W_{\xi})^* x)_{\xi \in A_k}$ is an element in $\mb{H}_k$; let $\widetilde{x}$ be its image in $\tens{E}{k} \otimes_\sigma H$ under the isomorphism $\gamma_k^*$.
If $S$ is any element in $\ms{L}(\tens{E}{k})$, then using Lemma \ref{handysums}\ref{handysums_1} and applying $\sigma^{\tens{E}{k}}$, 
we have $SS^* \otimes I_H =  \sum_{\xi  \in A_k } \left( \theta_{S\xi} \otimes I_H \right)$.
For each $\xi \in A_k$, the matrix for $\theta_{S\xi} \otimes I_H$, when viewed as an operator in $\ms{B}(\mb{H}_k)$, is $\left[ \sigma\langle \mu, S\xi  \rangle \sigma \langle S\xi , \nu \rangle \right]_{\mu,\nu \in A_k}$ by Lemma \ref{hilbertspaceisos}\ref{hilbertspaceisos_2}, so by Lemma \ref{g,A,B_step} parts \ref{g,A,B_step_6} and \ref{g,A,B_step_8}, a matrix computation reveals,
\begin{multline}
\label{xtilde} 
\langle \widetilde{x}, (SS^* \otimes I_H) \widetilde{x} \rangle \\
= \sum_{\xi  \in A_k } \left( \sum_{\mu \in A_k } \left( \sum_{\nu \in A_k } \left\langle \sigma \langle S\xi , \mu \rangle g^* \beta(W_{\mu})^* x,   \sigma\langle S\xi , \nu \rangle ( g^* \beta(W_{\nu})^* x) \right\rangle \right) \right) \\
= \sum_{\xi  \in A_k } \left( \sum_{\mu \in A_k } \left( \sum_{\nu \in A_k } \left\langle g^* \beta(W_{\mu  \langle \mu, S\xi  \rangle})^* x,     g^* \beta(W_{\nu  \langle \nu, S \xi  \rangle})^* x \right\rangle \right) \right) \\
= \sum_{\xi \in A_k } 
\left\langle g^* \beta(W_{ S\xi  })^* x,     g^* \beta(W_{S \xi  })^* x \right\rangle 
= \sum_{\xi \in A_k } 
\left\langle \beta(W_{S \xi  })gg^* \beta(W_{ S\xi  })^* x, x \right\rangle.
\end{multline} 
Using the identities, $Z^{(k)*} Z^{(k)} = R_k^{-2}$ and $ \sum_{j = 1}^k X_j \otimes R_{k-j}^2 = R_k^2$, as well as \emph{two} applications of equation \eqref{xtilde}, first when $S = X_j^{1/2} \otimes \mcZ{k-j}$ and second when $S = \mcZ{k}$, we obtain
\small
\begin{multline*}
\left\langle \sum_{j=1}^k \left( \sum_{\eta \in A_k}  \beta\left(W_{\left(X_j^{1/2} \otimes \mcZ{k-j}\right) \eta} \right) gg^* \beta\left(W_{\left(X_j^{1/2} \otimes \mcZ{k-j}\right) \eta} \right)^* \right) x, x \right\rangle \\
= \sum_{j = 1}^k \langle \widetilde{x}, \left( ( X_j \otimes R_{k-j}^2) \otimes I_{H} \right) \widetilde{x} \rangle 
= \langle \widetilde{x}, ( R_k^2 \otimes I_{H}) \widetilde{x} \rangle \\
= \sum_{\eta \in A_k} \left\langle \left(  \beta\left(W_{\mcZ{k} \eta}\right) gg^* \beta\left(W_{\mcZ{k} \eta}\right)^* \right) x, x \right\rangle.
\end{multline*}
\normalsize
Thus $\sum_{\eta \in A_k} \beta\left(W_{\mcZ{k} \eta}\right) gg^* \beta\left(W_{\mcZ{k} \eta}\right)^* $ converges in $\ms{B}(J_{m+1})$ and equation \eqref{twothirds.1} holds, which completes the proof of inequality  \eqref{givesFcontractive}.

Momentarily, we define $G_{m+1} \in \ms{B}(K, J_{m+1})$ as a $2 \times 2$ operator-valued matrix with respect to the decompositions $K = (K \ominus K_0) \oplus K_0$ and $J_{m+1} = J_m \oplus (J_{m+1} \ominus J_m)$.  
One of the columns of $G_{m+1}$ is determined by the adjoint of a contraction $F:J_{m+1} \to K \ominus K_0$ that we define as follows at $x \in J_{m+1}$.  
By  inequality \eqref{givesFcontractive}, 
\begin{multline}
\label{defineF.1}
\sum_{k=1}^\infty \left( \sum_{\xi \in A_k} \Vert g^* \beta\left(W_{\mcZ{k} \xi}\right)^* x \Vert^2 \right) \\
= \left\langle \sum_{k=1}^\infty \left(\sum_{\xi \in A_k} \beta\left(W_{\mcZ{k} \xi}\right) gg^* \beta\left(W_{\mcZ{k} \xi}\right)^* \right) x, x \right\rangle 
\leq \Vert x \Vert^2.
\end{multline} 
It follows that $\left( \delta_{k \geq 1} \ \left( g^* \beta\left(W_{\mcZ{k} \xi}\right)^* x \right)_{\xi \in A_k} \right)_{k = 0}^\infty$ is a well-defined element of $\sum_{k = 0}^\infty \oplus \mb{H}_k$ where $\delta_{k \geq 1}$ is $1$ when $k \geq 1$ and $0$ when $k = 0$.  
We define $F(x)$ to be its image under $\Gamma^*$ in $K$, namely
\begin{equation}
\label{defineF}
F(x)  = \sum_{k = 1}^\infty \left( \sum_{\xi \in A_k}  \widehat{\xi} \otimes g^* \beta\left(W_{\mcZ{k} \xi}\right)^*x \right).  
\end{equation}
The operator $F$ is linear, and by inequality \eqref{defineF.1},  $F$ is contractive.
If $a \in M$, $k \in \mb{N}$, and $\xi \in A_k$, then $\left\langle \widehat{\xi}, \widehat{a} \right\rangle = 0$.  
Thus, for $x \in J_{m+1}, a \in M$, and $h \in H$, 
\begin{equation*} 
\langle F(x) , \widehat{a} \otimes h \rangle  = 
\sum_{k=1}^\infty \left( \sum_{\xi \in A_k} \left\langle g^* \beta\left(W_{\mcZ{k} \xi}\right)^* x, \sigma\left\langle \widehat{\xi}, \widehat{a} \right\rangle h \right\rangle \right) 
= 0.
\end{equation*} 
It follows that $F(x) \in K \ominus K_0$, as desired.

For organizational purposes, we temporarily adopt the following notation:  if $L$ is a Hilbert space with closed subspace $L_0$, let $\inc{L_0}{L}$ denote the inclusion isometry  in $\ms{B}(L_0, L)$, and let $\proj{L_0}{L}$ denote the projection in $\ms{B}(L)$ onto $L_0$.
It follows that  $\inc{L_0}{L}^*$ is the range restriction of $\proj{L_0}{L}$ to $L_0$,  $\inc{L_0}{L}\inc{L_0}{L}^* = \proj{L_0}{L}$, and $\inc{L_0}{L}^*\inc{L_0}{L}$ is the identity on $L_0$.  
In this notation, let us show that $F$ and $G_m$ satisfy the following relationship:
\begin{equation}
\label{relationship} 
\inc{J_m}{J_{m+1}}^*  F^* = G_m \inc{K \ominus K_0}{K}.
\end{equation} 
If $x \in J_m$ then by equations \eqref{defineF} and  \eqref{Gm*_step},
\begin{multline*}
F \inc{J_m}{J_{m+1}}x  = 
\sum_{k = 1}^\infty \left( \sum_{\xi \in A_k }  \widehat{\xi} \otimes g^* \beta\left(W_{\mcZ{k} \xi} \right)^* \inc{J_m}{J_{m+1}} x \right)  \\
= \sum_{k = 1}^\infty \left( \sum_{\xi \in A_k}  \widehat{\xi} \otimes g^* \alpha\left(W_{\mcZ{k} \xi}\right)^*x \right) 
=  \inc{K \ominus K_0}{K}^* G_m^* x.
\end{multline*} 
Therefore, $F \inc{J_m}{J_{m+1}} = \inc{K \ominus K_0}{K}^* G_m^*$, and equation \eqref{relationship} follows.

In our $2 \times 2$ matricial definition of $G_{m+1} $, the following maps, $R$, $S$, and $T$, will comprise three of the four corners.
Define
\begin{alignat*}{2}
R  & = \inc{J_m}{J_{m+1}}^*   F^*     = G_m  \inc{K\ominus K_0}{K} && \quad \text{in } \ms{B}(K \ominus K_0, J_m);\\
S & = \inc{J_{m+1} \ominus J_m}{J_{m+1}}^*  F^* && \quad \text{in } \ms{B}(K \ominus K_0, J_{m+1} \ominus J_m); \text{ and}\\
T & = G_m  \inc{K_0}{K} && \quad \text{in } \ms{B}(K_0, J_m);
\end{alignat*}
noting that $R$ is well-defined by equation \eqref{relationship}.  
Matricially,
\begin{equation} 
\label{row_column}
F^* = \left[ \begin{matrix} R \\ S \end{matrix} \right] \quad \text{ and } \quad  G_m = \left[ \begin{matrix} R & T \end{matrix} \right].
\end{equation} 

By Parrott's lemma \cite{Parrott1978}, there exists $U \in \ms{B}(K_0,J_{m+1} \ominus J_m)$ such that 
\begin{equation} 
\label{Parrott.1}
\left\Vert \left[ 
\begin{matrix} R & T \\ S & U
\end{matrix}
\right] \right\Vert   =  \max \left\{ \left\Vert F \right\Vert, \left\Vert G_m \right\Vert\right\}, 
\end{equation}
so at last we define
\begin{equation*} 
\label{defineG_m+1}
G_{m+1} = \left[ 
\begin{matrix} R & T \\ S & U
\end{matrix}
\right]
\end{equation*} 
 in $\ms{B}(K, J_{m+1})$.
By equation \eqref{row_column},
\begin{align}
\label{row_column.2a}
G_{m+1} \inc{K \ominus K_0}{K} = F^*  \text{ and } \\
\label{row_column.2b}
 \inc{J_m}{J_{m+1}}^* G_{m+1} = G_m .
\end{align}

Our goal is to show that $\left(\{n_i\}_{i=1}^{m+1}, \{J_i\}_{i = 1}^{m+1},  \{G_i\}_{i = 1}^{m+1} \right)$ is an $(m+1)$-triple.
We have already shown that the first four $(m+1)$-level conditions are satisfied.  
Condition $(7,m+1)$ readily follows from equation \eqref{Parrott.1}.

Towards proving Condition $(5,m+1)$, recall that  $\sigma^{\ms{F}(E)} \circ \varphi_{\infty}: M \to \ms{B}(K)$ is the normal, unital $*$-homomorphism that maps $a \in M$ to $\phiinf{a} \otimes I_{H}$.  
Let each of 
$\psi_1 : M \to \ms{B}(K \ominus K_0)$, 
$\psi_2 : M \to \ms{B}(K_0)$,
$\pi_1 : M \to \ms{B}(J_m)$, and
$\pi_2 : M \to \ms{B}(J_{m+1} \ominus J_m)$
be the compression of $\sigma^{\ms{F}(E)} \circ \varphi_{\infty}$ to the indicated space; 
for instance, for every $a \in M$,  $\psi_1(a) = \inc{K \ominus K_0}{K}^* (\phiinf{a} \otimes I_{H}) \inc{K \ominus K_0}{K}$.
The spaces $K_0$, $J_m$, and $J_{m+1}$ reduce $\sigma^{\ms{F}(E)} \circ \varphi_{\infty}$, from which it follows that $K \ominus K_0$ and $J_{m+1} \ominus J_m$ also reduce $\sigma^{\ms{F}(E)} \circ \varphi_{\infty}$.  
Therefore, $\psi_1$, $\psi_2$, $\pi_1$, and $\pi_2$ are normal, unital $*$-homomorphisms. 
We want to show that 
\begin{equation}\label{RSTU}
R   \in \mc{I}(\psi_1,\pi_1); \
S   \in \mc{I}(\psi_1,\pi_2);\
T   \in \mc{I}(\psi_2,\pi_1); \text{ and } 
U   \in \mc{I}(\psi_2,\pi_2).
\end{equation}
Note that $V_m = \inc{J_m}{K}$ and $V_{m+1} = \inc{J_{m+1}}{K}$.
Since  $K \ominus K_0$ reduces $\sigma^{\ms{F}(E)} \circ \varphi_{\infty}$,  if $a \in M$, then by Condition  $(5,m)$,
\begin{multline*}
R  \psi_1(a) 
 = G_m  \proj{K \ominus K_0}{K}   (\phiinf{a} \otimes I_{H}) \inc{K \ominus K_0}{K} \\
 = G_m   (\phiinf{a} \otimes I_{H}) \inc{K \ominus K_0}{K} \\
 =   V_m^* (\phiinf{a} \otimes I_{H}) V_m G_m \inc{K \ominus K_0}{K} 
 = \pi_1(a) R,
\end{multline*}
Thus $R   \in \mc{I}(\psi_1,\pi_1)$; analogously $T   \in \mc{I}(\psi_2,\pi_1)$.  
To show that $S   \in \mc{I}(\psi_1,\pi_2)$, it suffices to show that for every $a \in M$, $x \in J_{m+1} \ominus J_m$, $\eta \in \tens{E}{j}$ with $j \in \mb{N}$, and $h \in H$,
\begin{equation}
\label{RST_step_e1} \langle \widehat{\eta} \otimes h, \psi_1(a)S^*x \rangle = \langle \widehat{\eta} \otimes h, S^*\pi_2(a) x\rangle.
\end{equation}
Since $K \ominus K_0$ reduces $\sigma^{\ms{F}(E)} \circ \varphi_{\infty}$, by equation \eqref{defineF},
\begin{align*}
\psi_1(a) S^*x 
& = (\phiinf{a} \otimes I_{H})F   x  \\
& = (\phiinf{a} \otimes I_{H}) \left( \sum_{k=1}^\infty \left( \sum_{\xi \in A_k }  \widehat{\xi} \otimes g^* \beta\left(W_{\mcZ{k} \xi}\right)^* x \right) \right) \\
& = \sum_{k=1}^\infty \left( \sum_{\xi \in A_k }  \widehat{a \cdot \xi} \otimes g^* \beta\left(W_{\mcZ{k} \xi}\right)^* x  \right).
\end{align*}
Thus by Lemma \ref{g,A,B_step}, parts \ref{g,A,B_step_6} and \ref{g,A,B_step_8},
\begin{multline} 
\label{RST_step_2b.1}
\langle \widehat{\eta} \otimes h, \psi_1(a)S^*x \rangle 
= \sum_{\xi \in A_j} \left\langle   h, \sigma \langle \eta, a \cdot \xi \rangle\left(g^* \beta\left(W_{\mcZ{j} \xi}\right)^* x \right) \right\rangle \\
= \sum_{\xi \in A_j} \left\langle   h, g^* \beta\left(W_{\mcZ{j} \xi \cdot \langle \xi, a^* \eta \rangle}\right)^* x  \right\rangle  
= \left\langle   h, g^* \beta\left(W_{\mcZ{j}  (a^* \eta)} \right)^* x  \right\rangle.
\end{multline}
On the other hand, since $J_{m+1} \ominus J_m$ reduces $\sigma^{\ms{F}(E)} \circ \varphi_\infty$,  by Lemma \ref{g,A,B_step}\ref{g,A,B_step_4},
\begin{multline*}
S^* \pi_2(a)x 
= F  \inc{J_{m+1} }{K}^*  (\phitens{a}) \inc{J_{m+1} }{K} x \\
  = \sum_{k=1}^\infty \left( \sum_{\xi \in A_k }  \widehat{\xi} \otimes \left(g^* \beta\left(W_{\mcZ{k} \xi}\right)^*\inc{J_{m+1} }{K}^*  (\phitens{a}) \inc{J_{m+1} }{K} x \right) \right)\\
  = \sum_{k=1}^\infty \left( \sum_{\xi \in A_k }  \widehat{\xi} \otimes g^* \beta\left(W_{a^* \cdot \mcZ{k} \xi}\right)^* x \right).
\end{multline*}
Therefore by Lemma \ref{g,A,B_step}, parts \ref{g,A,B_step_6} and \ref{g,A,B_step_8}, since $\mcZ{j} \in \varphi_j(M)^c$,
\begin{multline} 
\label{RST_step_2b.2}
\langle \widehat{\eta} \otimes h, S^*\pi_2(a) x\rangle
= \sum_{\xi \in A_j} \left\langle   h, \sigma \langle \eta,  \xi \rangle\left(g^* \beta\left(W_{a^* \cdot \mcZ{j} \xi}\right)^* x \right) \right\rangle \\
=  \sum_{\xi \in A_j} \left\langle   h, g^* \beta\left(W_{a^* \mcZ{j} \xi \langle \xi, \eta \rangle}\right)^* x  \right\rangle
=   \left\langle   h, g^* \beta\left(W_{\mcZ{j} (a^*  \eta) }\right)^* x  \right\rangle.
\end{multline}
Combining equations \eqref{RST_step_2b.1} and \eqref{RST_step_2b.2}, we obtain equation \eqref{RST_step_e1}, so $S   \in \mc{I}(\psi_1,\pi_2)$.  
Examining the proof of Parrott's Lemma in \cite{Parrott1978}, we see that $U$ is given as the weak operator limit of the sequence
\begin{equation*}
\left\{ -c_n S(I - c_n^2 R^* R )^{-1} R^* T \right\}_{n = 0}^\infty
\end{equation*}
for some sequence of numbers $\{c_n\}_{n = 0}^\infty$.
Since 
 $R   \in \mc{I}(\psi_1,\pi_1)$, 
$S   \in \mc{I}(\psi_1,\pi_2)$, and
$T   \in \mc{I}(\psi_2,\pi_1)$, it follows that $U   \in \mc{I}(\psi_2,\pi_2)$, which gives \eqref{RSTU}.

To establish Condition $(5,m+1)$, by Conditions $(4,m+1)$ and $(5,m)$, it suffices to show that 
\begin{equation} 
\label{cond5m+1}
(V_{m+1}^* (Y \otimes I_H) V_{m+1} )G_{m+1} = G_{m+1} (Y \otimes I_H) 
\end{equation}
in only two cases: when $Y = \varphi_\infty(a)$ for some $a \in M$ and when $Y = W_\xi$ for some $\xi \in E$.
Using properties \eqref{RSTU}, we compute
\begin{multline*}
V_{m+1}^* (\phitens{a}) V_{m+1} G_{m+1} 
 = 
\left[ \begin{matrix} \pi_1(a) & 0 \\ 0 & \pi_2(a) \end{matrix} \right] \left[ 
\begin{matrix} R & T \\ S & U
\end{matrix}
\right] \\
= \left[ \begin{matrix} \pi_1(a) R & \pi_1(a) T \\ \pi_2(a)  S & \pi_2(a) U
\end{matrix}
\right] 
= \left[ 
\begin{matrix} R \psi_1(a) & T \psi_2(a)\\ S \psi_1(a)& U \psi_2(a)
\end{matrix}
\right]  \\
= \left[ 
\begin{matrix} R & T \\ S & U
\end{matrix}
\right] \left[ \begin{matrix} \psi_1(a) & 0 \\ 0 & \psi_2(a) \end{matrix} \right] 
= G_{m+1} (\phitens{a}).
\end{multline*}
This establishes equation \eqref{cond5m+1} when $Y = \varphi_\infty(a)$. 
For the case when $Y = W_\xi$, it suffices to show that for every $j \in \mb{N}_0, \eta \in E^{\otimes j}, h \in H$, and $ x \in J_{m+1}$, 
\begin{equation}
\label{7_m+1.2} 
\left\langle V_{m+1}^* (\wtens{\xi}) V_{m+1} G_{m+1} \left( \widehat{\eta} \otimes  h\right)  ,  x  \right\rangle = \left\langle G_{m+1} (\wtens{\xi}) \left( \widehat{\eta} \otimes h \right), x \right\rangle.
\end{equation}
With $W_{\xi}(\widehat{\eta}) \otimes h = \widehat{Z_{j+1}(\xi \otimes \eta)} \otimes h \in K \ominus K_0$, by equations \eqref{row_column.2a} and \eqref{defineF},
\begin{multline*}
\left\langle G_{m+1} (\wtens{\xi}) \left( \widehat{\eta} \otimes h \right), x \right\rangle \\
= \left\langle  \widehat{Z_{j+1} (\xi \otimes \eta) }  \otimes h, \inc{K \ominus K_0}{K}^* G_{m+1}^* x \right\rangle 
= \left\langle  \widehat{Z_{j+1} (\xi \otimes \eta) } \otimes h, F x \right\rangle  \\
= \sum_{k = 1}^\infty \left( \sum_{\mu \in A_{k}} \left\langle  \widehat{Z_{j+1} (\xi \otimes \eta) } \otimes h, \widehat{\mu} \otimes g^* \beta\left(W_{\mcZ{k} \mu} \right)^*x \right\rangle \right).
\end{multline*}
For $k \in \mb{N}$ and $\mu \in A_k$,  $\left\langle \widehat{Z_{j+1} (\xi \otimes \eta)}, \widehat{\mu} \right\rangle$ is $\langle Z_{j+1} (\xi \otimes \eta), \mu \rangle$ when $k = j+1$ and is otherwise zero.
Recalling that $\mcZ{j+1} Z_{j+1} = I_1 \otimes \mcZ{j}$, we continue the preceding computation using Lemma \ref{g,A,B_step}, parts \ref{g,A,B_step_6}  and \ref{g,A,B_step_8}, 
\begin{multline}
\label{7_m+1.4}
\left\langle G_{m+1} (\wtens{\xi}) \left( \widehat{\eta} \otimes h \right), x \right\rangle \\
= \sum_{\mu \in A_{j+1}} \left\langle   h, \sigma\langle Z_{j+1} (\xi \otimes \eta), \mu \rangle \left( g^* \beta\left(W_{\mcZ{j+1} \mu} \right)^*x \right) \right\rangle \\
= \sum_{\mu \in A_{j+1}} \left\langle   h,  g^* \beta\left(W_{\mcZ{j+1} \mu \langle \mu, Z_{j+1} (\xi \otimes \eta) \rangle}\right)^*x  \right\rangle \\
=  \left\langle   h,  g^* \beta\left(W_{\mcZ{j+1} Z_{j+1} (\xi \otimes \eta) }\right)^*x  \right\rangle 
=  \left\langle   h,  g^* \beta\left(W_{\xi \otimes \mcZ{j}\eta}\right)^*x \right\rangle \\
=  \left\langle  \beta\left(W_{\xi \otimes \mcZ{j}\eta}\right) g(h), x \right\rangle.
\end{multline}
On the other hand, we have that $P_{m} V_{m+1} = V_m \inc{J_m}{J_{m+1}}^*$ and $V_{m+1}^* (\wtens{\xi}) = V_{m+1}^* (\wtens{\xi}) P_{m}$ by equation \eqref{cond5_projections_a}.  
Thus, by  equation \eqref{row_column.2b} and Lemma \ref{g,A,B_step}, parts \ref{g,A,B_step_5b} and \ref{g,A,B_step_3},
\begin{multline*}
V_{m+1}^* (\wtens{\xi}) V_{m+1} G_{m+1} \left( \widehat{\eta} \otimes  h \right) 
=  V_{m+1}^* (\wtens{\xi}) P_{m} V_{m+1} G_{m+1} L_{\widehat{\eta}} h \\
= V_{m+1}^*  (\wtens{\xi}) V_m \inc{J_m}{J_{m+1}}^* G_{m+1} L_{\widehat{\eta}} h 
=   \beta(W_{\xi}) G_m L_{\widehat{\eta}} h\\
=  \beta\left(W_{\xi}\right) \alpha\left(W_{\mcZ{j}\eta}\right) g(h)  
=  \beta\left(W_{\xi \otimes \mcZ{j}\eta}\right) g (h).
\end{multline*}
This fact together with equation \eqref{7_m+1.4} yields equation \eqref{7_m+1.2}, which completes the proof of Condition $(5,m+1)$.

All that remains is Condition $(6,m+1)$, and we need only consider the case when $j = m+1$. 
If $i = m+1$, then $V_i^* V_j G_j = G_{m+1} = G_i$, as desired.
Suppose $1 \leq i \leq m$.  
Since  $J_i \subseteq J_m \subseteq J_{m+1}\subseteq K$, it follows that $V_j^* V_i  = \inc{J_m}{J_{m+1}}\inc{J_i}{J_{m}}$ and $V_m^* V_i  = \inc{J_i}{J_{m}}$.  
Taking adjoints, applying equation \eqref{row_column.2b}, and using Condition $(6,m)$, we have 
$V_i^* V_j G_j
=  \inc{J_i}{J_{m}}^* \inc{J_m}{J_{m+1}}^*   G_{m+1} 
=  \inc{J_i}{J_{m}}^*  G_m 
= V_i^* V_m G_m 
= G_i$.  
Thus, Condition $(6,m+1)$ is satisfied.
Having shown all seven conditions, we conclude that 
$\left(\{n_i\}_{i=1}^{m+1}, \{J_i\}_{i = 1}^{m+1}, \{G_i\}_{i = 1}^{m+1}\right)$ is an $(m+1)$-triple.

Recall that we are considering the case when there is no $m \in \mb{N}$ such that an $m$-triple $\left(\{n_i\}_{i=1}^m, \{J_i\}_{i = 1}^m,  \{G_i\}_{i = 1}^m \right)$ exists with $J_m = K$.  
Since we know $\left( \{n_i\}_{i=1}^1, \{J_i\}_{i = 1}^1,  \{G_i\}_{i = 1}^1 \right)$ is a $1$-triple, the result we have just shown, together with an inductive argument, guarantees the existence of three sequences $\{n_i\}_{i=1}^\infty$, $\{J_i\}_{i = 1}^\infty$, and $\{G_i\}_{i = 1}^\infty$ such that  $\left(\{n_i\}_{i=1}^m, \{J_i\}_{i = 1}^m,  \{G_i\}_{i = 1}^m \right)$ is an $m$-triple for every $m \in \mb{N}$.
As before, for every $i \in \mb{N}$, let $V_i \in \ms{B}(J_i, K)$ be the inclusion map, and let  $P_i \in \ms{B}(K)$ be the projection map  onto $J_i$. 

Towards defining $\widetilde{G}$, note that $\{J_i\}_{i = 1}^\infty$ is an increasing family of subspaces of $K$ because Condition $(2,m)$ holds for every $m$, so $J_\infty : = \bigcup_{k = 1}^\infty J_k$ is a linear subspace of $K$.  
An inductive argument using Condition $(1,m)$ for $m \geq 2$ implies that for all $k \in \mb{N}_0$, $k \leq n_{k+2}$.  
Thus for each $k \in \mb{N}_0$, Condition $(3,k+2)$ implies that $K_k \subseteq K_{n_{k+2}} \subseteq J_{k+2}$.  
Therefore $\bigcup_{k = 0}^\infty K_k \subseteq J_\infty$. 
It follows that $K =  \overline{J_\infty}$.  
Suppose $x \in J_\infty$ and $x \in J_i \cap J_j$ for some $i,j \in \mb{N}$.  
Assuming without loss of generality that $i \leq j$, $J_i \subseteq J_j$ by Condition $(2,j)$, and by Condition $(8,j)$, $V_i^* V_j G_j = G_i$.  
Therefore, $G_i^*x  = G_j^* V_j^* V_i x =  G_j^*x$.
Using the Conditions $(7, m)$ for $m \geq 1$, it follows that there is a well-defined contraction $C \in \ms{B}(K)$ such that for every $k \in \mb{N}$ and $x \in J_k$, $C(x) = G_k^*(x)$. 
  
Let us define $\widetilde{G} : = C^* $
and show that $\widetilde{G}$ satisfies the four properties stated in the conclusion of the theorem.
First, note that for any $k \in \mb{N}$ and $x \in J_k$, $\widetilde{G}^* V_{k}x = Cx = G_k^*x$.  Thus
\begin{equation}
\label{WCT_e1} \widetilde{G}^*V_{k} = G_k^*, \quad   \forall k \in \mb{N}.  
\end{equation}
In particular, $\widetilde{G}^* V = G_1^*$ since $V = V_1$.  
Since $G_1 = GV^*$ and $P = VV^*$ is the projection map onto $J$ in $\ms{B}(K)$, we have $\widetilde{G}^*P  = V G^*V^*$.
Properties \ref{WCL_theorem_i1} and \ref{WCL_theorem_i2} follow readily.
Let $Y \in H^\infty(E,Z)$ and $k \in \mb{N}$.
By Condition $(4,k)$, $V_k^*(Y \otimes I_H) =  V_{k}^* (Y \otimes I_H) P_k$; by Condition $(5,k)$ and two applications of the adjoint of equation \eqref{WCT_e1}, 
$V_k^* (Y \otimes I_H) \widetilde{G}
= V_{k}^* (Y \otimes I_H)   P_k \widetilde{G}
= V_{k}^* (Y \otimes I_H)   V_{k} G_k
= G_k (Y \otimes I_H) 
= V_{k}^*\widetilde{G}(Y \otimes I_H)$.
Property \ref{WCL_theorem_i3} follows since $K = \overline{J_\infty}$.  
Finally, $C$ is a contraction, but also $\Vert C \Vert \geq 1$ since $\Vert G_1 \Vert = 1$.  
This gives property \ref{WCL_theorem_i4}, completing the proof of the theorem in the case when $\Vert G \Vert = 1$.  
If $G$ is the zero operator in $\ms{B}(J)$, then the zero operator in $\ms{B}(K)$  satisfies the desired properties for $\widetilde{G}$.  
If $G \neq 0$, then a straightforward scaling argument utilizing the case treated above produces the desired result.
\end{proof}

The following corollary generalizes Theorem \ref{WCL_theorem} to the case where the $W^*$-algebra is represented on \emph{two} Hilbert spaces.  
The proof makes use of the so-called Putnam trick.
It is this corollary that we use to prove the weighted Nevanlinna-Pick interpolation theorem in Section \ref{wnpChap}.

\begin{corollary}\label{GenWeightedCommutantLifting}
Let $\sigma_1: M \to \ms{B}(H_1)$ and $\sigma_2: M \to \ms{B}(H_2)$ be faithful, normal, unital $*$-homomorphism for  Hilbert spaces $H_1$ and $H_2$, and let $Z$ be a  sequence of weights associated with $X$.  
For $j = 1, 2$, suppose that $J_j$ is a closed linear subspace of $\ms{F}(E) \otimes_{\sigma_j} H_j$ such that for every $Y \in H^\infty(E,Z)$, $\left(Y^* \otimes I_{H_j} \right) (J_j)  \subseteq J_j$.
For $j = 1,2$, let $V_j$ be the inclusion map of ${J_j}$ into $\ms{F}(E) \otimes_{\sigma_j} H_j$.  
Suppose there exists $G \in \ms{B}(J_1, J_2)$ such that for every $Y \in H^\infty(E,Z)$,
$G \left( V_1^* (Y \otimes I_{H_1}) V_1 \right) = \left(V_2^* (Y \otimes I_{H_2}) V_2 \right)  G$.
Then there exists $\widetilde{G} \in \ms{B}(\ms{F}(E) \otimes_{\sigma_1} H_1, \ms{F}(E) \otimes_{\sigma_2} H_2)$ such that 
\begin{enumerate}[label=(\arabic*),ref=(\arabic*)]
\item  \label{wcl.cor.1} $\widetilde{G}^* (J_2) \subseteq J_1$, 
\item \label{wcl.cor.2} $ V_2^* \widetilde{G}  V_1 = G$, 
\item  \label{wcl.cor.3} $\widetilde{G}  (Y \otimes I_{H_1}) = (Y \otimes I_{H_2}) \widetilde{G}$ for all $Y \in H^\infty(E,Z)$, and
\item \label{wcl.cor.4} $\Vert \widetilde{G} \Vert = \Vert G \Vert$.
\end{enumerate}
\end{corollary}

\begin{proof}
We will use the ``Putnam trick'' to translate the two-space problem into a one-space problem where we may apply the original result, Theorem \ref{WCL_theorem}.  
We then return to the two-space setting by picking off the lower left-hand entries in the $2 \times 2$ matricial expressions for certain operators.

Let $H = H_1 \oplus H_2$. 
Let $\sigma = \sigma_1 \oplus \sigma_2: M \to \ms{B}(H)$.  
Define the induced representation spaces $K_1 = \F{E} \otimes_{\sigma_1} H_1$, $K_2 = \F{E} \otimes_{\sigma_2} H_2$, and $K = \F{E} \otimes_\sigma H$.  
We identify $K$ with $K_1 \oplus K_2$ in the usual way and let $J$ be the image of $J_1 \oplus J_2$ in $K$. 
For notational simplicity we will omit the implied isomorphisms in our computations. 
Let $V$ be the inclusion map of $J$ into $K$.  
Let $P_1, P_2$, and $P$ be the projection maps in $\ms{B}(K_1)$, $\ms{B}(K_2)$, and $\ms{B}(K)$ onto $J_1, J_2$, and $J$, respectively.  
Define
\begin{equation}
\label{corollary.8}
G_0 : = \left[ \begin{matrix} 0 & 0 \\ G & 0 \end{matrix} \right] \in \ms{B}(J).
\end{equation}
Towards applying Theorem \ref{WCL_theorem}, let us show that for all $Y \in H^\infty(E,Z)$,
\begin{align}
\label{corollary.5}
(Y^* \otimes I_H)(J) & \subseteq J \text{ and} \\
\label{corollary.7}
 G_0 (V^*(Y \otimes I_H)V) & = (V^*(Y \otimes I_H)V) G_0.
\end{align}
By hypothesis, $P_j(Y \otimes I_{H_j}) = P_j (Y \otimes I_{H_j})P_j$  for $j = 1,2$. 
Thus,
\begin{multline*}
P(Y \otimes I_H) =
\left[ \begin{matrix} P_1 & 0 \\ 0 & P_2
\end{matrix} \right] 
\left[ \begin{matrix} Y \otimes I_{H_1} & 0 \\ 0 & Y \otimes I_{H_2}
\end{matrix} \right] \\
 = \left[ \begin{matrix} P_1 & 0 \\ 0 & P_2
\end{matrix} \right]\left[ \begin{matrix} Y \otimes I_{H_1} & 0 \\ 0 & Y \otimes I_{H_2}
\end{matrix} \right]\left[ \begin{matrix} P_1 & 0 \\ 0 & P_2
\end{matrix} \right] 
= P(Y \otimes I_H)P.
\end{multline*}
Containment \eqref{corollary.5} follows.
Since $VV^* = P$ and $V^*V$ is the identity on $J$,
\begin{multline}
\label{corollary.10}
G_0 V^*(Y \otimes I_H)V 
= V^*V G_0 V^*P(Y \otimes I_H)PV \\
= V^* \left[ \begin{matrix} 0 & 0 \\ V_2 G V_1^* & 0 \end{matrix} \right] \left[ \begin{matrix} P_1 & 0 \\ 0 & P_2 \end{matrix} \right] \left[ \begin{matrix} Y \otimes I_{H_1} & 0 \\  0 & Y \otimes I_{H_2} \end{matrix} \right] \left[ \begin{matrix} P_1 & 0 \\ 0 & P_2 \end{matrix} \right] V \\
= V^*  \left[ \begin{matrix} 0 & 0 \\ V_2 G V_1^*(Y \otimes I_{H_1}) P_1 & 0 \end{matrix} \right]  V.
\end{multline} 
On the other hand, 
\begin{multline}
\label{corollary.9}
 V^*(Y \otimes I_H)V G_0 
 =V^*P(Y \otimes I_H)PV G_0 V^* V \\
 = V^*\left[ \begin{matrix} P_1 & 0 \\ 0 & P_2 \end{matrix} \right] \left[ \begin{matrix} Y \otimes I_{H_1} & 0 \\  0 & Y \otimes I_{H_2} \end{matrix} \right] \left[ \begin{matrix} P_1 & 0 \\ 0 & P_2 \end{matrix} \right] \left[ \begin{matrix} 0 & 0 \\ V_2 G V_1^* & 0 \end{matrix} \right]  V \\
 = V^*  \left[ \begin{matrix} 0 & 0 \\ P_2 (Y \otimes I_{H_2})V_2 G V_1^* & 0 \end{matrix} \right]  V.
\end{multline}
By our hypothesis, $V_2 G V_1^*(Y \otimes I_{H_1}) P_1 = P_2(Y \otimes I_{H_2}) V_2 G V_1^*$.
Equation \eqref{corollary.7} now follows from  equations \eqref{corollary.10} and \eqref{corollary.9}.

Having shown containment \eqref{corollary.5} and equation \eqref{corollary.7}, we apply Theorem \ref{WCL_theorem} and conclude that there exists $\widetilde{G_0} \in \ms{B}(K)$ such that:
\begin{enumerate}[label=(\arabic*$'$),ref=(\arabic*$'$)]
\item \label{wcl_app.1} $\widetilde{G_0}^* (J) \subseteq J$,
\item \label{wcl_app.2} $ V^*  \widetilde{G_0}  V = G_0$,
\item \label{wcl_app.3} $\widetilde{G_0}    (Y \otimes I_{H})  =  (Y \otimes I_{H}) \widetilde{G_0}$ for all $Y \in H^\infty(E,Z)$, and
\item \label{wcl_app.4} $\Vert \widetilde{G_0} \Vert = \Vert G_0 \Vert$.
\end{enumerate}  
With $A \in \ms{B}(K_1), B \in \ms{B}(K_2, K_1), C \in \ms{B}(K_2)$, and $\widetilde{G} \in \ms{B}(K_1, K_2)$ such that
\begin{equation}
\label{corollary.11}
\widetilde{G_0} = \left[ \begin{matrix} A & B \\ \widetilde{G} & C \end{matrix} \right],
\end{equation} 
let us show that $\widetilde{G}$ satisfies properties \ref{wcl.cor.1}-\ref{wcl.cor.4}.
By property \ref{wcl_app.1}, $P\widetilde{G_0} = P \widetilde{G_0} P$, so
\begin{multline*}
\left[ \begin{matrix} P_1A & P_1B  \\ P_2 \widetilde{G} & P_2C \end{matrix} \right] 
=  \left[ \begin{matrix} P_1 & 0 \\ 0 & P_2 \end{matrix} \right] \left[ \begin{matrix} A & B  \\ \widetilde{G} & C \end{matrix} \right] 
=  P \widetilde{G_0} 
= P \widetilde{G_0}  P \\
=  \left[ \begin{matrix} P_1 & 0 \\ 0 & P_2 \end{matrix} \right] \left[ \begin{matrix} A & B  \\ \widetilde{G} & C \end{matrix} \right]  \left[ \begin{matrix} P_1 & 0 \\ 0 & P_2 \end{matrix} \right] 
= \left[ \begin{matrix} P_1AP_1 & P_1BP_2  \\ P_2\widetilde{G} P_1 & P_2C P_2 \end{matrix} \right].
\end{multline*} 
Equating the lower left-hand entries, property \ref{wcl.cor.1} holds. 
It follows from property  \ref{wcl_app.2} that $P \widetilde{G_0} P = V G_0V^*$, so by similar computations, 
\begin{equation*}
\left[ \begin{matrix} P_1AP_1 & P_1BP_2 \\ P_2\widetilde{G}P_1 & P_2CP_2 \end{matrix} \right] 
=  \left[ \begin{matrix} 0 & 0 \\ V_2 G V_1^* & 0 \end{matrix} \right].
\end{equation*}  
Property \ref{wcl.cor.2} follows by equating the lower left-hand entries.
Property \ref{wcl_app.3} implies
\begin{equation*}
\left[ \begin{matrix} A(Y \otimes I_{H_1}) & B(Y \otimes I_{H_2}) \\ \widetilde{G}(Y \otimes I_{H_1})& C(Y \otimes I_{H_2}) \end{matrix} \right]  = 
\left[ \begin{matrix} (Y \otimes I_{H_1})A & (Y \otimes I_{H_1})B \\ (Y \otimes I_{H_2})\widetilde{G}& (Y \otimes I_{H_2})C \end{matrix} \right] 
\end{equation*}
for $Y \in H^\infty(E,Z)$.  
Again, equating the lower left-hand entries, we obtain property \ref{wcl.cor.3}.
Finally, it follows from equation \eqref{corollary.11} that $\Vert \widetilde{G} \Vert \leq  \Vert \widetilde{G_0} \Vert$.  
By property \ref{wcl_app.4} and equation \eqref{corollary.8},  $\Vert \widetilde{G_0} \Vert = \Vert G_0 \Vert =  \Vert G \Vert $.
Since  $\Vert G \Vert  \leq \Vert \widetilde{G} \Vert$ by property \ref{wcl.cor.2}, we conclude that  $\Vert G \Vert = \Vert \widetilde{G} \Vert$.  
This gives property \ref{wcl.cor.4} and completes the proof.  
\end{proof}

\section{The Commutant and Double Commutant}\label{dblcomm}

In this section we identify the commutant of an induced image of the weighted Hardy algebra of the dual correspondence and the double commutant of an induced image of the weighted Hardy algebra of the original correspondence.
For the remainder of the paper, we will assume that $E$ is full.  

Let $\sigma: M \to \ms{B}(H)$ be a fixed faithful, normal, unital $*$-homomor\-phism for a Hilbert space $H$, and let $Z$ be a sequence of  weights associated with $X$.
In Section 7 of  \cite{Muhly2016}, Muhly and Solel use $X$ to construct an admissible sequence for $E^\sigma$, $X' = \left\{X'_k \right\}_{k = 0}^\infty$, and $Z$ to construct a sequence of weights associated with $X'$, which we will denote by $Z' = \left\{Z'_k \right\}_{k = 0}^\infty$.
Our proof of the weighted Nevanlinna-Pick theorem will involve the commutant of the algebra $\left\{ Y \otimes I_H \mid Y \in H^\infty(E^\sigma, Z') \right\}$ in $\ms{B}(\ms{F}(E^\sigma) \otimes_\iota H)$; our next goal is to identify it with $H^\infty(E,Z)$ in a certain way.   
A note of caution is in order, for what we will call $Z'$ is referred to as $C$ in \cite{Muhly2016}.  
Let us summarize the construction of $X'$ and $Z'$, simultaneously clarifying our notation. 
For $k \in \mb{N}_0$, we shall write $\mc{A}_k$ for  $\varphi_k(M)^c$, the commutant of $\varphi_k(M)$ in $\ms{L}(\tens{E}{k})$, and   $\mc{A}'_k$ for $\varphi_k'(\sigma(M)')^c$, the commutant of the image of $\sigma(M)'$ under $\varphi_k'$ in $\ms{L}(\tens{(E^\sigma)}{k})$.
Using Lemma 7.2 of \cite{Muhly2016}, for each $k \in \mb{N}_0$ we define $X_k'$ to be the unique element in  $\mc{A}'_k$ such that 
$ U_k^{\sigma *} (X_k \otimes I_H) U_k^{\sigma} = X_k' \otimes I_H$.
For $k \in \mb{N}_0$, define
\begin{equation*}
C_k = 
\begin{cases}
I_M & \text{ if } k = 0 \\
Z^{(k)} (Z^{(k-1)} \otimes I_1)^{-1}  & \text{ if } k \geq 1
\end{cases}.
\end{equation*}
Since $C_k \in \mc{A}_k$,  we define $Z_k'$ to be unique element in $\mc{A}'_k$ such that $U_k^{\sigma *} (C_k \otimes I_H) U_k^{\sigma} =  Z_k' \otimes I_H  $ (\cite{Muhly2016}, Lemma 7.2).  
Here is where we have modified the notation used in \cite{Muhly2016}; we have interchanged the roles of $C$ and $Z'$ as they were introduced in Lemma 7.4 of \cite{Muhly2016}.
By Theorem 7.6 of \cite{Muhly2016}, $X' = \{X'_k\}_{k = 0}^\infty$ is an admissible sequence and $Z' = \{Z'_k\}_{k = 0}^\infty$ is a sequence of weights associated with $X'$.

We now repeat the construction with $E^\sigma$ replacing $E$ and $E^{\sigma \iota}$ replacing $E^\sigma$. 
For $k \in \mb{N}_0$, let $\mc{A}''_k$ denote $\varphi_{k}''(\sigma(M))^c$, the commutant of $\varphi_{k}''(\sigma(M))$ in $\ms{L}(\tens{(E^{\sigma \iota})}{k})$.
For $k \in \mb{N}_0$ define 
\begin{equation*}
C_k' : = 
\begin{cases}
I_{\sigma(M)'} & \text{ if } k = 0 \\
Z'^{(k)} (Z'^{(k-1)} \otimes I'_1)^{-1}  & \text{ if } k \geq 1
\end{cases}.
\end{equation*}
Since $X_k',C_k' \in \mc{A}'_k$, there exist unique $X_k '',Z''_k \in \mc{A}''_k$ such that $ U_k^{\iota*} (X_k' \otimes I_H)U_k^{\iota} = X_k'' \otimes I_H$ and $ U_k^{\iota*} (C_k' \otimes I_H)U_k^{\iota} = Z''_k \otimes I_H$ (\cite{Muhly2016}, Lemma 7.2).
It follows from Theorem 7.6 of \cite{Muhly2016} that $X'' = \left\{X''_k \right\}_{k = 0}^\infty$ is an admissible sequence for  $E^{\sigma \iota}$  and $Z'' = \left\{Z''_k \right\}_{k = 0}^\infty$ is a sequence of weights associated with $X''$.
The following fact will be useful in future computations.
\begin{lemma}\label{ZandC} For all $k \in \mb{N}_0$, \quad
$ U_k^{\sigma *} (Z_k \otimes I_H) U_k^{\sigma } = C_k' \otimes I_H$.
\end{lemma}

\begin{proof}
The case when $k = 0$ follows from the fact that $Z_0$ is the identity on $M$ and $C_0'$ is the identity on $\sigma(M)'$.
Suppose $k \geq 1$.  
By the discussion preceding Theorem 7.6 in \cite{Muhly2016}, $Z'^{(k)} \otimes I_H = U_k^{\sigma *} (Z^{(k)} \otimes I_H) U_k^{\sigma }$; moreover, by Lemma 7.2 of \cite{Muhly2016}, 
$Z'^{(k-1)} \otimes I'_{1} \otimes I_H
 =  U_k^{\sigma*} \left( I_1 \otimes Z^{(k-1)} \otimes I_H \right)  U_k^{\sigma}$.
Therefore, since $Z^{(k)} \left( I_1 \otimes Z^{(k-1)} \right)^{-1}= Z_k $,
\begin{multline*}
C_{k}' \otimes I_H 
 = (Z'^{(k)} \otimes I_H )(Z'^{(k-1)} \otimes I'_{1} \otimes I_H)^{-1} \\
 =  U_k^{\sigma *}  (Z^{(k)} \otimes I_H) \left( I_1 \otimes Z^{(k-1)} \otimes I_H \right)^{-1} U_k^{\sigma } 
 =U_k^{\sigma *} (Z_{k} \otimes I_H) U_k^{\sigma},
\end{multline*} 
as desired.
\end{proof}

We already know how to naturally identify $M$ with $\sigma(M)$, $E$ with $E^{\sigma \iota}$, and $\sigma$ with $\jmath$.  
To identify $X$ with $X''$ and $Z$ with $Z''$, that is to show that 
\begin{equation}
\label{ZandS_i2} \omega_kZ_k\omega_k^* = Z''_k,  \quad \text{ and } \quad  \omega_kX_k\omega_k^*  = X_k'',
\end{equation}
note that  for all $k \in \mb{N}_0$, $U_k^\sigma U_k^\iota (\omega_k \otimes I_H)$ is the identity operator on $\tens{E}{k} \otimes_\sigma H$. 
By Lemma \ref{ZandC}, $U_k ^{\iota *} (C_k ' \otimes I_H) U_k ^{\iota } 
= U_k ^{\iota *} U_k^{\sigma *} (Z_k \otimes I_H) U_k^\sigma U_k^\iota 
= \omega_kZ_k\omega_k^* \otimes I_H$, so by the uniqueness of $Z''_k$, $\omega_kZ_k\omega_k^* = Z''_k$.  
Analogously, $\omega_kX_k\omega_k^*  = X_k''$.

\begin{lemma}\label{ZandS}  $Ad(\omega_\infty)(H^\infty(E,Z)) = H^\infty(E^{\sigma \iota}, Z'')$.
\end{lemma}

\begin{proof}
It suffices to show that 
\begin{align}
\label{ZandS_i5} 
\omega_\infty W^Z_\xi \omega^*_\infty
&= W_{\omega_k(\xi)}^{Z''}, \quad \forall k \in \mb{N}, \forall \xi \in \tens{E}{k}, \quad \text{ and }  \\
\label{ZandS_i6} 
\omega_\infty \phiinf{a} \omega^*_\infty
&= \varphi''_{\infty}(\sigma(a)), \quad \forall a \in M. 
\end{align} 
By straightforward computation with equation \eqref{ZandS_i2}, 
$\omega_i Z^{(i,i-k)} \omega_i^*= Z''^{(i,i-k)}$ when  $i \geq k \geq 0$, so $\omega_\infty D_k^Z \omega_\infty^* = D_k^{Z''}$.  
Since $\omega_\infty T_\xi \omega_\infty^*= T_{\omega_k(\xi)}$, equation \eqref{ZandS_i5} follows.
Equation \eqref{ZandS_i6} is simply the case when $k = 0$.
\end{proof}

The following theorem gives two commutant results.  
The first plays a role at the crucial step of the weighted Nevanlinna-Pick theorem in the next section.  
The second is a weighted double commutant theorem.

\begin{theorem} \label{double_comm_theorem}
In the notation established,
\begin{enumerate}[label=(\arabic*),ref=(\arabic*)]
\item \label{double_comm_theorem_i1} $\left( \iota^{\ms{F}(E^\sigma)}(H^\infty(E^\sigma, Z')) \right)'  = \pi^\sigma(H^\infty(E,Z))$, and
\item \label{double_comm_theorem_i2} $\left( \sigma^{\ms{F}(E)} \left( H^\infty(E,Z) \right) \right)'' = \sigma^{\ms{F}(E)} \left( H^\infty(E,Z) \right)$.
\end{enumerate}
\end{theorem}

\begin{proof}
By Theorem 7.7 in \cite{Muhly2016}, 
\begin{equation}
\label{double_comm_theorem_e1} \left( \sigma^{\F{E}} \left( H^\infty \left(E, Z \right) \right)  \right)' 
 = \rho^\sigma \left( H^\infty \left(E^\sigma, Z' \right) \right).
\end{equation}
That same theorem, now applied to $E^\sigma$, gives 
$\left( \iota^{\F{E^\sigma}} \left( H^\infty \left(E^\sigma, Z' \right) \right)  \right)' 
= \rho^\iota \left( H^\infty\left(E^{\sigma \iota}, Z'' \right)  \right)$.
Now $\rho^\iota \circ Ad(\omega_\infty) = \pi^\sigma $ since $U_\infty^\sigma U_\infty^\iota \left(\omega_\infty \otimes I_H \right)$ is the identity operator on $\F{E} \otimes_\sigma H$, so by  Lemma \ref{ZandS},
\begin{equation*}
 \left( \iota^{\F{E^\sigma}} \left( H^\infty \left(E^\sigma, Z' \right) \right)  \right)' 
=  \rho^\iota \left( Ad(\omega_\infty)  \left( H^\infty(E,Z) \right)  \right) = \pi^\sigma \left( H^\infty(E,Z) \right),
\end{equation*}
which gives part \ref{double_comm_theorem_i1}.  
Since $\rho^\sigma = Ad(U_\infty^\sigma) \circ \iota^{\ms{F}(E^\sigma)}$, part \ref{double_comm_theorem_i1} now implies that 
\begin{align*}
\big( \rho^\sigma \left( H^\infty \left(E^\sigma, Z' \right) \right) \big)'
& = Ad(U_\infty^\sigma) \left(  \left( \iota^{\F{E^\sigma}} \left( H^\infty \left(E^\sigma, Z' \right) \right) \right)' \right) \\
& = Ad(U_\infty^\sigma) \left( \pi^\sigma \left( H^\infty(E,Z) \right) \right).
\end{align*}  
Since $Ad(U_\infty^\sigma) \circ \pi^\sigma =  \sigma^{\ms{F}(E)}$  equation \eqref{double_comm_theorem_e1} gives
\begin{multline*}
\left( \sigma^{\ms{F}(E)} \left( H^\infty(E,Z) \right) \right)'' 
= \big( \rho^\sigma \left( H^\infty \left(E^\sigma, Z' \right) \right) \big)' \\
= Ad(U_\infty^\sigma) \left( \pi^\sigma \left( H^\infty(E,Z) \right) \right) 
= \sigma^{\ms{F}(E)} \left( H^\infty(E,Z) \right),
\end{multline*}
which completes the proof of part \ref{double_comm_theorem_i2}.
\end{proof}

The following corollary to Theorem \ref{double_comm_theorem}\ref{double_comm_theorem_i1} will be of use in the next section.
For $s, t \in \mb{N}$, $H^{(s)}$ denotes the direct sum of $s$ copies of $H$, $M_{s \times t}(\ms{B}(H))$ denotes the collection of $s \times t$ matrices with entries in $\ms{B}(H)$, identified with $\ms{B}(H^{(t)}, H^{(s)})$ in the usual way, and $M_s(\ms{B}(H))$ denotes the collection of $s \times s$ matrices with entries in $\ms{B}(H)$.  
We define $\iota^{(s)}: \sigma(M)' \to M_s(\ms{B}(H))$ to be the direct sum of $s$ copies of  $\iota$.
Define $\pi^\sigma_{s \times t}: M_{s \times t}\left(\ms{L}(\F{E}) \right) \to M_{s \times t} \left( \ms{B}(\F{E^\sigma} \otimes_\iota H) \right)$ so that 
\begin{equation*}
\pi^\sigma_{s \times t} \left( \left. [ Y_{af} ]_{a = 1}^s \right. _{f = 1}^t \right) = \left. [ \pi^\sigma(Y_{af}) ]_{a = 1}^s \right. _{f = 1}^t.
\end{equation*} 
We identify, as usual, $\F{E^\sigma} \otimes_{\iota^{(s)}} H^{(s)}$ with $ \left(\F{E^\sigma} \otimes_{\iota} H \right)^{(s)}$.

\begin{corollary} \label{double_comm_corollary} 
Let $s, t \in \mb{N}$.  
The image of $M_{s \times t}\left(H^\infty \left(E, Z \right) \right)$ under 
$\pi^\sigma_{s \times t}$ is precisely the collection of elements $A \in M_{s \times t} \left( \ms{B}(\F{E^\sigma} \otimes_\iota H) \right)$ such that for every $S \in H^\infty \left(E^\sigma, Z' \right)$,
\begin{equation*}
 A \circ  (\iota^{(t)})^{\F{E^\sigma}} (S) =  (\iota^{(s)})^{\F{E^\sigma}} (S)   \circ A .
\end{equation*}
\end{corollary}

\section{Weighted Nevanlinna-Pick Interpolation} \label{wnpChap}

In this section we give our primary theorem, a weighted $W^*$-version of the classic Nevanlinna-Pick interpolation theorem.  
While the key component of its proof, the weighted commutant lifting theorem, is already at our disposal, we begin this section with a series of technical lemmas that further facilitate the proof.  
Let $\sigma: M \to \ms{B}(H)$ be a fixed faithful, normal, unital $*$-homomorphism for a Hilbert space $H$, and let $Z$ be a sequence of  weights associated with $X$.

Let us establish some notation.  
If $\mf{z} \in \mb{D}(X, \sigma)$, then the map 
$\Phi_\mf{z}: \sigma(M)' \to \sigma(M)'$, defined for $A \in \sigma(M)'$ by 
$\Phi_\mf{z}(A) = \sum_{k = 1}^\infty \mf{z}^{(k)} (X_k \otimes A) \mf{z}^{(k)*}$
is completely positive, linear, and ultraweakly continuous by Lemma 4.4 of \cite{Muhly2016};  in addition,  $\Vert \Phi_\mf{z} \Vert < 1$.
Therefore, in the Banach algebra of completely bounded maps on $\sigma(M)'$, $\sum_{j = 0}^\infty \Phi_{\mf{z}}^j$ converges in norm.  
In fact, 
\begin{equation}
\label{Phi_z_lemma}
\sum_{j = 0 }^\infty \Phi_\mf{z}^j (A) 
 = \sum_{k = 0 }^\infty \mf{z}^{(k)} (R_k^2 \otimes A) \mf{z}^{(k)*}, \qquad A \in \sigma(M)'
\end{equation}
by Theorem 4.5 of \cite{Muhly2016}, where the summation on the left-hand side of the equation is with respect to the norm topology on $\sigma(M)'$, and the summation on the right-hand side is with respect to the ultraweak topology on  $\sigma(M)'$.

Let $\mf{z} \in \mc{I}(\sigma^E \circ \varphi, \sigma)$.  
For $k \in \mb{N}_0$,  $\left( \mcZ{k} \right)^*$ belongs to  $\varphi_{k}(M)^c$.  
Thus $\left( \left( \mcZ{k} \right)^* \otimes I_H \right) \mf{z}^{(k)*}$ belongs to $\left(\tens{E}{k} \right)^\sigma$, and we define $c^Z_\mf{z}(k)= c_\mf{z}(k)$ to be its image in $\tens{(E^\sigma)}{k}$ under the isomorphism $\Lambda_k^{\sigma *}$.  
For $\mf{w,z} \in \mb{D}(X, \sigma)$ and $A \in \sigma(M)'$, we have
$\left\langle c_\mf{w}(k), A \cdot c_\mf{z}(k) \right\rangle
= \mf{w}^{(k)} (R_k^2 \otimes A) \mf{z}^{(k)*}$.
Define  the \emph{$Z$-Cauchy kernel at $\mf{z}$} to be the tuple,
\begin{equation*} 
c_\mf{z}^Z = c_\mf{z}= \left( c_\mf{z}(k) \right)_{k = 0}^\infty.
\end{equation*}

\begin{lemma}\label{weightedCauchy} 
If $\mf{z} \in \mb{D}(X, \sigma)$, then $c_\mf{z} \in \ms{F}(E^\sigma)$.
\end{lemma}

\begin{proof} 
By equation \eqref{Phi_z_lemma}, $\sum_{k = 0 }^\infty \mf{z}^{(k)} (R_k^2 \otimes I_H) \mf{z}^{(k)*}$ converges in $\sigma(M)'$, so for any $N \in \mb{N}$,
$\sum_{k = 0}^N \left\langle c_\mf{z}(k), c_\mf{z}(k) \right\rangle
= \sum_{k = 0}^N \mf{z}^{(k)} (R_k^2 \otimes I_H) \mf{z}^{(k)*} 
\leq \sum_{k = 0 }^\infty \mf{z}^{(k)} (R_k^2 \otimes I_H) \mf{z}^{(k)*}$.
Since $c_\mf{z}(k) \in \tens{(E^\sigma)}{k}$ for each $k \in \mb{N}_0$, the result follows.
\end{proof}

Let $\mf{w,z} \in \mb{D}(X, \sigma)$.  
We define $\ms{K}(\mf{w},\mf{z}): \sigma(M)' \to \ms{B}(H)$  by
\begin{equation*}
\ms{K}(\mf{w},\mf{z})(A) 
= \left\langle c_\mf{w}, A \cdot c_\mf{z} \right\rangle
= \sum_{k = 0}^\infty \mf{w}^{(k)} (R_k^2 \otimes A) \mf{z}^{(k)*}, \quad A \in \sigma(M)'.
\end{equation*} 
For $s \in \mb{N}$, let $\ms{K}^{(s)}(\mf{w}, \mf{z}):  \sigma(M)' \to M_s(\ms{B}(H))$ be the map defined at $A \in \sigma(M)'$ by $\ms{K}^{(s)}(\mf{w}, \mf{z})(A) = \iota^{(s)}(\ms{K}(\mf{w}, \mf{z})(A))$ where, as in the preceding section, $\iota^{(s)}$ is the direct sum of $s$ copies of  $\iota$.

\begin{remark}
While we will not explore the consequences here, it is readily shown that 
$\ms{K}: \mb{D}(X, \sigma) \times \mb{D}(X, \sigma) \to CB_* (\sigma(M)', \ms{B}(H))$ is a ``normal completely positive kernel'' in the sense of \cite{Barreto2004} and \cite{Meyer2010}, independent of the choice of $Z$, where $CB_* (\sigma(M)', \ms{B}(H))$ denotes the completely bounded ultraweakly continuous maps from $\sigma(M)'$ to $\ms{B}(H)$.  
As such, by Proposition 41 of  \cite{Meyer2010}, $\ms{K}$ has an associated ``Reproducing Kernel $W^*$-Correspondence'', a $W^*$-analogue of the classical Reproducing Kernel Hilbert Space; in fact $\ms{K}$, which we may refer to  as the \emph{$(X, \sigma)$-Szeg\"{o} kernel},  is a weighted $W^*$-version of the classic Szeg\"{o} kernel.
\end{remark}

Recall that $\widehat{I_H} = \left( \delta_{i = 0} \ I_H \right)_{i = 0}^\infty$ belongs to $\ms{F}(E^\sigma)$ since $I_H$ is the multiplicative identity of $\sigma(M)' = \tens{(E^\sigma)}{0}$.
If $\mf{z} \in \mb{D}(X, \sigma)$, then  $c_\mf{z}$ belongs to $\ms{F}(E^\sigma)$ by Lemma \ref{weightedCauchy}.  
We write $L_{I}$ and $L_\mf{z}$ for the insertion operators $L^H_{\widehat{I_H}}$ and $L^H_{c_\mf{z}}$, respectively.

\begin{lemma}\label{rhofacts2}
For $\mf{z} \in \mb{D}(X, \sigma)$ and $Y \in H^\infty(E, Z)$,
\begin{enumerate}[label=(\arabic*),ref=(\arabic*)]
\item \label{rhofacts2_i1} $L_{\mf{z}}^*  \pi^\sigma(Y)  L_{I}  L_{\mf{z}}^* = L_{\mf{z}}^*  \pi^\sigma(Y)$,
\item \label{rhofacts2_i2} $L_{\mf{z}}^*  \pi^\sigma(Y)  L_{I} = \widehat{Y}(\mf{z})$, and
\item \label{rhofacts2_i3} $L_{\mf{z}}^* \pi^\sigma(Y) = \widehat{Y}(\mf{z})  L_{\mf{z}}^*$.
\end{enumerate}
\end{lemma}

\begin{proof}
Let us show that parts \ref{rhofacts2_i1} and \ref{rhofacts2_i2} hold when $Y = \varphi_\infty(a)$ when $a \in M$.
Since $\pi^\sigma(\phiinf{a}) = I'_\infty \otimes \sigma(a)$, $\langle c_\mf{z}, \widehat{I_H} \rangle = I_H$, and $\sigma(a)
=( \sigma \times \mf{z}) (\phiinf{a})$, 
\begin{multline*}
L_{\mf{z}}^* \pi^\sigma (\phiinf{a}) L_{I} 
= L_{\mf{z}}^* \left( I'_\infty \otimes \sigma(a) 
\right) L_{I} 
= L_{\mf{z}}^*  L_{I}  \sigma(a) \\
= \iota  \langle c_\mf{z}, \widehat{I_H} \rangle  \sigma(a) 
= \sigma(a)
=( \sigma \times \mf{z}) (\phiinf{a}),
\end{multline*} 
which gives part \ref{rhofacts2_i2}.  
Part \ref{rhofacts2_i1} now follows from the fact that if $\eta \in \F{E^\sigma}$ and $x \in H$, then $L_{\mf{z}}^* L_{\eta} = \langle c_\mf{z}, \eta \rangle$ belongs to $\sigma(M)'$, so using part \ref{rhofacts2_i2},
\begin{multline*}
L_{\mf{z}}^* \pi^\sigma (\phiinf{a}) L_{I}   L_{\mf{z}}^* (\eta \otimes x)
= \sigma(a) L_{\mf{z}}^* L_{\eta}x
= L_{\mf{z}}^* L_{\eta}\sigma(a)x \\
=  L_{\mf{z}}^*(I'_\infty \otimes \sigma(a)) L_\eta x
=  L_{\mf{z}}^* \pi^\sigma(\phiinf{a})( \eta \otimes x).
\end{multline*}

Now we show that parts \ref{rhofacts2_i1} and \ref{rhofacts2_i2} hold for $Y = W_\xi$ when $\xi \in E$.   
The $i^{th}$ diagonal entry of the diagonal operator $\pi^\sigma(D_1)$ is zero when $i = 0$ and is otherwise $U_i^{\sigma*}( Z_i \otimes I_H)U_i^{\sigma} = C'_i \otimes I_H$  (Lemma \ref{ZandC}).
It follows that
\begin{align}
\notag
\pi^\sigma(W_\xi) 
&= \pi^\sigma(D_1) \pi^\sigma(T_\xi) \\
\label{rhofacts1_e-2}
& = \left[ 
\begin{matrix}
0 & 0   & 0 & \\ 
(C_1' \otimes I_H)\omega\xi & 0 & 0 &  \\ 
0 & (C_2' \otimes I_H)(I'_{1} \otimes \omega \xi) & 0 &  \\ 
0 & 0 & (C_3' \otimes I_H)(I'_{2} \otimes \omega \xi) & \ddots \\ 
&  & \ddots & \ddots
\end{matrix}
\right].
\end{align}
If $L_i$ denotes the insertion operator $L^H_{c_\mf{z}(i)}$ for each $i \in \mb{N}_0$, then matricially $L_\mf{z}$ is the column $[ L_i ]_{i = 0}^\infty$; hence
$L_{\mf{z}}^*  \pi^\sigma(W_\xi)  L_{I}
 = L_1^*(C'_1 \otimes I_H) (\omega \xi)$.
By Lemma \ref{ZandC}, $C'_1 \otimes I_H  =  U_1^{\sigma *}  (Z_1 \otimes I_H) U_1^\sigma$.  
Since $U_1^\sigma \circ ( \omega \xi ) = L_\xi$ and $U_1^\sigma(c_\mf{z}(1) \otimes y) = c_\mf{z}(1)  (y)$, when $x, y \in H$,
\begin{multline*}
\left\langle L_{\mf{z}}^*  \pi^\sigma(W_\xi)  L_{I} x, y \right\rangle
= \left\langle (C'_1 \otimes I_H) (\omega \xi )x, c_\mf{z}(1) \otimes y\right\rangle \\
= \left\langle (Z_1 \otimes I_H) L_\xi x, \left( (Z_1^{-1})^* \otimes I_H \right) \mf{z}^{*} y \right\rangle \\
= \left\langle \mf{z} L_\xi x, y \right\rangle
= \left\langle (\sigma \times \mf{z})(W_\xi) x, y \right\rangle.
\end{multline*}
We deduce that part \ref{rhofacts2_i2} holds for $Y = W_\xi$.
Let us compare the row matrices for the operators on either side of the equality in part  \ref{rhofacts2_i1}. 
By part  \ref{rhofacts2_i2},
$L_{\mf{z}}^* \pi^\sigma(W_\xi)  L_{I}  L_{\mf{z}}^*  
= \left[  \mf{z}L_\xi  L_j^*  \right]_{j = 0}^\infty$.
On the other hand, by equation \eqref{rhofacts1_e-2},
$L_{\mf{z}}^*  \pi^\sigma(W_\xi)  
 = \left[ L_{j+1}^* (C_{j+1}' \otimes I_H) 
\left( I'_{j}  \otimes \omega \xi \right) \right]_{j = 0}^\infty$ .
Therefore to obtain part \ref{rhofacts2_i1},  it suffices to show that for every $j \in \mb{N}_0$, 
\begin{equation}
\label{rhofacts1_e0}
\mf{z}L_\xi  L_j^*  =  L_{j+1}^* (C_{j+1}' \otimes I_H) 
\left( I'_{j}  \otimes \omega \xi \right).
\end{equation}
For any $k \in \mb{N}_0$, $U_{k}^\sigma L_{k} =
\left( (\mcZ{k})^* \otimes I_H \right) \mf{z}^{(k)*}$ by definition of $c_\mf{z}(k)$, so 
\begin{equation}
\label{rhofacts1_e1}
 L^*_{k} U_{k}^{\sigma *} =
 \mf{z}^{(k)} \left( (\mcZ{k}) \otimes I_H \right), \quad k \in \mb{N}_0.
\end{equation}
Since $Z^{(j+1)} = Z_{j+1}(I_1 \otimes Z^{(j)})$ and $\mf{z}^{(j+1)} = \mf{z}(I_1 \otimes \mf{z}^{(j)})$, two applications of equation \eqref{rhofacts1_e1}, first when $k = j+1$ and then when $k = j$, yield
\begin{multline}
\label{rhofacts1_e2}
 L_{j+1}^* U_{j+1}^{\sigma *}\left( Z_{j+1} \otimes I_H \right) 
 = \mf{z}^{(j+1)}  \left( (\mcZ{j+1}) \otimes I_H \right) \left( Z_{j+1} \otimes I_H \right) \\
 = \mf{z} \Big( I_1 \otimes \mf{z}^{(j)}\left( (\mcZ{j}) \otimes I_H \right)  \Big)  
 =  \mf{z} \left( I_1 \otimes  L_j^* U_j^{\sigma*}\right) 
\end{multline} 
As observed in Section 2.2, $U_{j+1}^\sigma
\left( I_{j}'  \otimes \omega \xi \right)  
= L_\xi^{(\tens{E}{j} \otimes_\sigma H)} U^\sigma_j $.  
Thus by Lem\-ma \ref{ZandC} and  equation \eqref{rhofacts1_e2},
\begin{multline*}
 L_{j+1}^* (C_{j+1}' \otimes I_H) 
\left( I'_{j}  \otimes \omega \xi \right)  
 =  L_{j+1}^* U_{j+1}^{\sigma  *} \left( Z_{j+1} \otimes I_H \right) U_{j+1}^\sigma
\left( I'_{j}  \otimes \omega \xi \right)  \\
 = \mf{z}(I_1 \otimes L_j^* U_j^{\sigma *}) L_\xi^{(\tens{E}{j} \otimes_\sigma H)} U^\sigma_j 
 = \mf{z} L_\xi L_j^*,
\end{multline*}
noting that the final equality follows readily from an elementwise computation.  
This gives equation \eqref{rhofacts1_e0}, so part \ref{rhofacts2_i1} holds for $Y = W_\xi$ when $\xi \in E$.

To show that part \ref{rhofacts2_i1} holds for arbitrary operators in $H^\infty(E,Z)$, 
suppose, inductively, that $k \in \mb{N}$ exists such that part \ref{rhofacts2_i1} holds whenever $Y = W_\xi$ for $\xi \in \tens{E}{j}$ with $0 \leq j \leq k$.
For $\xi \in \tens{E}{k}$ and $\eta \in E$, $W_{\xi \otimes \eta} = W_\xi \otimes W_\eta$, so
\begin{multline*}
L_{\mf{z}}^*  \pi^\sigma(W_{\xi \otimes \eta})  L_{I}  L_{\mf{z}}^* 
= L_{\mf{z}}^* \pi^\sigma(W_\xi) \pi^\sigma(W_\eta)  L_{I}  L_{\mf{z}}^* \\
= L_{\mf{z}}^*  \pi^\sigma(W_\xi) L_{I}  L_{\mf{z}}^*\pi^\sigma(W_\eta)  L_{I} L_{\mf{z}}^* 
= L_{\mf{z}}^*  \pi^\sigma(W_\xi) L_{I}  L_{\mf{z}}^*\pi^\sigma(W_\eta) \\
= L_{\mf{z}}^*  \pi^\sigma(W_\xi) \pi^\sigma(W_\eta)
= L_{\mf{z}}^*  \pi^\sigma(W_{\xi \otimes \eta}).
\end{multline*}
Part (1) follows from routine approximation arguments.
Towards obtaining part \ref{rhofacts2_i2} in the general case, define the  linear and ultraweakly continuous map $\tau: H^\infty(E, Z) \to \ms{B}(H)$ at $Y \in H^\infty(E, Z)$ by $\tau(Y) = L_{\mf{z}}^*  \pi^\sigma(Y)  L_{I}$.   
By part  \ref{rhofacts2_i1}, for $Y_1, Y_2 \in H^\infty(E,Z)$, 
\begin{equation*}
\tau(Y_1 Y_2) 
= L_{\mf{z}}^*  \pi^\sigma(Y_1) \pi^\sigma(Y_2)  L_{I} 
= L_{\mf{z}}^*  \pi^\sigma(Y_1) L_{I} L_{\mf{z}}^*\pi^\sigma(Y_2)  L_{I}
= \tau(Y_1) \tau(Y_2).
\end{equation*}
Thus $\tau$ is multiplicative.  
As both $\tau$ and $(\sigma \times \mf{z})$ are  linear, multiplicative, and ultraweakly continuous  and they agree on elements of the form $Y = \phiinf{a}$ when $a \in M$ and $Y = W_\xi$ when $\xi \in E$, it follows that $\tau = (\sigma \times \mf{z})$, which gives part \ref{rhofacts2_i2}.
Finally, parts \ref{rhofacts2_i1} and \ref{rhofacts2_i2} together imply part  \ref{rhofacts2_i3}.
\end{proof}

Our final technical lemma concerns a weighted creation operator $W_\xi^{Z'} = W'_\xi \in H^\infty(E^\sigma, Z')$.  
Recall that if $\mf{z} \in \mb{D}(X, \sigma)$, then $\mf{z}^*$ is an element in $E^\sigma$ and for any $\xi \in E^\sigma$,  $\langle \xi, \mf{z}^* \rangle = \xi^* \mf{z}^*$ is an element of $\sigma(M)'$.

\begin{lemma}\label{iotaWstar}
For $\mf{z} \in \mb{D}(X, \sigma)$, $\xi \in E^\sigma$, and $D \in \sigma(M)'$, 
\begin{equation*}
\left(W'_\xi \right)^* \left( D \cdot c_\mf{z} \right) = \langle D^* \cdot \xi, \mf{z}^* \rangle \cdot c_\mf{z}.
\end{equation*}
\end{lemma}

\begin{proof}
First, let us show that the result holds when $D = I_H$.
We have that
$(W'_\xi )^* c_\mf{z}  = 
\left( T_\xi^{(k) *} (Z_{k+1}')^* \left( c_\mf{z}(k+1) \right)\right)_{k = 0}^\infty$ where $T_\xi^{(k)}$ maps $\eta \in \tens{(E^\sigma)}{k}$ to $\xi \otimes \eta \in \tens{(E^\sigma)}{k+1}$.  
Also,
 $\langle  \xi, \mf{z}^* \rangle \cdot c_\mf{z}  = 
\left( \xi^* \mf{z}^* \cdot  c_\mf{z}(k) \right)_{k = 0}^\infty$, so it suffices to show that for every $k \in \mb{N}_0$ and $\eta \in \tens{(E^\sigma)}{k}$,
\begin{equation}
\label{iotaWstar.e2}
\left\langle T_\xi^{(k) *} (Z_{k+1}')^* \left( c_\mf{z}(k+1) \right), \eta \right\rangle = 
\Big\langle\xi^* \mf{z}^* \cdot  c_\mf{z}(k), \eta \Big\rangle.
\end{equation}
If $\zeta \in \tens{(E^\sigma)}{k+1}$ and $x \in H$, then by the definition of $Z'$ in terms of $C$,
\begin{multline*}
\Lambda_{k+1}^\sigma \left( Z'_{k+1} \zeta \right) x 
= U_{k+1}^\sigma (Z'_{k+1} \zeta \otimes x) \\
= (C_{k+1} \otimes I_H) U_{k+1}^\sigma (\zeta \otimes x) 
= \left( C_{k+1} \otimes I_H \right)  \Lambda_{k+1}^\sigma (\zeta) x.
\end{multline*}
Thus $\Lambda_{k+1}^\sigma \left( Z'_{k+1} \zeta \right) = \left( C_{k+1} \otimes I_H \right)  \Lambda^\sigma_{k+1} (\zeta)$.  
In particular, when $\zeta = \xi \otimes \eta$, 
\begin{multline}
\label{iotaWstar.e4}
\left\langle T_\xi^{(k) *} (Z_{k+1}')^* \left( c_\mf{z}(k+1) \right), \eta \right\rangle 
= \left\langle   c_\mf{z}(k+1) , Z'_{k+1} (\xi \otimes \eta) \right\rangle \\
=  \mf{z}^{(k+1) }\left(\mcZ{k+1} \otimes I_H \right) \Lambda_{k+1}^\sigma  \left( Z'_{k+1} (\xi \otimes \eta) \right) \\
=  \mf{z}^{(k+1) }\left(\mcZ{k+1} C_{k+1} \otimes I_H \right) \Lambda_{k+1}^\sigma   (\xi \otimes \eta) \\
=  \mf{z}^{(k+1) }\left(\mcZ{k} \otimes I_1 \otimes I_H \right) \Lambda_{k+1}^\sigma   (\xi \otimes \eta).
\end{multline}
On the other hand, 
\begin{multline}
\label{iotaWstar.e5}
\Big\langle \xi^* \mf{z}^* \cdot  c_\mf{z}(k), \eta \Big\rangle
=  \mf{z}^{(k)}\left(\mcZ{k} \otimes I_H \right)  \Lambda_k^\sigma \left(  \mf{z} \xi  \cdot  \eta \right) \\
=  \mf{z}^{(k)}\left(\mcZ{k} \otimes I_H \right)  \left( I_{k} \otimes \mf{z}\xi \right) \Lambda_k^\sigma \left(   \eta \right) \\
=  \mf{z}^{(k+1)} \left(\mcZ{k} \otimes I_1 \otimes I_H \right) \Lambda_{k+1}^\sigma \left(  \xi \otimes \eta \right).
\end{multline}
From equations \eqref{iotaWstar.e4} and \eqref{iotaWstar.e5}, we obtain equation \eqref{iotaWstar.e2}.  
Thus, the conclusion of the lemma holds when $D = I_H$.
For arbitrary $D$, we observe that $(W'_\xi )^*  \varphi'_\infty (D) 
 = (W'_{D^* \cdot \xi} )^*$.  
Therefore, by our first case, 
$(W'_\xi )^* \left( D \cdot c_\mf{z} \right) 
=  (W'_{D^* \cdot \xi})^* \cdot c_\mf{z} 
= \langle D^* \cdot \xi, \mf{z}^* \rangle \cdot c_\mf{z}$,
as desired.
\end{proof}

We are ready to present our main result, a weighted Nevanlinna-Pick interpolation theorem.  
We phrase and prove a more-general, matricial version of the theorem that, aside from increased notational complexity, poses no additional difficulty; the main result, that which generalizes the classic Nevanlinna-Pick interpolation theorem, occurs when we take $s = t = 1$ in the statement below.  
While we make necessary adjustments for the weights, our proof mirrors that of the unweighted Nevanlinna-Pick interpolation result given as Theorem 5.3 in \cite{Muhly2004a}.

\begin{theorem}[Weighted Nevanlinna-Pick Interpolation]\label{WeightedNevanlinnaPick}
Fix $s, t \in \mb{N}$.  
Let $k \in \mb{N}$.  
Choose $\{\mf{z}_i \}_{i=1}^k \subseteq \mb{D}(X, \sigma)$ and two collections, $\{ B_i \}_{i=1}^k \subseteq M_s \left(\ms{B}\left( H \right) \right)$ and $\{ F_i \}_{i=1}^k \subseteq M_{s \times t} \left(\ms{B}\left( H \right) \right)$. 
Define $\ms{A}: M_k(\sigma(M)') \to M_k \left(M_s \left(\ms{B}\left( H \right) \right)\right)$ at $[A_{ij}]_{i,j = 1}^k \in  M_k(\sigma(M)')$ by
\begin{multline*} 
\ms{A}\left( [A_{ij}]_{i,j = 1}^k  \right) = \\
 \left[B_i  \cdot \ms{K}^{(s)}(\mf{z}_i,\mf{z}_j)(A_{ij})  \cdot B_j^* - F_i \cdot \ms{K}^{(t)}(\mf{z}_i, \mf{z}_j)(A_{ij}) \cdot F_j^*\right]_{i,j = 1}^k.
\end{multline*}
Then the following conditions are equivalent:
\begin{enumerate}[label=(\arabic*),ref=(\arabic*)]
\item  \label{WeightedNevanlinnaPick_i1} 
For every  $n \in \mb{N}$ and for any $ \left\{ h_{pi} \mid 1 \leq i \leq k, \ 1 \leq p \leq n \right\} \subseteq H^{(s)}$ and  $ \left\{A_{pi} \mid 1 \leq i \leq k, \ 1 \leq p \leq n \right\} \subseteq \sigma(M)'$,
\begin{equation*}
\left\Vert \sum_{p = 1}^n \left(\sum_{i = 1}^k A_{pi} \cdot c_{\mf{z}_i} \otimes F_i^*h_{pi} \right) \right\Vert^2 \leq \left\Vert \sum_{p = 1}^n \left( \sum_{i = 1}^k A_{pi} \cdot c_{\mf{z}_i} \otimes B_i^*h_{pi} \right) \right\Vert^2,
\end{equation*}
where on the left-hand side of the inequality, the norm occurs in the Hilbert space $\F{E^\sigma} \otimes_{\iota^{(t)}} H^{(t)}$ and on the right in $\F{E^\sigma} \otimes_{\iota^{(s)}} H^{(s)}$;
\item \label{WeightedNevanlinnaPick_i3} $\ms{A}$ is completely positive; and
\item \label{WeightedNevanlinnaPick_i4} 
there exists $Y =  \left. [ Y_{af} ]_{a = 1}^s \right. _{f = 1}^t \in M_{s \times t} \left(H^\infty(E,Z)\right)$ such that $\Vert Y \Vert \leq 1$ and
\begin{equation*}
B_i  \left. \left[ \widehat{Y_{af}}(\mf{z}_i)  \right]_{a = 1}^s \right. _{f = 1}^t = F_i
\end{equation*}
for every $i \in \mb{N}$ with $1 \leq i \leq k$.
\end{enumerate}
\end{theorem}

\begin{remark}
The map $\ms{A}$ defined in the statement of Theorem \ref{WeightedNevanlinnaPick} does not depend on $Z$, so if the theorem is satisfied for \emph{one} sequence of weights associated with $X$ it holds for \emph{every} sequence of weights associated with $X$.
\end{remark}

\begin{proof} 
First let us perform a preliminary computation.  
Let $n \in \mb{N}$ and choose two collections  $ \left\{ h_{pi} \mid 1 \leq i \leq k, \ 1 \leq p \leq n \right\} \subseteq H^{(s)}$ and  $ \left\{A_{pi} \mid 1 \leq i \leq k, \ 1 \leq p \leq n \right\} \subseteq \sigma(M)'$.
Define $\widetilde{h} : = \left( ( h_{pi} )_{i = 1}^k \right)_{p = 1}^n$ in $((H^{(s)})^{(k)})^{(n)}$.  
Since $\ms{K}(\mf{z}_i, \mf{z}_j)(A_{pi}^*A_{qj}) =  \langle A_{pi} \cdot c_{\mf{z}_i}, A_{qj} \cdot c_{\mf{z}_j}\rangle$,
\begin{align*}
& \left\langle \widetilde{h},  \left[ \left[ F_i \cdot \ms{K}^{(t)}(\mf{z}_i, \mf{z}_j)(A_{pi}^*A_{qj}) \cdot F_j^*\right]_{i,j=1}^k \right]_{p,q = 1}^n ( \widetilde{h} ) \right\rangle\\
& \quad \quad \quad = \sum_{p,q = 1}^n \left(\sum_{i,j =1}^k
\left\langle F_i^* h_{pi},  \iota^{(t)} \left( \ms{K}(\mf{z}_i, \mf{z}_j)(A_{pi}^*A_{qj})  \right)  F_j^*  h_{qj}\right\rangle \right) \\
&  \quad \quad \quad = \left\Vert \sum_{p = 1}^n \left( \sum_{i = 1}^k A_{pi} \cdot c_{\mf{z}_i} \otimes F_i^*h_{pi} \right) \right\Vert^2,
\end{align*}
where the norm occurs in the Hilbert space $\F{E^\sigma} \otimes_{\iota^{(t)} } H^{(t)}$ and the inner product is taken in $((H^{(s)})^{(k)})^{(n)}$.
Since an analogous result follows with each $F_i$ replaced with $B_i$ we compute,
\begin{align}
\notag
& \left\langle \widetilde{h}, \ms{A}_n \left( \left[ [ A^*_{pi}A_{qj} ]_{i,j = 1}^k \right]_{p,q = 1}^n \right) ( \widetilde{h} )  \right\rangle\\*
\notag
& \quad = \left\langle\widetilde{h},  \left[ \left[ B_i  \cdot \ms{K}^{(s)}(\mf{z}_i, \mf{z}_j)(A^*_{pi}A_{qj})  \cdot  B_j^* \right]_{i,j = 1}^k \right]_{p,q = 1}^n ( \widetilde{h} ) \right\rangle \\*
\notag
& \qquad \qquad - \left\langle\widetilde{h},   \left[ \left[ F_i  \cdot \ms{K}^{(t)}(\mf{z}_i, \mf{z}_j)(A^*_{pi}A_{qj})  \cdot F_j^* \right]_{i,j = 1}^k \right]_{p,q = 1}^n ( \widetilde{h} ) \right\rangle\\*
\label{WeightedNevanlinnaPick_e2} 
&  \quad = \left\Vert \sum_{p = 1}^n \left( \sum_{i = 1}^k A_{pi} \cdot c_{\mf{z}_i} \otimes B_i^*h_{pi} \right) \right\Vert^2 - \left\Vert \sum_{p = 1}^n \left( \sum_{i = 1}^k A_{pi} \cdot c_{\mf{z}_i} \otimes F_i^*h_{pi} \right) \right\Vert^2.
\end{align}

Let us show that condition  \ref{WeightedNevanlinnaPick_i3} implies condition  \ref{WeightedNevanlinnaPick_i1}.  
Suppose $\ms{A}$ is completely positive.   
Let  $A : = \left[ \delta_{p = 1} \ [ \delta_{i = 1} \  A_{qj} ]_{i,j = 1}^k \right]_{p,q = 1}^n$ where $\delta_{i = 1}$ is $1$ when $i = 1$ and is otherwise equal to $0$.  
Then $A^*A = \left[ [ A^*_{pi}A_{qj} ]_{i,j = 1}^k \right]_{p,q = 1}^n$
is positive in $M_n(M_k(\sigma(M)'))$.  
Since $\ms{A}_n$ is positive, we have that $\ms{A}_n (A^*A)$ is positive in $\ms{B} \left(((H^{(s)})^{(k)})^{(n)}\right)$.  
Taking $\widetilde{h} : = \left( ( h_{pi} )_{i = 1}^k \right)_{p = 1}^n$ in $ ((H^{(s)})^{(k)})^{(n)}$ and using equation \eqref{WeightedNevanlinnaPick_e2} we have that
\begin{align*}
0 
& \leq \left\langle \widetilde{h}, \ms{A}_n (A^*A) ( \widetilde{h} )  \right\rangle\\
& = \left\Vert \sum_{p = 1}^n \left( \sum_{i = 1}^k A_{pi} \cdot c_{\mf{z}_i} \otimes B_i^*h_{pi} \right) \right\Vert^2 - \left\Vert \sum_{p = 1}^n \left( \sum_{i = 1}^k A_{pi} \cdot c_{\mf{z}_i} \otimes F_i^*h_{pi} \right) \right\Vert^2.
\end{align*}
Thus, condition \ref{WeightedNevanlinnaPick_i3} implies condition  \ref{WeightedNevanlinnaPick_i1}.

Now let us assume condition \ref{WeightedNevanlinnaPick_i1} and show that condition \ref{WeightedNevanlinnaPick_i3} holds.  
Let $n \in \mb{N}$ and choose a collection $\left\{ A_{pi} \mid 1 \leq i \leq k,  \ 1 \leq p \leq n \right\}$ of elements in $\sigma(M)'$.  
An arbitrary element of  $\ms{B} \left( ((H^{(s)})^{(k)})^{(n)} \right)$ may be written $\widetilde{h} = \left( ( h_{pi} )_{i = 1}^k \right)_{p = 1}^n$ for some  $\left\{ h_{pi} \mid 1 \leq i \leq k,  \ 1 \leq p \leq n \right\} \subseteq  H^{(s)}$.   
Since we are assuming condition  \ref{WeightedNevanlinnaPick_i1}, we deduce from equation \eqref{WeightedNevanlinnaPick_e2} that 
\begin{multline*}
 \left\langle \widetilde{h} , \ms{A}_n \left( \left[ [ A^*_{pi}A_{qj} ]_{i,j = 1}^k \right]_{p,q = 1}^n\right) ( \widetilde{h} )  \right\rangle\\
  = \left\Vert \sum_{p = 1}^n \left( \sum_{i = 1}^k A_{pi} \cdot c_{\mf{z}_i} \otimes B_i^*h_{pi} \right) \right\Vert^2 
 - \left\Vert \sum_{p = 1}^n \left( \sum_{i = 1}^k A_{pi} \cdot c_{\mf{z}_i} \otimes F_i^*h_{pi} \right)\right\Vert^2 
\end{multline*}
is nonnegative.  
Therefore $\ms{A}_n \left( \left[ [ A^*_{pi}A_{qj} ]_{i,j = 1}^k \right]_{p,q = 1}^n \right)$ is positive.   
Since an arbitrary positive element in $M_n(M_k(\sigma(M)'))$ can be expressed as a sum of elements of the form $\left[ [ A^*_{pi}A_{qj} ]_{i,j = 1}^k \right]_{p,q = 1}^n$, it follows that 
 $\ms{A}_n$ is positive.  
Therefore, $\ms{A}$ is completely positive; thus condition \ref{WeightedNevanlinnaPick_i1} implies condition \ref{WeightedNevanlinnaPick_i3}.

Let us pause for some notational simplifications.  
If $1 \leq i \leq k$, then $L_i$, $L_i^s$, and $L_i^t$ denote the insertion operators $L_{c_{\mf{z}_i}}^H$, $L_{c_{\mf{z}_i}}^{H^{(s)}}$, and $L_{c_{\mf{z}_i}}^{H^{(t)}}$, respectively.  
Observe that $L_i$ is $L_{\mf{z}_i}$ in  the notation of Lemma \ref{rhofacts2}.
Also, $L_I$, $L_I^s$, and $L_I^t$ denote the insertion operators $L_{\widehat{I_H}}^H$, $L_{\widehat{I_H}}^{H^{(s)}}$, and $L_{\widehat{I_H}}^{H^{(t)}}$, respectively.
Since the identity operators on $\tens{E}{s}$ and $\tens{E}{t}$ do not occur in our computations for the remainder of the proof, we temporarily write $I_s$ in place of  $I_{H^{(s)}}$ and $I_t$ in place of  $I_{H^{(t)}}$.  
Note that when $A \in \sigma(M)'$, $1 \leq i \leq k$, and $h \in H^{(s)}$,
\begin{equation} 
\label{WeightedNevanlinnaPick_e1} 
A \cdot c_{\mf{z}_i} \otimes h 
= \left(\varphi'_\infty(A) \otimes I_{s} \right) (c_{\mf{z}_i} \otimes h )
= (\iota^{(s)})^{\ms{F}(E^\sigma)} (\varphi'_\infty(A)) L_i^s h.
\end{equation}
A similar result holds with $t$ in place of $s$.  

Suppose $Y = [ Y_{af} ]$ satisfies the properties in condition \ref{WeightedNevanlinnaPick_i4}, using  $a \in \{1, \ldots, s\}$ to indicate indicate rows and $f \in \{1, \ldots, t\}$ to indicate columns for an $s \times t$ matrix.
We show condition \ref{WeightedNevanlinnaPick_i1}.
Let $i \in \{1, \ldots, k\}$.  
Since $L_i^s$ and $L_i^t$ are diagonal operators with $L_i$ in each diagonal entry, 
by Lemma  \ref{rhofacts2}\ref{rhofacts2_i3}, 
\begin{multline*}
F_i L_i^{t*}   
= B_i  \left[ \widehat{Y_{af}}(\mf{z}_i)  \right] L_i^{t*}
= B_i  \left[   \widehat{Y_{af}}(\mf{z}_i) L_i^*  \right] \\
= B_i   \left[ L_i^* \pi^\sigma(Y_{af}) \right]
= B_i  L_i^{s*}  \pi^\sigma_{s \times t} \left( Y\right).
\end{multline*}
Taking adjoints, we have $L_i^tF_i^* = \left( \pi_{s \times t}^\sigma (Y) \right)^* L_i^s B_i^*$. 
Thus by equation \eqref{WeightedNevanlinnaPick_e1} and Corollary \ref{double_comm_corollary}, for $1 \leq p \leq n$,
\begin{align*}
A_{pi} \cdot c_{\mf{z}_i} \otimes F_i^*h_{pi} 
& = (\iota^{(t)})^{\ms{F}(E^\sigma)} (\varphi'_{\infty}(A_{pi})) L_i^t F_i^*h_{pi} \\
& =(\iota^{(t)})^{\ms{F}(E^\sigma)} (\varphi'_\infty (A_{pi})) \left(\pi^\sigma_{s \times t} \left( Y\right) \right)^*  L_i^s B_i^* h_{pi} \\
& = \left(\pi^\sigma_{s \times t} \left( Y\right) \right)^*   (\iota^{(s)} )^{\ms{F}(E^\sigma)} (\varphi'_\infty(A_{pi})) L_i^s B_i^* h_{pi} \\
& =  \left(\pi^\sigma_{s \times t} \left( Y\right) \right)^*   \left( A_{pi} \cdot c_{\mf{z}_i} \otimes B_i^*h_{pi} \right).
\end{align*}
Since $\Vert Y \Vert \leq 1$ and  $\pi^\sigma_{s \times t}$ is a linear isometry,
\begin{multline*}
\left\Vert \sum_{p = 1}^n \left( \sum_{i = 1}^k A_{pi} \cdot c_{\mf{z}_i} \otimes F_i^*h_{pi} \right) \right\Vert^2 \\
 \leq \left\Vert \pi^\sigma_{s \times t} \left( Y\right) \right\Vert^2 \left\Vert  \sum_{p = 1}^n \left( \sum_{i = 1}^k A_{pi} \cdot c_{\mf{z}_i} \otimes B_i^*h_{pi} \right) \right\Vert^2\\
  \leq \left\Vert \sum_{p = 1}^n \left( \sum_{i = 1}^k A_{pi} \cdot c_{\mf{z}_i} \otimes B_i^*h_{pi} \right) \right\Vert^2.
\end{multline*}
Therefore condition \ref{WeightedNevanlinnaPick_i4} implies condition  \ref{WeightedNevanlinnaPick_i1}.

Finally, we show that condition  \ref{WeightedNevanlinnaPick_i1} implies condition  \ref{WeightedNevanlinnaPick_i4}. 
Let $ J_B$ be the norm-closure of the subspace of $\F{E^\sigma} \otimes_{\iota^{(s)}} H^{(s)}$ comprised of the elements of the form
\begin{equation*}
\sum_{p = 1}^n \left( \sum_{i = 1}^k A_{pi} \cdot c_{\mf{z}_i} \otimes B_i^*h_{pi}  \right)
\end{equation*}
for some $n \in \mb{N}$ and some collections $\left\{ h_{pi} \mid 1 \leq i \leq k,  \ 1 \leq p \leq n \right\} \subseteq  H^{(s)}$ and $\left\{ A_{pi} \mid 1 \leq i \leq k,  \ 1 \leq p \leq n \right\} \subseteq \sigma(M)'$.  
Similarly, define $J_F$ to be the norm-closure of the subspace of $\F{E^\sigma} \otimes_{\iota^{(t)}} H^{(t)}$ consisting of the elements of the form $\sum_{p = 1}^n \left( \sum_{i = 1}^k A_{pi} \cdot c_{\mf{z}_i} \otimes F_i^*h_{pi}  \right)$.
As we are assuming condition  \ref{WeightedNevanlinnaPick_i1}, it follows that there is a well-defined, contractive linear map $R: J_B \to J_F$ such that for any $1 \leq i \leq k$, $A \in \sigma(M)'$, and $h \in H^{(s)}$,
\begin{equation*}
R \left( A \cdot c_{\mf{z}_i} \otimes B_i^*h\right) = A \cdot c_{\mf{z}_i} \otimes F_i^*h.
\end{equation*}
We aim to apply Corollary \ref{GenWeightedCommutantLifting}, corollary to the weighted commutant lifting theorem, with $\sigma(M)'$, $E^\sigma$, $X'$, and $Z'$ replacing $M$, $E$, $X$, and $Z$.  
In the statement of that corollary, we take $\sigma_1 = \iota^{(t)}$ on $H^{(t)}$, $\sigma_2 = \iota^{(s)}$ on $H^{(s)}$, $J_1 =   
J_F$ in $\F{E^\sigma} \otimes_{\iota^{(t)}} H^{(t)}$, and $J_2 = J_B$ in $\F{E^\sigma} \otimes_{\iota^{(s)}} H^{(s)}$.  
Let $V_F$ and $V_B$ denote the inclusion maps of $J_F$ into $\F{E^\sigma} \otimes_{\iota^{(t)}} H^{(t)}$ and $J_B$ into $\F{E^\sigma} \otimes_{\iota^{(s)}} H^{(s)}$, respectively.
Let $G = R^*$ in $\ms{B}(J_F, J_B)$.
Let us show that for any $S \in H^\infty(E^\sigma, Z')$,
\begin{enumerate}[label=(\roman*),ref=(\roman*)]
\item \label{applyCor_i1} $\left( S^* \otimes I_t \right)(J_F) \subseteq J_F$,
\item \label{applyCor_i2} $\left(S^* \otimes I_s\right)(J_B) \subseteq J_B$, and
\item \label{applyCor_i3} $G \left(V_F^* \left(S \otimes I_t\right) V_F\right) = \left(V_B^* \left(S \otimes I_s\right) V_B\right) G$.
\end{enumerate}
It suffices to show that each of these properties holds when $S = \varphi'_\infty(A)$ for $A \in \sigma(M)'$ and when $S = W'_\xi $ for $\xi \in E^\sigma$.
When $D \in \sigma(M)'$, $1 \leq i \leq k$, and $h \in H^{(s)}$, we have
$\left(\varphi'_\infty(A)^* \otimes I_t \right)\left(D \cdot c_{\mf{z}_i} \otimes F_i^* h \right)  = A^*D \cdot c_{\mf{z}_i} \otimes F_i^* h$.
Moreover, by Lemma \ref{iotaWstar},
$\left( (W'_\xi )^* \otimes I_t \right)\left(D \cdot c_{\mf{z}_i} \otimes F_i^* h \right)   = \langle D^* \cdot \xi, \mf{z}_i^* \rangle \cdot c_{\mf{z}_i} \otimes F_i^* h$.
Property \ref{applyCor_i1} now follows from the definition of $J_F$;  property \ref{applyCor_i2} holds by similar reasoning.
To obtain property \ref{applyCor_i3}, we observe that if $D \in \sigma(M)'$, $1 \leq i \leq k$, and $h \in H^{(s)}$, 
\begin{multline*}
( V_B^* ( \varphi'_\infty(A) \otimes I_s) V_B  G )^* ( D \cdot c_{\mf{z}_i} \otimes B_i^* h) \\
= R V_B^* (\varphi'_\infty(A^*) \otimes I_s) V_B ( D \cdot c_{\mf{z}_i} \otimes B_i^* h) \\
=  (A^*D) \cdot c_{\mf{z}_i} \otimes F_i^* h
= V_F^* (\varphi'_\infty(A^*) \otimes I_t) V_FR \Big( D \cdot c_{\mf{z}_i} \otimes B_i^* h\Big) \\
= (G  V_F^* (\varphi'_\infty(A) \otimes I_t) V_F)^* ( D \cdot c_{\mf{z}_i} \otimes B_i^* h )
\end{multline*}
Thus $(G  V_F^* (\varphi'_\infty(A) \otimes I_t) V_F)^* = ( V_B^* (\varphi'_\infty(A) \otimes I_s) V_B G )^*$, and by taking adjoints, we obtain property \ref{applyCor_i3} for $S = \varphi'_\infty(A)$.  
Finally, by Lemma \ref{iotaWstar}, 
\begin{multline*}
 \left( V_B^* \left(W'_\xi  \otimes I_s \right) V_B G \right)^* \left( D \cdot c_{\mf{z}_i} \otimes B_i^* h \right) \\
 = RV_B^* \left((W'_\xi)^*  \otimes I_s \right) V_B \left( D \cdot c_{\mf{z}_i} \otimes B_i^* h \right) 
 = \langle D^* \cdot \xi, \mf{z}_i^* \rangle \cdot c_{\mf{z}_i} \otimes F_i^* h \\
 = V_F^* \left((W'_\xi)^* \otimes I_t \right) V_F R \left( D \cdot c_{\mf{z}_i} \otimes B_i^* h \right) \\
= \left( G  V_F^* \left(W'_\xi \otimes I_t \right) V_F \right)^* \left( D \cdot c_{\mf{z}_i} \otimes B_i^* h \right)
\end{multline*}
Therefore $\left(G V_F^* \left(W'_\xi  \otimes I_t \right) V_F  \right)^* = \left(  V_B^* \left(W'_\xi  \otimes I_s \right) V_B  G \right)^*$, and by taking adjoints, we obtain property \ref{applyCor_i3} for $S = W'_\xi$, as desired.

Having satisfied the hypothesis of Corollary \ref{GenWeightedCommutantLifting}, we conclude that there exists an operator $\widetilde{G} \in \ms{B}\left(\F{E^\sigma} \otimes_{\iota^{(t)}} H^{(t)}, \F{E^\sigma} \otimes_{\iota^{(s)}} H^{(s)} \right)$ such that
\begin{enumerate}[label=(\Roman*),ref=(\Roman*)]
\item \label{applyCor2_i1} $\widetilde{G}^* (J_B) \subseteq J_F$,
\item \label{applyCor2_i2} $V_B^*  \widetilde{G}   V_F = R^*$,
\item \label{applyCor2_i3} $ \widetilde{G}   \left(S \otimes I_t \right) =  \left( S \otimes I_s \right)  \widetilde{G}$ for all $S \in H^\infty(E^\sigma, Z')$, and
\item \label{applyCor2_i4} $\Vert \widetilde{G} \Vert = \Vert R \Vert$.
\end{enumerate}
Property \ref{applyCor2_i3} and Corollary \ref{double_comm_corollary} together imply the existence of some $Y =  [ Y_{af} ]$ in $M_{s \times t}\left(H^\infty \left(E, Z \right) \right) $ such that $\widetilde{G}  = \pi^\sigma_{s \times t} (Y)$.  
Since $\pi^\sigma_{s \times t}$ is an isometry, property \ref{applyCor2_i4} implies that 
$\left\Vert Y  \right\Vert = \Vert \pi^\sigma_{s \times t}(Y) \Vert = \Vert \widetilde{G} \Vert  = \Vert R \Vert \leq 1$.
To complete the proof, let $P_F \in \ms{B}(\F{E^\sigma} \otimes_{\iota^{(t)}} H^{(t)})$ denote the projection map onto $J_F$, and let $P_B \in \ms{B}(\F{E^\sigma} \otimes_{\iota^{(s)}} H^{(s)})$ denote the projection map onto $J_B$.  
Then $V_F V_F^* = P_F$ and $V_B V_B^* = P_B$.  
For $h \in H^{(s)}$, 
$L_i^sB_i^* h = c_{\mf{z}_i} \otimes B_i^* h $
and 
$L_i^tF_i^* h = c_{\mf{z}_i} \otimes F_i^* h $. 
It follows that 
$L_i^sB_i^* = P_B L_i^sB_i^*$ and 
$V_F R V_B^*L_i^sB_i^* = L_i^tF_i^*$.  
Thus by properties  \ref{applyCor2_i1} and \ref{applyCor2_i2},
$\widetilde{G}^* L_i^sB_i^*
= \widetilde{G}^* P_B L_i^sB_i^*
= (V_FV_F^*)\widetilde{G}^* (V_BV_B^*) L_i^sB_i^*
= V_F R V_B^*L_i^sB_i^*
=  L_i^tF_i^*$.
Taking adjoints, we obtain
\begin{equation}
\label{WeightedNevanlinnaPick_e3} 
B_i L_i^{s*} \widetilde{G}= F_i L_i^{t*}.
\end{equation}
By Lemma \ref{rhofacts2}\ref{rhofacts2_i2},
\begin{equation}
\label{WeightedNevanlinnaPick_e9} 
 \left[ \widehat{Y_{af}}(\mf{z}_i)  \right]
= \left[ L_i^* \pi^\sigma \left(Y_{af} \right) L_{I} \right]
= L_i^{s*}   \widetilde{G}   L_I^t.
\end{equation}
Since $c_{\mf{z}_i}(0) = I_H$, we have $L_i^{t*} L_I^t = I_t$.
Thus, by equations \eqref{WeightedNevanlinnaPick_e9} and  \eqref{WeightedNevanlinnaPick_e3},
\begin{equation*}
B_i  \left[ \widehat{Y_{af}}(\mf{z}_i)  \right]
= B_i L_i^{s*}   \widetilde{G}   L_I^t 
= F_i  L_i^{t*}  L_I^t 
= F_i.
\end{equation*}
Thus condition \ref{WeightedNevanlinnaPick_i1} implies condition \ref{WeightedNevanlinnaPick_i4}, which completes the proof.
\end{proof}

\begin{remark}
In  the scalar case, when $M = E = H = \mb{C}$, our kernel $\ms{K}$ is simply the reproducing kernel to a weighted Hardy space.  
Because one of the hypothesis of the admissible sequence $X$ is that each $X_k$ is positive, it follows from Theorem 7.33 of \cite{Agler2002} that the weighted Hardy spaces under consideration are those that satisfy the \emph{complete Pick Property} that is described in Definition 5.13 in \cite{Agler2002}.  
Thus, for example, our theorem applies to the Hardy and Dirichlet spaces, but not the Bergman space, by Corollary 7.37, Corollary 7.41, and Example 5.17 of \cite{Agler2002}.
The original complete Pick property involves certain matrix-valued multipliers, and while we will not include the discussion in the present paper, notions such as reproducing kernel Hilbert spaces and their spaces of multipliers can be generalized to the $W^*$-setting, which produces interesting examples of noncommutative functions.
A promising topic of future work involves the formulation of a $W^*$-version of the complete Pick property and an investigation of its implications for the representation theory of $H(E,Z)$, as was begun in \cite{Muhly2016}.

\end{remark}

\section*{Acknowledgements}
Thank you to Paul Muhly and Baruch Solel for their support, encouragement, and helpful feedback in this endeavor.

\bibliographystyle{plain}
\bibliography{C:/Users/Jenni/Documents/MuhlyMaster/MuhlyMaster}

\end{document}